\documentclass[10pt]{article}

\usepackage[T1]{fontenc} 
\usepackage[utf8]{inputenc}
\usepackage[english]{babel}
\usepackage{amsmath,amsfonts,amssymb, amsthm}
\usepackage{mathrsfs}
\usepackage{enumitem}
\usepackage{wrapfig}
\usepackage{graphicx}
\usepackage{bbold}
\usepackage{bbm}
\usepackage{dsfont}
\usepackage{hyperref}
\usepackage{csquotes}
\usepackage{relsize}
\usepackage{exscale}
\usepackage{float}
\usepackage{empheq}
\usepackage{setspace}
\usepackage[toc,page]{appendix}

\usepackage{subfig}
\usepackage{titlesec}
\usepackage{pdfpages}
\usepackage{mdframed}
\usepackage{listings}

\usepackage[
  paper=letterpaper,
  textwidth=6.25in,
  textheight=9.0in,
  left=1.2in,
  right=1.2in,
  top=1.1in,
  bottom=1.1in
]{geometry}

\allowdisplaybreaks

\titleformat{\paragraph}
{\normalfont\normalsize\bfseries}{\theparagraph}{1em}{}
\titlespacing*{\paragraph}
{0pt}{3.25ex plus 1ex minus .2ex}{1.5ex plus .2ex}

\newcommand{\eps}{\varepsilon}

\newcommand{\mB}{\mathcal{B}}

\newcommand{\mD}{\mathcal{D}}
\newcommand{\mE}{\mathcal{E}}

\newcommand{\mH}{\mathcal{H}}
\newcommand{\mI}{\mathcal{I}}

\newcommand{\mM}{\mathcal{M}}
\newcommand{\mN}{\mathcal{N}}

\newcommand{\mP}{\mathcal{P}}

\newcommand{\mR}{\mathcal{R}}
\newcommand{\mS}{\mathcal{S}}
\newcommand{\mT}{\mathcal{T}}
\newcommand{\mU}{\mathcal{U}}

\newcommand{\mW}{\mathcal{W}}

\newcommand{\NN}{\mathbb{N}}

\newcommand{\PP}{\mathbb{P}}

\newcommand{\RR}{\mathbb{R}}

\newcommand{\A}{\operatorname{A}}
\newcommand{\B}{\operatorname{B}}

\newcommand{\E}{\operatorname{E}}

\renewcommand{\P}{\operatorname{P}}
\newcommand{\Q}{\operatorname{Q}}

\newcommand{\U}{\operatorname{U}}

\newcommand{\X}{\operatorname{X}}
\newcommand{\Y}{\operatorname{Y}}

\renewcommand{\k}{\kappa}

\newcommand{\wh}{\widehat}
\newcommand{\ol}{\overline}

\newcommand{\Par}{\operatorname{Par}}
\newcommand{\Leb}{\operatorname{Leb}}

\newcommand{\supp}{\operatorname{supp}}

\newcommand{\card}{\operatorname{Card}}

\newcommand{\Var}{\operatorname{Var}}

\newcommand{\KL}{\operatorname{KL}}

\newcommand{\sY}{\mathrm{Y}}
\newcommand{\sX}{\mathrm{X}}

\newcommand{\ul}{\underline}

\SetEnumerateShortLabel{i}{\textit{\roman}}

\newtheorem{theorem}{Theorem}
\newtheorem{lemma}[theorem]{Lemma}
\newtheorem{proposition}[theorem]{Proposition}
\newtheorem{corollary}[theorem]{Corollary}
\newtheorem{definition}[theorem]{Definition}

\newtheorem{example}[theorem]{Example}

\newtheorem{remark}[theorem]{Remark}

\makeatletter
\@addtoreset{proofpart}{theorem}
\makeatother

\usepackage[backend=biber, maxnames=99]{biblatex}
\title{A Minimax Theory of Nonparametric \\Regression Under Covariate Shift}
\author{\large{Petr Zamolodtchikov}\vspace{.15cm}\\ Universit\"at Bielefeld\vspace{.05cm}\\
pzamolod@math.uni-bielefeld.de}
\date{\today}

\begin{document}

\maketitle

\begin{abstract}
    We consider nonparametric regression under covariate shift, where we observe samples from both the target distribution and a related but distinct source distribution. We introduce a novel object, the transfer function, and show that properties of its domain determine our minimax rates. Those exhibit a variety of regimes, including classical rates, governed by the better of source-only and target-only rates, as well as regimes in which the convergence rates exhibit multiplicative interactions between the sample sizes and are faster than the best-of-two benchmark. The rates are shown to be achieved up to logarithmic factors by a design-adaptive estimator. Compared with existing theory, our results cover the case in which covariates have unbounded support.
\end{abstract}

\begingroup
\renewcommand{\thefootnote}{}
\footnotetext{\textit{keywords:} Covariate Shift, Minimax, Nonparametric Regression, Transfer Learning, Unbounded Support.}
\endgroup

\section{Introduction}
\label{sec.intro}

Transfer learning is a widely used approach to improving the performance of a predictive model by incorporating related data into its training set. The central premise is that different datasets may contain information about the task at hand, which the model extracts and leverages, leading to numerous successful implementations over the years. Examples include computer vision \cite{li2020transfer,wang2022transfer}, healthcare \cite{ali2021enhanced,ebbehoj2022transfer}, natural language processing \cite{ruder2019transfer}, and experimental particle physics \cite{mokhtar2025fine, camaiani2022model}. Beyond the natural incentive to improve model performance, transfer learning is particularly interesting when relevant data are scarce.\\

In the setting of supervised learning, transfer learning aims to design algorithms that perform well on a \textit{target} distribution $\Q\mathstrut_{\!\X, \sY}$ while being fit on a dataset either comprised of $n$ samples from a \textit{source} distribution $\P_{\sX, \sY},$ or a mix of both $n$ samples from the source and $m$ samples from the target. Most of classical statistical learning theory had been derived under the assumption of distributional invariance between training and test datasets, that is $\P_{\sX, \sY} = \Q\mathstrut_{\!\sX, \sY},$ and, therefore, lacks the ability to explain, predict, and control transfer learning in practice. Hence, transfer learning is an active area of research \cite{zhu2025}. The theoretical literature has largely crystallised around two mathematically distinct, yet complementary, paradigms arising from distributional shifts. The first is called \textit{posterior drift} and refers to the setting in which the covariates' marginals remain the same, $\P_{\sX} = \Q\mathstrut_{\!\sX},$ but the conditional distributions of the outputs vary across tasks, $\P_{\sY|\sX} \neq \Q\mathstrut_{\!\Y\!|\sX}.$ See \cite{auddy2025minimax, cai2024transfer} and references therein. This paper addresses \textit{covariate shift} (CS) \cite{shimodaira2000improving, sugiyama2007covariate}, which is the second setting where only the covariates' marginals vary, $\P_{\sX} \neq \Q\mathstrut_{\!\sX},$ while the output's conditionals remain equal, $\P_{\sY|\sX} = \Q\mathstrut_{\!\Y\!|\sX}.$ \\

The main theoretical goal of analysing CS is two-fold. Firstly, one aims to identify the correct object that describes and quantifies transferability by modulating convergence rates. Secondly, one aims to design algorithms that optimally adapt to distributional shifts. In this article, transferability is quantified by the notion of \textit{transfer function}, a $(\P_{\sX}, \Q\mathstrut_{\!\sX})$-dependent univariate function whose analytic behaviour encodes transferability. More precisely, the latter diverges at the boundary of its domain, and the location of this boundary governs achievable minimax rates. The second goal is achieved, up to logarithmic factors, by a design-adaptive $k$-nearest-neighbours regression estimator.\\

Quantification of transfer is a central topic in the CS literature. In the linear model, transfer is naturally described by the interactions of the source and target covariance matrices \cite{mallinar2024minimum, song2024generalization}. A dominant approach for empirical risk minimisers (ERM) addresses CS by re-weighting the empirical risk, thereby encoding transferability via moment assumptions about the density ratio \cite{ma2023optimally, gogolashvili2023importance, chen2024high, xu2025estimating, liu2025spectral, della2025computational, feng2024deep}.\\

The line of work we now extend has departed from density-ratio-type assumptions and focused on geometric-type regularity of the source-target pair. Notably, in the classification setting, the pioneering work \cite{kpotufe2021marginal} introduced the \textit{transfer exponent}, which measures the relative singularity of the source-target pair via the growth of their ball-mass ratios; see also \cite{cai2024transferbandit}. In the same spirit, and in the context of nonparametric regression, \textit{$\alpha$-families} of source-target pairs were introduced in \cite{pathak2022new} and \cite{trottner2024covariate}.\\

In the nonparametric setting, typical convergence rates have the form $(n^{r_S} + m^{r_T})^{-1} \asymp (n^{r_S} \vee m^{r_T})^{-1},$ which we call the \textit{wedge rate} and identify as the minimum of a \textit{source rate} and a \textit{target rate}. The wedge rate almost always corresponds to the convergence rate of an estimator that selects the best between two separate estimators trained solely on $n$ source samples and $m$ target samples, respectively.\\

In the context of $\alpha$-families, Nadaraya-Watson estimators with fixed bandwidths are minimax optimal. However, in \cite{schmidt2022local}, it was shown that the unpenalised least squares estimator exhibits much finer transfer properties, comparable to those of a Nadaraya-Watson estimator with pointwise-optimal, or design-adaptive, bandwidth. This led to faster rates for power-to-uniform transfer, where $\P_{\sX}$ is a power law and $\Q\mathstrut_{\!\sX}$ is uniform, and motivated the study of design-adaptive estimators in \cite{zamolodtchikov2024transfer}, where a local $k$-NN was first analysed. Additionally, \cite{schmidt2022local} identifies a prototypical instance of multiplicative rates. The recent preprint \cite{zhou2025synergistic} establishes a lower bound on the integrated convergence rates of power-to-uniform covariate shift under general smoothness, which provides finer insights into the multiplicative regime.\\

These observations reveal gaps in existing minimax theories pertaining to CS in the nonparametric regression setting. Firstly, the class of source-target pairs in \cite{pathak2022new} incorporates very irregular distributions. Those constitute the hard cases and lead to a theory that smooths away finer mechanisms, such as multiplicative rates. Secondly, transfer exponents and $\alpha$-families break down as soon as the design becomes unbounded. Indeed, both quantities then become infinite. \\

The approach via transfer functions leads to a simple yet expressive minimax theory of nonparametric regression under CS. The result is a single scalar parameter per occurring distribution that governs minimax rates. The analysis is carried out over a restricted class of source-target pairs whose local mass is sufficiently well-behaved, enabling our rates to capture the subtle mechanics of transfer. In particular, our approach yields a rich atlas of clearly identified regimes. As discussed in Section \ref{sec.4.3}, the restriction isolates a clean phenomenon that is structurally robust, and its principled weakening should only broaden the theory. The hard cases are very regular Pareto source-target pairs, a feature of a realistic minimax theory. The transfer function remains meaningful independently of the boundedness, or lack thereof, of the covariate support. Since all the rates are new, our theory also contributes to nonparametric regression with unbounded covariate support.\\

\noindent \textbf{Paper organisation.} We introduce the statistical model, the transfer function and our working assumptions in Section \ref{sec.2}. We then state and discuss our main results in Section \ref{sec.3}. Section \ref{sec.4} is dedicated to discussing our notion of transfer function, the design of our estimator, and the necessity or lack thereof of our working assumptions. All the proofs are deferred to the appendix.\\

\noindent\textbf{Notations.}
We use the notation $a_n \lesssim b_n$ to denote that there exists a positive constant $c$ such that $a_n \leq c b_n$, for all $n \in \NN$. $a_n \gtrsim b_n$ is defined in a similar way, and we denote $a_n \asymp b_n$ whenever $a_n \lesssim b_n$ and $a_n \gtrsim b_n.$
For any $x \in \RR^d$ and $r > 0$, we denote by $\B(x,r)$ the closed ball centred in $x$ with radius $r$ in the Euclidean norm denoted by $\|\cdot\|.$ The infinite norm of a function $f$ is denoted by $\|f\|_{\infty}.$ 

\section{Preliminaries}\label{sec.2}

\noindent\textbf{Nonparametric regression model.} We consider the standard nonparametric regression model. Letting $(X, Y)$ be a pair of random variables with values in $\RR^d \times \RR,$ we assume the existence of a function $f_* \colon \RR^d \to \RR$ satisfying
\begin{align}
    \label{eq.model}
    Y = f_*(X) + \eps,
\end{align}
where $\eps$ is a centred random variable, independent from $X.$ We additionally assume that $\eps$ is sub-exponential with parameters $\alpha, \nu >0.$ That is, for all $|\lambda| \leq \alpha^{-1}$
\begin{align*}
    \E[\exp(\lambda \eps)] \leq \exp\Big(\frac{\nu^2\lambda^2}{2}\Big). 
\end{align*}
Moreover, we will assume that $f_*$ lies in a H\"older ball $\mH(L, \beta)$ defined as 
\begin{align*}
    \mH(L, \beta) := \bigg\{f \colon \RR^d \to \RR: \|f\|_{\infty} + \sup_{\substack{x, y \in \RR^d\\ x\neq y}} \frac{|f(x) - f(y)|}{\|x - y\|^\beta} \leq L\bigg\},
\end{align*}
where $L > 0$ and $\beta \in (0, 1]$ are fixed constants.\\

\noindent\textbf{Covariate shift.} In the setting of Covariate Shift (CS), we are given a sample of $m$ observations from the target distribution $\Q\mathstrut_{\!\X, \sY}$ and $n$ observations from the source distribution $\P_{\sX, \sY}$ with the constraint that they agree when conditioned on the covariates,
\begin{align}
\label{eq.covariate.shift}
\P_{\sY|\sX} = \Q\mathstrut_{\!\Y\!|\sX}. \tag{CS}
\end{align}
Under \eqref{eq.covariate.shift}, if $(X, Y)$ satisfies Model \eqref{eq.model} under the source (resp.\ the target) distribution, then the exact same model is satisfied under the target (resp.\ the source) distribution. This means that the regression function $f_*$ and the noise variable's distribution are the same under $\P_{\sX, \sY}$ and $\Q\mathstrut_{\!\X, \sY}.$ The mismatch takes effect on the level of marginal distributions $\P_{\sX}$ and $\Q\mathstrut_{\!\sX}.$\\

\noindent\textbf{Prediction loss.} Given a sample $\mD_{\P} = \{(X_i, Y_i)\}_{i=1}^n$ of i.i.d.\ observations from $\P_{\sX, \sY}$ and an additional independent sample $\mD_{\Q} = \{(X_j', Y_j')\}_{j=1}^m$ of i.i.d.\ observations from $\Q\mathstrut_{\!\X, \sY},$ the predictive goal is to construct an estimator $\smash{\wh f}(\cdot, \mD_{\Q}, \mD_{\Q})$ that minimises the Mean Squared Error $\ell(\smash{\wh f}\,) := \E_{\Q\mathstrut_{\!\X, \sY}}[(\smash{\wh f}(X') - Y')^2|\mD_{\Q}, \mD_{\Q}].$ Under Model \eqref{eq.model}, this is equivalent to estimating $f_*$ in the $L^2(\Q\mathstrut_{\!\sX})$ loss 
\begin{align*}
    \ell(\wh f\,) - \ell(f_*) = \E_{X' \sim \Q\mathstrut_{\!\sX}}\big[(\wh f(X') - f_*(X'))^2\big] = \big\|\wh f - f_*\big\|_{L^2(\Q\mathstrut_{\!\sX})}^2.
\end{align*} 
Under Model \eqref{eq.model}, the distribution of $Y$ is entirely determined by the distribution of the covariate $X,$ the regression function $f_*,$ and the distribution of $\eps,$ meaning that specifying joint distributions of $(X, Y)$ is redundant. Consequently, we will call $\P_{\sX}$ the source distribution and $\Q\mathstrut_{\!\sX}$ the target distribution in the rest of this article.\\

\noindent\textbf{Transfer function.} Within \eqref{eq.covariate.shift} and Model \eqref{eq.model}, it is clear that minimax estimation rates depend on the source-target pair $(\P_{\sX}, \Q\mathstrut_{\!\sX}),$ and, as mentioned in Section \ref{sec.intro}, a key step is to identify and formalise a quantity that captures this dependence. Here, this dependence is quantified through the behaviour of the \textit{transfer function} and its domain. Let us first introduce $\mM$ as the class of Borel probability measures on $\RR^d$ that admit a Lebesgue density.
\begin{definition}[Transfer function]
    For a pair of distributions $\P, \Q \in \mM,$ and a real number $\gamma \in [0, \infty),$ the transfer function is defined as
    \begin{align*}
        \mT(\P, \Q, \gamma) := \E_{X \sim \Q}[p(X)^{-\gamma}] \in [0, \infty],
    \end{align*}
    where $p$ and $q$ are the respective densities of $\P$ and $\Q.$
\end{definition}
Intuitively, $\mT(\P, \Q, \cdot)$ measures the quantity of mass that $\Q$ assigns to low-density regions of $\P.$ For fixed $\P, \Q,$ the transfer function is a $\log$-convex function of $\gamma$ that typically blows up to infinity once its argument reaches the boundary of its domain. This leads to the definition of integrability indices.
\begin{definition}[Integrability index]
    For $\P, \Q \in \mM,$ we define the integrability index of $(\P, \Q)$ as
    \begin{align*}
        \gamma^*(\P, \Q) := \sup\{\gamma \geq 0: \mT(\P, \Q, \gamma) < \infty\} \in [0, \infty].
    \end{align*}
\end{definition}
Further properties of the transfer function and integrability indices are discussed in Section \ref{sec.4.1}.\\

\noindent\textbf{Regularity.} The derivations in our results rely on additional regularity assumptions that restrict the class of covariates' distributions in the scope of this paper. First, denote by $\mM(D)$ the class of measures in $\mM$ whose Lebesgue density is upper bounded by the constant $D>0,$ which is arbitrary but fixed.
\begin{definition}
\label{def.class.proba}
Let $0 < \theta < \infty.$ We define $\mP(D, \theta)$ to be the class of distributions $\P \in \mM(D),$ with density $p,$ such that for all $x \in \supp(\P)$ and all $r \in (0, 1],$
\begin{align}
    \label{eq.minimal.maximal.mass}
    \theta^{-1}p(x)r^d \leq \P\{\B(x, r)\} \leq \theta p(x)r^d.
\end{align}
\end{definition}
We will refer to \eqref{eq.minimal.maximal.mass} as the \textit{local mass} assumption. Definition \ref{def.class.proba} is satisfied for uniform distributions, or distributions with densities bounded away from zero and infinity, which is one of the most standard settings for nonparametric regression. Standard light-tailed distributions such as Gaussians do not satisfy \eqref{eq.minimal.maximal.mass}, while heavier-tailed distributions such as Pareto and exponential do (see Examples \ref{ex.pareto} and \ref{ex.exponential}). If the density $p$ of $\P$ cancels at $x_0 \in \RR^d,$ and $p(x)$ is proportional to $\|x - x_0\|^\alpha$ in a neighbourhood of $x_0$ and some positive $\alpha,$ then $\P$ does not satisfy \eqref{eq.minimal.maximal.mass}. We further discuss Definition \ref{def.class.proba} and its impact on our results in Section \ref{sec.4.3}

\section{Main Results} \label{sec.3}
In this section, we state our upper and lower bounds before giving an analysis of the resulting rate exponents. 
We work under the setting of Model \eqref{eq.model} under \eqref{eq.covariate.shift}. We fix $D, \theta > 0.$ We will always denote the source distribution by $\P_{\sX}$ and the target distribution by $\Q\mathstrut_{\!\sX}.$ As such, we also simplify notations and denote by $\gamma^* := \gamma^*(\P_{\sX}, \Q\mathstrut_{\!\sX})$ and by $s^* := \gamma^*(\Q\mathstrut_{\!\sX}, \Q\mathstrut_{\!\sX}).$ For readability, we also denote by $r_\beta := \tfrac{2\beta}{2\beta + d}.$

\subsection{Learnability and impact of covariate shift}
\label{sec.3.1}

In what follows, we abstract away the specific expressions of the multiplicative constants in the upper and lower bounds, retaining only their essential dependencies on the parameters. Explicit expressions can be found in the more precise statements in Appendix \ref{app.main.results}, see Theorems \ref{th.transfer.local.knn.proof}, and \ref{th.rates.local.k.two.sample.proof}. We now state our main results. On the one hand, we provide an upper bound for the supremum risk in probability over the class $\mH(L, \beta),$ which is valid pointwise for any source-target pair. On the other hand, we provide a lower bound on the supremum risk in expectation over $\mH(L, \beta)$ as well as classes of bounded transferability indices.
\begin{theorem}[Upper bound]
    \label{th.upper.bound.two.sample}
    Let $\P_{\sX}, \Q\mathstrut_{\!\sX} \in \mP(D, \theta).$ Then, there exists an estimator $\wh f,$ an integer $N = N(\theta, D)$ and a constant $C(\tau) = C(\tau, d, \beta, \theta, \gamma^*, s^*)$, such that for all $\tau > 1,$ all $m\wedge n \geq N,$ all $\gamma \in [0, \gamma^*)$ and all $s \in [0, s^*),$
    \begin{align*}
        \sup_{f_* \in \mH(L, \beta)}\PP\Bigg\{\big\|\wh f - f_*\big\|_{L^2(\Q\mathstrut_{\!\sX})}^2 \geq C(\tau)\mR(n, m, \gamma, s)\Bigg\} \leq 3(n^{1-\tau} + m^{1-\tau}),
    \end{align*}
    where the rate $\mR = \mR(n, m, \gamma, s)$ is defined as
    \begin{align*}
        \mR &= \begin{dcases}
            \mT_{\Q\mathstrut_{\!\sX}}(\P_{\sX}, \gamma)^{\tfrac{r_\beta - s}{\gamma - s}}\mT_{\Q\mathstrut_{\!\sX}}(\Q\mathstrut_{\!\sX}, s)^{\tfrac{\gamma - r_\beta}{\gamma - s}}\Big(\frac{\log(nm)}{n}\Big)^{\gamma \tfrac{r_\beta - s}{\gamma - s}}\Big(\frac{\log(nm)}{m}\Big)^{s\tfrac{\gamma - r_\beta}{\gamma - s}} &\!\!\! \vcenter{\hbox{\shortstack[l]{if $(\gamma-r_\beta)(s-r_\beta)<0$\\and $m\in[n,n^{\gamma/s}]$,}}}\\
            \Big[\mT_{\Q\mathstrut_{\!\sX}}(\P_{\sX}, \gamma \wedge r_\beta)\Big(\frac{\log(nm)}{n}\Big)^{\gamma \wedge r_\beta}\Big]\wedge \Big[\mT_{\Q\mathstrut_{\!\sX}}(\Q\mathstrut_{\!\sX}, s \wedge r_\beta)\Big(\frac{\log(nm)}{m}\Big)^{s\wedge r_\beta}\Big] &\!\!\!\text{ otherwise.}
        \end{dcases}
    \end{align*}
\end{theorem}
The pointwise nature of the result is desirable here: conditionally on the regularity imposed by $\mP(D, \theta),$ the rate depends on the particular source-target pair only through the transfer function. It is important to note that the estimator in Theorem \ref{th.upper.bound.two.sample} does not depend on the particular choice of source-target pair, nor on any knowledge of the transferability indices $(\gamma^*, s^*).$ It depends, however, on the knowledge of $2\beta/(2\beta + d).$ In fact, we design a \textit{local} $k$-nearest neighbours regressor and discuss its construction in Section \ref{sec.4.2}. Theorem \ref{th.upper.bound.two.sample} identifies conditions on $(\gamma, s, r_\beta)$ for which a special regime \textit{might} emerge, where the rate becomes multiplicative. This situation occurs when $(\gamma - r_\beta)(s - r_\beta) < 0$. In this setting, the rates become multiplicative if, in addition, $m \in [n, n^{\gamma/s}].$ The latter interval is understood as $[n^{\gamma/s}, n]$ if $\gamma/s < 1$ and $[n, n^{\gamma/s}]$ when $\gamma/s > 1.$ As a result, the convergence rates depend on five parameters $(\gamma^*, s^*, r_\beta, n, m).$\\

It is noteworthy that the double role of $\Q\mathstrut_{\!\sX}$ as both (part of) the data-generating process, and the testing distribution in Theorem \ref{th.upper.bound.two.sample} is arbitrary. A careful reading of the proofs indicates that the same result holds if we instead consider that the data consists of $n$ samples from $\P_{\sX}^{(1)}$ and $m$ samples from $\P_{\sX}^{(2)}$ while keeping $\Q\mathstrut_{\!\sX}$ as the testing distribution. In this case, the relevant parameters in the rates' exponents are $\gamma^*(\P_{\sX}^{(1)}, \Q\mathstrut_{\!\sX})$ and $\gamma^*(\P_{\sX}^{(2)}, \Q\mathstrut_{\!\sX}).$ This observation suggests natural extensions to multitask estimation \cite{blanchard2024estimation} or domain generalisation \cite{blanchard2021domain}. \\

We now state a more specialised result that corresponds formally to the case $m = 0$ in Theorem \ref{th.upper.bound.two.sample}, in which the training data is sampled from the source distribution only.
\begin{corollary}[Upper bound, pure transfer]
    \label{cor.upper.bound.one.sample.P}
    Let $\P_{\sX}, \Q\mathstrut_{\!\sX} \in \mP(D, \theta).$ Then, there exists an estimator $\wh f,$ an integer $N = N(\theta, D) \geq 2d\vee 3,$ and a constant $C(\tau) = C(\tau, d, \beta, \theta, \gamma^*)$, such that for all $\tau > 1,$ all $n \geq N,$ and all $\gamma \in [0, \gamma^*),$
    \begin{align*}
        \sup_{f_* \in \mH(L, \beta)}\PP\Big\{\big\|\wh f - f_*\big\|_{L^2(\Q\mathstrut_{\!\sX})}^2 \geq C(\tau)\mT(\P_{\sX}, \Q\mathstrut_{\!\sX}, \gamma\wedge r_\beta)\Big(\frac{\log n}{n}\Big)^{\gamma \wedge r_\beta}\Big\} \leq 3n^{1-\tau}.
    \end{align*}
\end{corollary}
In view of Corollary \ref{cor.upper.bound.one.sample.P}, one notices that in configurations of $(\gamma^*, s^*, r_\beta, n, m)$ where the rate is not multiplicative, the rates in Theorem \ref{th.upper.bound.two.sample} are essentially the minimum of the rates obtained from selecting the best among two separate estimators, trained only on the source sample and only the target sample, respectively.\\

We complement the upper bounds from Theorem \ref{th.upper.bound.two.sample} with uniform lower bounds over classes of pairs with bounded transferability indices defined as follows
\begin{align*}
    \ul\mP(D, \theta, s, \gamma) := \Big\{(\P_{\sX}, \Q\mathstrut_{\!\sX}) \in \mP(D, \theta)^2: \gamma^*(\P_{\sX}, \Q\mathstrut_{\!\sX}) \geq \gamma \ \text{ and } \ \gamma^*(\Q\mathstrut_{\!\sX}, \Q\mathstrut_{\!\sX}) \geq s\Big\}.
\end{align*}
With this definition, we obtain the following lower bound on the minimax risk over classes $\ul \mP(D, \theta, s, \gamma).$
\begin{theorem}
    \label{th.lower.bound}
    Work under \eqref{eq.covariate.shift} and within the regression model \eqref{eq.model} with Gaussian noise $\eps \sim \mN(0, \sigma^2)$ with $\sigma^2 > 0.$ Let $L, D > 0, \ \theta \geq \Leb(\B(0, 1)), \beta \in (0, 1]$ and $(\gamma, s) \in (0, \infty)\times (0, 1).$ Then, there exists a constant $C_0 > 0$ independent of $n,m$ such that for all $n\vee m >1$, 
    \begin{align*}
        \inf_{\wh f}\sup_{\substack{f_* \in \mH(L, \beta)\\ \P_{\sX}, \Q\mathstrut_{\!\sX} \in \ul\mP(D, \theta, \gamma, s)}} \E\Big[\|\wh f - f_*\|_{L^2(\Q\mathstrut_{\!\sX})}^2\Big] \geq  \begin{dcases}
            C_0m^{- s\tfrac{\gamma - r_\beta}{\gamma - s}}n^{-\gamma\tfrac{r_\beta - s}{\gamma - s}} & \vcenter{\hbox{\shortstack[l]{if $(\gamma-r_\beta)(s-r_\beta)<0,$\\and $m\in[n,n^{\gamma/s}]$}}}\\
            C_0\Big[n^{-(\gamma\wedge r_\beta)} \wedge m^{-(s \wedge r_\beta)}\Big] & \text{ else. }
        \end{dcases}
    \end{align*}
    where the $\inf$ is taken over all estimators obtained from $n$ samples from $\P_{\sX, \sY}$ and $m$ samples from $\Q\mathstrut_{\!\X, \sY}.$
\end{theorem}
The same conclusion holds uniformly over classes of source-target pairs with fixed, rather than lower-bounded, transferability indices. The requirement $\theta \geq \Leb(\B(0, 1))$ can be omitted, as $\mP(D, \theta)$ is empty for $\theta < \Leb(\B(0, 1))$ as a consequence of Lebesgue's differentiation theorem. The lower bound is proved for $s\in (0,1)$ using Pareto source-target pairs. Borderline cases such as $\gamma \wedge s = 0$ or $s =1$ can be treated with alternative constructions for the target distribution at the cost of additional technicalities and are not pursued here. The cases $s > 1$ or $\gamma = \infty$ are out of the scope of this article as far as formal lower bounds constructions are concerned.\\

\subsection{Analysis of the rates}
\label{sec.3.2}

As shown in Theorems \ref{th.upper.bound.two.sample} and \ref{th.lower.bound}, the transferability indices govern the exponents in our convergence rates. Our statements of upper and lower bounds are asymmetric in the sense that the exponents in the upper bound are given for any $0 \leq \gamma < \gamma^*$ and  $0\leq s < s^*,$ while the exponents in the lower bounds are expressed in terms of transferability indices. Moreover, the upper bounds also incorporate the transfer function as a multiplicative constant. This is natural, as the existence of lower bounds shows that the multiplicative constants must blow up for exponents that are larger than $\gamma^*$ and $s^*$. The transfer function's appearance as a multiplicative constant captures this effect. Under certain conditions on the type of singularity at the transferability index, those can be absorbed in the rate, at the expense of a poly-logarithmic factor. While we discuss this more in detail in Section \ref{sec.4.1}, we now proceed to analyse the convergence rates obtained in Section \ref{sec.3.1}.\\

The goal of the following sections is to describe the structure of the aforementioned convergence rates, their transitions and the different regimes arising from our main results. This leads us to prioritise readability and clarity. As a first step in this direction, we introduce vocabulary to describe the different situations arising from Section \ref{sec.3.1}. \\

\noindent\textbf{Configurations.} We will refer to a configuration of $(\gamma, s, r_\beta)$ as \textit{supercritical} when $(\gamma - r_\beta)(s - r_\beta) < 0.$ We will call a configuration of the same parameters \textit{(sub)critical} whenever $(\gamma - r_\beta)(s - r_\beta) = 0$ (resp.\ $(\gamma - r_\beta)(s - r_\beta) > 0$).\\

\noindent\textbf{Regimes.} We will call \textit{wedge regime} any configuration of $(\gamma, s, r_\beta, n, m)$ in which the rates are proportional to the \textit{wedge rate} $n^{-(\gamma\wedge r_\beta)} \wedge m^{-(s\wedge r_\beta)}.$ Similarly, the \textit{acceleration regime} refers to any configuration of $(\gamma, s, r_\beta, n, m)$ in which the minimax rates are proportional to the \textit{accelerated rates} $n^{-\gamma\tfrac{r_\beta - s}{\gamma - s}}m^{-s\tfrac{\gamma - r_\beta}{\gamma - s}}.$\\

To further simplify exposition for the remaining Sections \ref{sec.3.2} and \ref{sec.3.3}, we omit all logarithmic factors and set the multiplicative constants to be equal to one. In particular, we omit the multiplicative occurrences of transfer functions. Finally, we will refer to the $n$-dependent part of the convergence rates as the \textit{source rate} and to the $m$-dependent part as the \textit{target rate}.

\subsubsection{Phase diagrams}
\label{sec.3.2.1}

\begin{figure}[H]
\centering
\begin{tabular}{@{}ccc@{}}
    \includegraphics[width=0.315\textwidth, trim=3pt 0cm 4pt 0cm, clip]{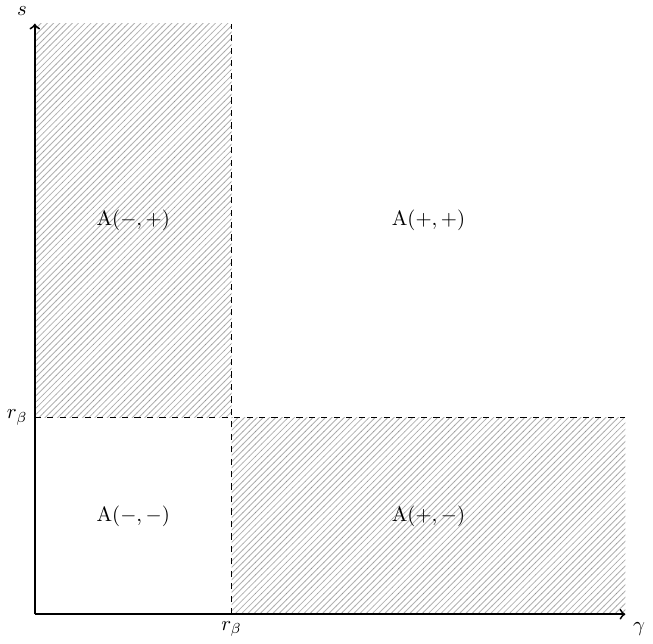} &
    \includegraphics[width=0.315\textwidth, trim=3pt 0cm 4pt 0cm, clip]{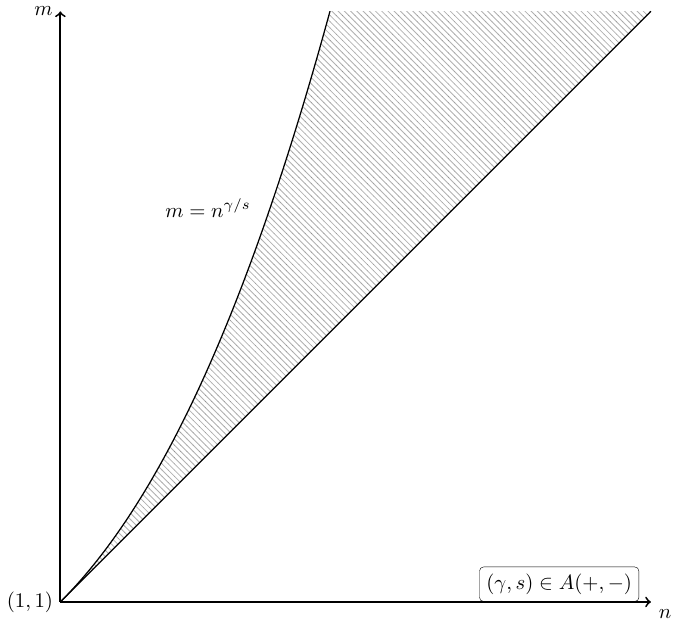} &
    \includegraphics[width=0.315\textwidth, trim=3pt 0cm 4pt 0cm, clip]{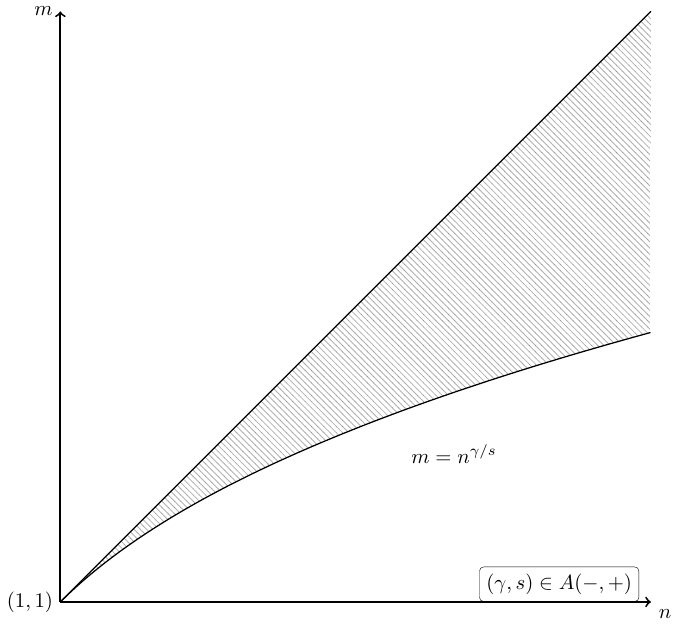}
\end{tabular}
\caption{\label{fig.rates.regimes} \textbf{Left:} Phase diagram in the $(\gamma, s)$ coordinates system. The dashed regions correspond to supercritical configurations. \textbf{Center:} Phase diagram in $(n, m)$ coordinates for $(\gamma, s) \in A(+, -).$ \textbf{Right:} Phase diagram in $(n,m)$ coordinates for $(\gamma, s) \in A(-, +).$ The dashed regions correspond to the acceleration regime.}
\label{fig.rates.regimes_accel}
\end{figure}
A global phase diagram is presented on the leftmost panel of Figure \ref{fig.rates.regimes}. The space of transferability indices is sharply partitioned into four regions. Region $\A(-, -)\cup\A(+, +)$ represent the set of subcritical configurations, where the accelerated regime cannot emerge, while $\A(-, +)\cup\A(+, -)$ the set of supercritical configurations where the accelerated regime might emerge, depending on $(n,m).$ The boundary (dotted lines) between $\A(-, -)\cup\A(+, +)$ and $\A(-, +)\cup\A(+, -)$ represents the set of critical configurations. In the region $\A(-, -)$ the rate is always $n^{-\gamma}\wedge m^{-s}$ and in $\A(+, +)$ the rate is always $n^{-r_\beta} \wedge m^{-r_\beta}.$ After fixing $(\gamma, s),$ it is possible to view the phase diagram in $(n, m)$ coordinates. This results in the central and right-most panels of Figure \ref{fig.rates.regimes_accel}. They describe the situation where $(\gamma, s)$ is fixed and supercritical. Since the two acceleration regimes are mutually exclusive, we obtain two diagrams. The central diagram corresponds to $(\gamma, s) \in \A(+, -),$ meaning that $\gamma/s > 1$ while the right-most diagram corresponds to $(\gamma, s) \in \A(-, +),$ where $\gamma/s < 1.$ On both the central and leftmost diagrams of Figure \ref{fig.rates.regimes}, the shaded regions correspond to the acceleration regime.
\begin{wrapfigure}[16]{l}{0.32\textwidth}
    \vspace{-\baselineskip}
  \begin{center} 
    \includegraphics[width=0.315\textwidth]{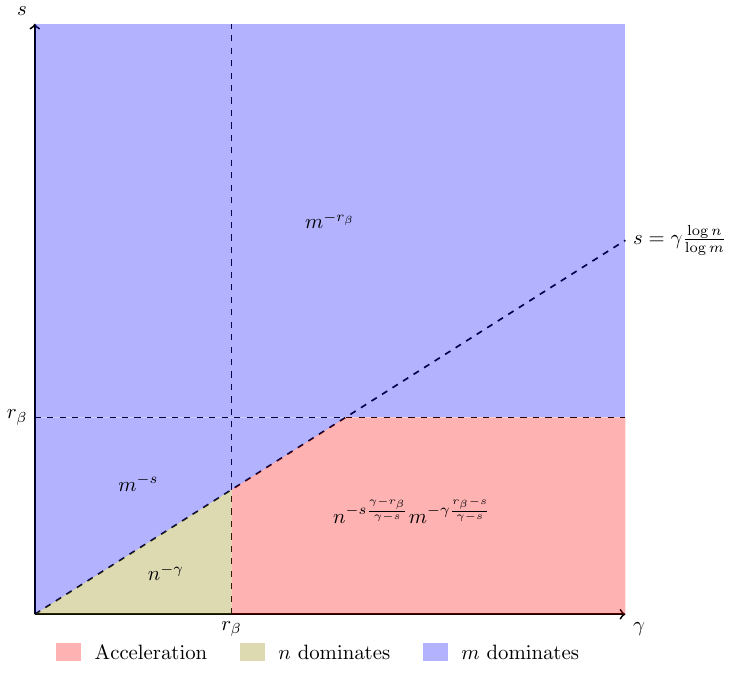}
  \end{center}
  \vspace{-\baselineskip}
  \caption{\label{fig.rates.regimes.nm} Subdivision of $(\gamma, s)$ phases for fixed $(n,m)$ pair. Here, $n < m.$}
\end{wrapfigure}
\indent It is also instructive to fix a particular pair $(n,m)$ and visualise the resulting subdivision of the leftmost diagram of Figure \ref{fig.rates.regimes}. We fix $n < m$ and obtain Figure \ref{fig.rates.regimes.nm}. As will be seen in Section \ref{sec.3.2.3}, the rates transition smoothly from one regime to another. Since $n < m,$ the only possible acceleration regime is the case $s < r_\beta < \gamma,$ corresponding to $\A(+, -).$ Indeed, this is the only way we can have $m \in [n, n^{\gamma/s}].$ The acceleration regime is represented as the red region in Figure \ref{fig.rates.regimes.nm}, while the rest of the diagram corresponds to the wedge regime. In the green region, the wedge rate is driven by the source rate. In the blue region, the rate is driven by the target rate. The blue region is subdivided into two regions. Below the line $s = r_\beta$ the rate is $m^{-s}$ and the rate is $m^{-r_\beta}$ above this line. We see that the red region is the intersection of $\A(+, -)$ with the area under the line $s = \gamma\tfrac{\log n}{\log m}.$ As the ratio $\tfrac{\log n}{\log m}$ increases towards $1$ the red region increases in area. As the ratio becomes larger than $1,$ the diagram shifts to its reflection with respect to the line $s = \gamma.$ That is because if $n > m,$ then the only way for acceleration to happen is to have $n^{\gamma/s} \leq m,$ meaning $\gamma < s.$ Hence, the acceleration regime is represented by the intersection of $\A(-, +)$ with the area above the line $s =\gamma \tfrac{\log n}{\log m}.$ \\

This concludes our general mapping of configurations and regimes and provides a basis for the next sections, where we discuss and visualise phase transitions for $(\gamma, s)$ and $(n,m).$ Since the acceleration regimes are symmetric, we will focus on the case $s < r_\beta < \gamma.$ 

\subsubsection{Phase transition of transferability indices}
\label{sec.3.2.2}
The acceleration regime is better represented in a $(\log n, \log m)$ coordinate system. We begin by illustrating a phase transition in Figure \ref{fig.rates.transition.regimes_accel_log}. The mechanics go as follows. We fix $\gamma$ and $r_\beta$ and give three snapshots, in $(\log n, \log m)$ coordinates, of a smooth transition of $s$ from $s \in (r_\beta, \gamma)$ towards $s \in (0, r_\beta),$ or, equivalently, $(\gamma, s)$ transitions from $\A(+, +)$ to $\A(+, -).$ Figure \ref{fig.rates.transition.regimes_accel_log} represents three stages of this transition.
\begin{figure}[H]
\centering
\begin{tabular}{@{}ccc@{}}
    \includegraphics[width=0.315\textwidth, trim=3pt 0cm 4pt 0cm, clip]{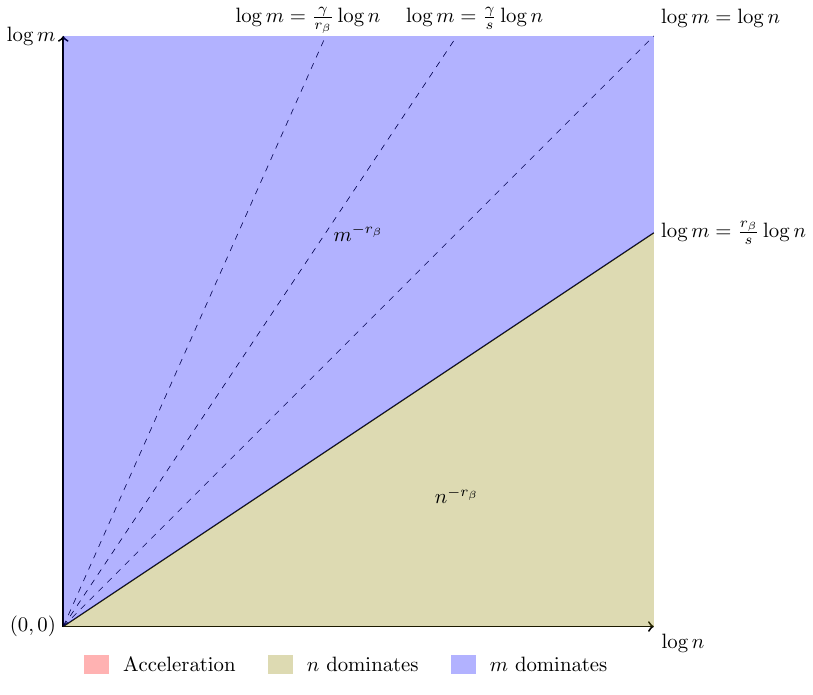} &
    \includegraphics[width=0.315\textwidth, trim=3pt 0cm 4pt 0cm, clip]{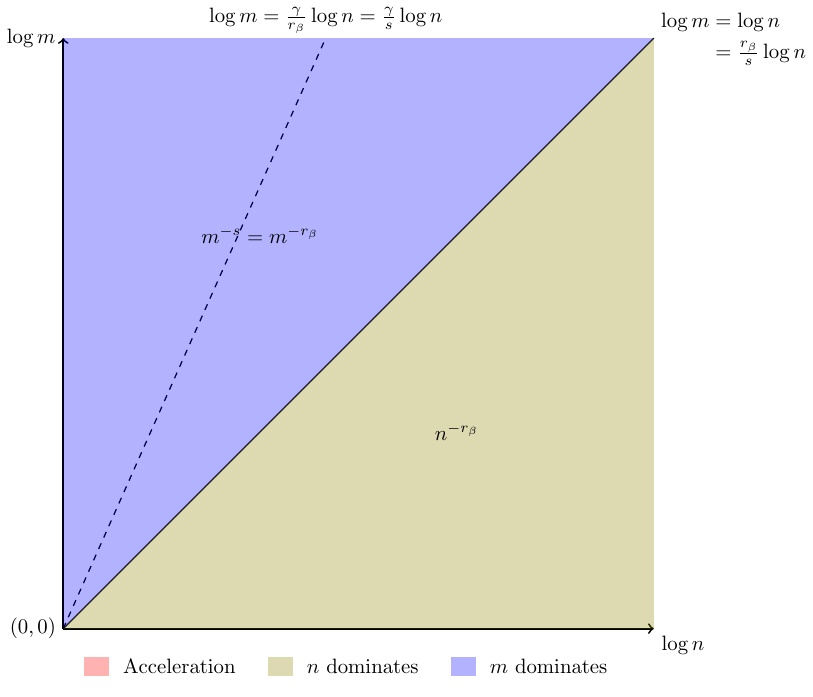} &
    \includegraphics[width=0.315\textwidth, trim=3pt 0cm 4pt 0cm, clip]{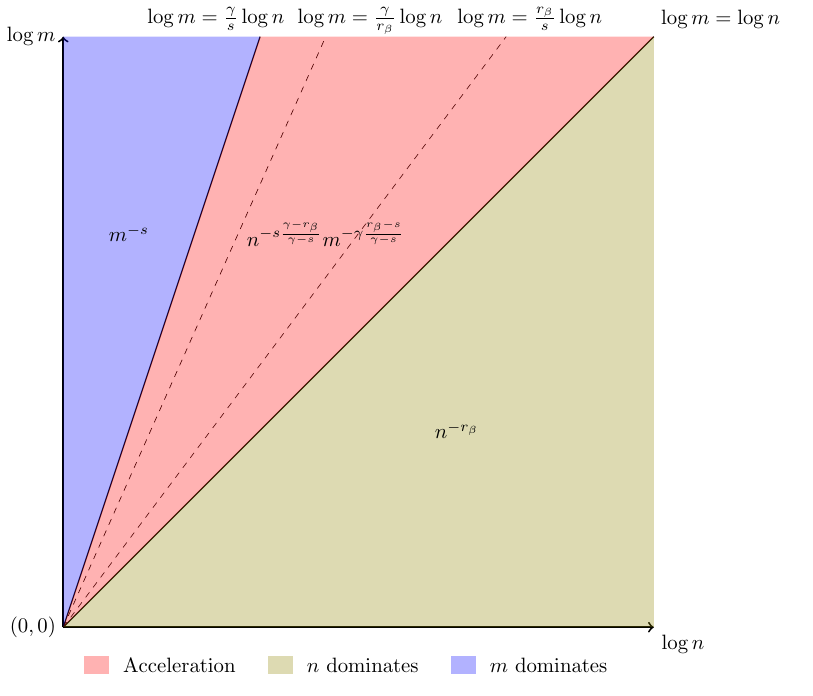}
\end{tabular}
\caption{Phase transition in $(\log n, \log m)$ coordinates with $\gamma/r_\beta$ fixed. \textbf{Left:} $r_\beta < s < \gamma,$ subcritical configuration. \textbf{Centre:} $r_\beta = s < \gamma,$ critical configuration. \textbf{Right:} $s < r_\beta < \gamma,$ supercritical configuration.}
\label{fig.rates.transition.regimes_accel_log}
\end{figure}
The leftmost diagram of Figure \ref{fig.rates.transition.regimes_accel_log} corresponds to $(\gamma, s) \in A(+, +).$ The rate is the wedge rate. The green region represents the pairs $(\log n, \log m)$ for which the rate is driven by the source rate and the blue region represents the pairs for which it is driven by the target rate. We now describe the solid and dashed lines on the leftmost diagram. The boundary between these two regions is given by the line $s\log m = r_\beta\log n$, which we will call $a(s)$. The line $s\log m = \gamma \log n,$ which we will call $b(s)$. The latter represents $m = n^{\gamma/s}$, corresponding to the maximal $m$ for which the acceleration regime would occur, if it were possible. In contrast, the line $\log m = \log n$ represents the smallest $m$ for which acceleration would occur and is called $(I)$. The line $r_\beta \log m = \gamma \log n$, denoted $(M),$ corresponds to pairs $(n,m)$ that balance the two factors in the accelerated rates and is represented for reference.\\

As $s$ decreases towards $r_\beta,\ a(s)$ collapses towards $(I)$ while $b(s)$ collapses towards $(M)$, resulting in the central diagram of Figure \ref{fig.rates.transition.regimes_accel_log}. The latter represents the critical configuration $s = r_\beta.$ As soon as $s < r_\beta,$ we obtain the rightmost diagram of Figure \ref{fig.rates.transition.regimes_accel_log}. There, the configuration $(\gamma, s, r_\beta)$ is supercritical. The lines $a(s)$ and $(M)$ are in the red cone delimited by $(I)$ and $b(s)$. The red cone is the set of $(\log n, \log m)$ corresponding to the acceleration regime.\\

We have described the emergence of acceleration regimes as one passes a phase transition in $(\gamma, s).$ Another complementary perspective arises from fixing $(\gamma, s)$ at supercriticality and visually analysing the rates along paths of source-target sample sizes $(n,m)$. This leads us to the next section.

\subsubsection{Samples sizes paths}
\label{sec.3.2.3}
We now freeze the configuration of the rightmost diagram of Figure \ref{fig.rates.transition.regimes_accel_log} and explore the evolution of the rates along two paths of $(n,m)$ in the $(\log n, \log m)$ plane. One main point of this section is to show, visually, that the transitions between the wedge and accelerated rates are smooth, and that the accelerated rates are significantly smaller than the wedge rates in the acceleration regime.\\

Given the structure of the rates, we represent their evolution on a logarithmic scale. The first path is called \textit{linear path} and is given by pairs $(n,m)$ satisfying $\log n + \log m = \log nm = R.$ Letting $R = \log(B),$ such a path can be parametrised by $\lambda \in [0, 1]$ as $n = B^{1-\lambda}$ and $m = B^{\lambda}.$ The second path corresponds to a \textit{fixed budget} path where the total sample size is fixed. That is $n+m = B,$ or $\lambda \mapsto ((1-\lambda)B, \lambda B),$ and $n\wedge m \geq 1.$ The paths in the $(\gamma, s)$ plane and the evolution of the logarithm of the corresponding rates along them in Figure \ref{fig.rates.transition.nm}.
\begin{figure}[H]
\centering
\begin{tabular}{@{}ccc@{}}
    \includegraphics[width=0.315\textwidth, trim=3pt 0cm 4pt 0cm, clip]{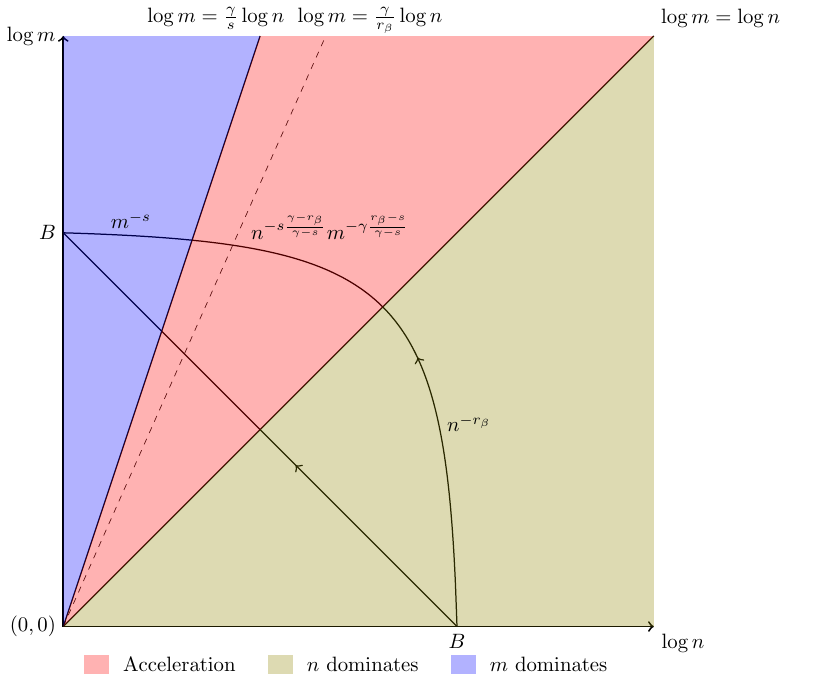} &
    \includegraphics[width=0.315\textwidth, trim=3pt 0cm 4pt 0cm, clip]{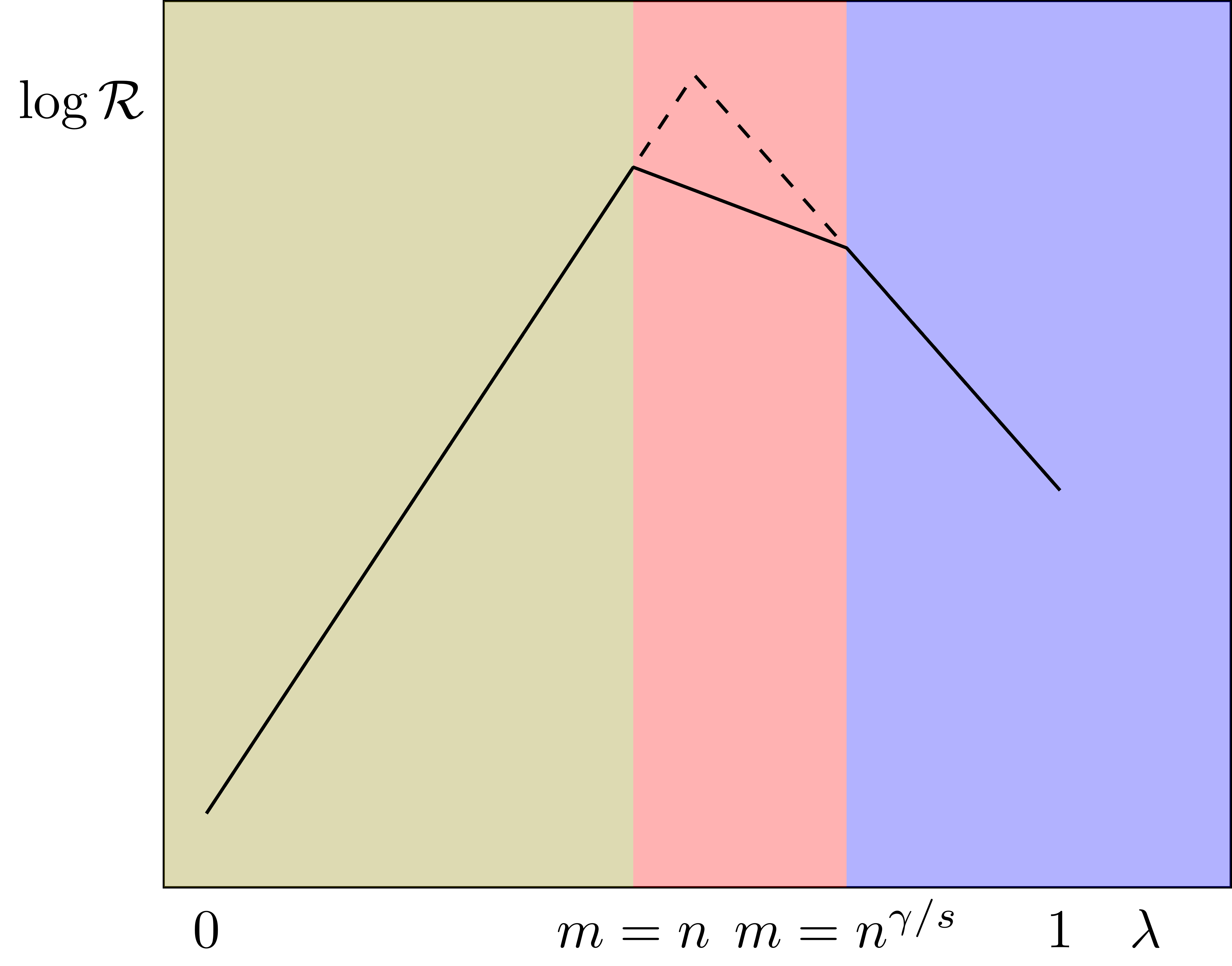} &
    \includegraphics[width=0.315\textwidth, trim=3pt 0cm 4pt 0cm, clip]{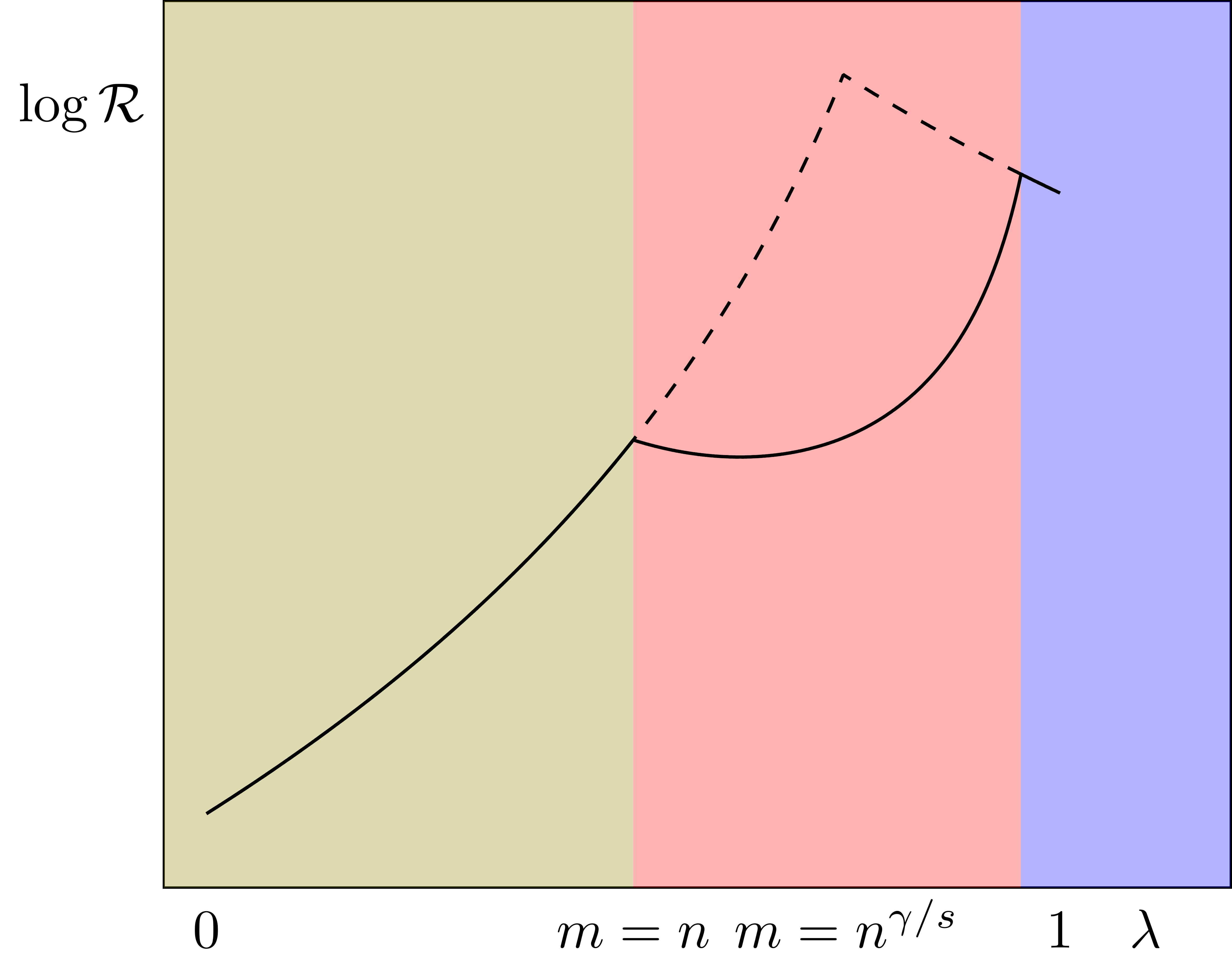}
\end{tabular}
\caption{\textbf{Left:} Two paths in $(\log n, \log m)$ coordinates. The linear path corresponds to $(n, m) = (B^{1-\lambda}, B^{\lambda})$ and the concave path corresponds to $(n, m) = ((1 - \lambda)B, \lambda B).$ \textbf{Centre:} evolution of the logarithm of the rates along the linear path as a function of $\lambda.$ \textbf{Right:} evolution of the logarithm of the rates along the fixed budget path as a function of $\lambda.$}
\label{fig.rates.transition.nm}
\end{figure}
In the central and leftmost plots of Figure \ref{fig.rates.transition.nm}, the dotted curves correspond to the wedge rate. We see clearly that the accelerated rate interpolates smoothly the wedge rate at the boundary points $m = n$ and $m = n^{\gamma/s},$ and that the accelerated rate is lower than the wedge rate. The latter can also be seen by parametrising $m^s = n^{u \gamma + (1 - u)s}$ for $u \in [0, 1].$ Then the accelerated rate becomes $n^{-(r_\beta + u(\gamma - r_\beta))},$ which is strictly smaller than the wedge rate for all $u \in (0, 1)$. This concludes our exposition of phase diagrams that emerge from the minimax rates of Theorems \ref{th.upper.bound.two.sample} and \ref{th.lower.bound}.

\subsection{Heuristics, related works and examples}
\label{sec.3.3}

We now begin by providing a heuristic explanation of how and why the accelerated rates emerge. We then discuss our results in light of another work that obtained similar regimes, and conclude with examples of the resulting rates for concrete source-target pairs.\\

\noindent\textbf{Heuristics for the emergence of an accelerated regime.} Let us now heuristically extract the mechanism behind the emergence, or lack thereof, of an accelerated regime. Firstly, we pick an estimator that optimises the bias-variance tradeoff \textit{pointwise} in $x.$ The corresponding steps are not relevant here and are detailed in Section \ref{sec.4.2}. This rules out any possibility of improving the \textit{integrated} error rate. The resulting pointwise squared error is of the form
\begin{align*}
    (\wh f(x) - f(x))^2 &\leq (np(x) + mq(x))^{-r_\beta} \asymp (np(x) \vee mq(x))^{-r_\beta},
\end{align*}
where we see that the bias-variance mechanics impose the exponent $r_\beta,$ and render it unavoidable. In particular, its appearance does not depend on any transferability index of the transfer functions, and this property would still hold, even for a weakening of the regularity assumptions of $\mP(D, \theta).$ We now introduce a pair of admissible parameters $(\gamma, s).$ In order to derive an excess risk bound, the next step is to integrate the latter against $\Q\mathstrut_{\!\sX}.$ Writing this integral, and using $a \wedge b = \inf_{\lambda \in [0, 1]} a^\lambda b^{1-\lambda}$ for all $\lambda \in [0, 1]$ and for $a, b > 0,$ results in
\begin{align}
    \label{eq.geom.interp.inside.E}
    \big\|\wh f - f_*\big\|_{L^2(\Q\mathstrut_{\!\sX})}^2 &\leq \E_{\Q\mathstrut_{\!\sX}}\bigg[\inf_{\lambda \in [0, 1]}\Big\{(np(X))^{-\lambda r_\beta} (mq(X))^{-(1 - \lambda)r_\beta}\Big\}\bigg].
\end{align}
The minimising $\lambda$ is zero if $mq(x) > np(x)$ and one if $mq(x) < np(x).$ However, in order to make the accelerated regime appear, we instead introduce a parameter $\tau > 1$ and a subset $E_\tau := \{x: \tau^{-1}mq(x) \leq np(x) \leq \tau mq(x)\},$ where $np(x)$ is \textit{comparable} with $mq(x).$ For $x \in E_\tau,$ it holds that $(np(x))\wedge (nq(x)) \asymp (np(x))^{\lambda}(mq(x))^{1-\lambda}$ for any $\lambda \in [0, 1],$ so $\lambda$ can be seen a free tuning parameter instead of being imposed by the imbalance between $np(x)$ and $mq(x).$ Thus, we can separate the integral into three terms. To do so, let us denote by $E_- := \{x: np(x) < \tau^{-1}mq(x)\}$ and by $E_+ := \{x: \tau mq(x) < np(x)\}.$ This, and the previous considerations on the choice of $\lambda,$ leads to
\begin{align}
    \label{eq.three.terms}
    \begin{split}
    \big\|\wh f - f_*\big\|_{L^2(\Q\mathstrut_{\!\sX})}^2 &\leq \E_{\Q\mathstrut_{\!\sX}}\bigg[(mq(X))^{-r_\beta}\mathds{1}_{E_-}(X)\bigg] + \E_{\Q\mathstrut_{\!\sX}}\bigg[(np(X))^{-r_\beta}\mathds{1}_{E_+}(X) \bigg]\\
    &\qquad +\E_{\Q\mathstrut_{\!\sX}}\bigg[(np(X))^{-\lambda r_\beta} (mq(X))^{-(1 - \lambda)r_\beta}\mathds{1}_{E_\tau}(X) \bigg].
    \end{split}
\end{align}
As a result, the integrated error will be of the order of the dominating term in \eqref{eq.three.terms}. It is crucial to mention that the third term, which is obviously the one leading to an accelerated rate, can not appear when all the samples come from either 
$\P_{\sX}$ or $\Q\mathstrut_{\!\sX}.$ Intuitively, the third term \textit{can} dominate under the condition that $\Q\mathstrut_{\!\sX}(E_\tau)$ is large, i.e.\ when $np(x)$ is of the same \enquote{order} as $mq(x)$ on a large enough subset of $\RR^d.$ This explains that the emergence of the accelerated regime depends on $n$ and $m$. Strictly speaking, the size of $\Q\mathstrut_{\!\sX}(E_\tau)$ also depends on $\gamma$ and $s.$ However, their roles are better explained as follows. Since the parameters $\gamma$ and $s$ are tied to the existence of moments of $p(X)^{-1}$ and $q(X)^{-1}$ under $\Q\mathstrut_{\!\sX},$ they dictate integrability, in terms of $\lambda,$ of both terms inside the expectation in \eqref{eq.three.terms}. As a consequence, they fully determine the optimal $\lambda$ that leads (or not, depending on the configuration $(\gamma, s)$) to the accelerated rate. This concludes our heuristic explanation of the accelerated rates and their emergence.\\

\noindent\textbf{Related work.} \label{rem.synergistic} We now compare our results with the recent work \cite{zhou2025synergistic}, which, alongside \cite{schmidt2022local}, is, to our knowledge, one of the rare works in the literature that reports a multiplicative convergence rate for the covariate shift problem.\\

The authors show a similar result in the nonparametric regression setting with $d=1$ for $\Q\mathstrut_{\!\sX} = \mU([0, 1])$ and $\P_{\sX}$ in a one-parameter class of power laws, that is, for $a > 0,$ the density of $\P_{\sX}$ is proportional to $x^{a - 1}\mathds{1}(x \in [0, 1]).$ Although their rates are expressed in terms of $a$ and $r_\beta,$ it is possible to compute the corresponding transferability indices. We find $\gamma^* = 1/(a - 1)$ and $s^* = \infty.$ With these notations, their minimax rates read
\begin{align*}
    \inf_{\wh f}\sup_{f_* \in \mH(L, \beta)}\E\Big[\big\|\wh f - f_*\big\|_{L^2(\Q\mathstrut_{\!\sX})}^2\Big] \gtrsim \begin{dcases}
        n^{-r_\beta}\wedge m^{-r_\beta} & \text{ if } \gamma \geq r_{\beta}\\
        n^{-\tfrac{2\beta + 1}{2\beta + a}}\wedge m^{-r_\beta} & \text{ if } \gamma^* < r_\beta \ \text{ and } m \notin \Big[n, n^{\tfrac{2\beta + 1}{2\beta + a}}\Big]\\
        m^{-(r_\beta-\gamma^*)}n^{-\gamma^*} & \text{ if } \gamma^* < r_\beta \ \text{ and } m \in \Big[n, n^{\tfrac{2\beta + 1}{2\beta + a}}\Big].
    \end{dcases}
\end{align*}
First, we observe that our notion of (super/sub)critical configurations is consistent with their results. Indeed, since $s^* = \infty,$ the only supercritical configuration is $\gamma^* < r_\beta.$ Next, our setting does not include power-law behaviour due to the local mass assumption. However, heuristically sending $s$ to infinity in our accelerated rates yields the same rate exponents as those in \cite{zhou2025synergistic}. These matchings suggest a certain universality in the transferability indices. Supercritical configurations correspond to $r_\beta \in [\gamma^*, s^*]$ and, under additional conditions on $(n,m)$ lead to an accelerated rate of the form $m^{-a}n^{-b}$ with $a + b = r_\beta.$\\

We now discuss aspects of the results that do not match. Sending $s^*$ to infinity in our results leads to the acceleration regime being conditional on $m \in [1, n],$ which is (typically) much larger than their interval $[n^{(2\beta + 1)/(2\beta + a)}, n].$ In the (sub)critical configuration, our wedge rates do match theirs, but not in the supercritical non-accelerated case. The discrepancy in the rates indicates that the class of distributions itself influences both the rates and the range of $(n,m)$ for which acceleration happens. Although tempting, we will not attempt to infer the missing parameter(s) by comparing the two results, especially since convergence rates computed for explicit densities tend to obscure the identification of key parameters, as will be seen in the examples below.\\

\noindent\textbf{Examples.} We now give illustrative examples of rates obtained with Pareto and exponential source-target pairs. The proofs for those examples can be found in Appendix \ref{app.examples}. We stress again that we omit logarithmic factors and set multiplicative constants to one.
\begin{example}[Pareto source-target pairs]
    \label{ex.pareto}
    We consider the two-parameter family of Pareto distributions on $\RR$ denoted by $\Par(\alpha, \sigma).$ For $\alpha, \sigma > 0,$ the density $r$ of $\Par(\alpha, \sigma)$ is defined as
    \begin{align*}
        r(x) = \frac{\alpha}{\sigma}\Big(1 + \frac{x}{\sigma}\Big)^{-(\alpha+1)}\mathds{1}(x\geq 0).
    \end{align*}
    Let $\alpha_{\P}, \alpha_{\Q}, \sigma_{\P}, \sigma_{\Q} > 0$ and consider $\P_{\sX} = \Par(\alpha_{\P}, \sigma_{\P})$ and $\Q\mathstrut_{\!\sX} = \Par(\alpha_{\Q}, \sigma_{\Q})$ with respective densities $p$ and $q.$ For $D, \theta$ chosen as
    \begin{align*}
        D = (\alpha_{\P}/\sigma_{\P})\vee(\alpha_{\Q}/\sigma_{\Q}), \ \text{ and } \ \theta = 2\big(1 + 1/(\sigma_{\P}\wedge\sigma_{\Q})\big)^{\alpha_{\P}\wedge\alpha_{\Q} + 1},
    \end{align*}
    it holds that $\P_{\sX}, \Q\mathstrut_{\!\sX} \in \mP(D, \theta).$ Additionally, the transferability indices are
    \begin{align*}
        \gamma^* = \gamma^*(\P_{\sX}, \Q\mathstrut_{\!\sX}) = \frac{\alpha_{\Q}}{\alpha_{\P} + 1}, \ \text{ and } s^* = \gamma^*(\Q\mathstrut_{\!\sX}, \Q\mathstrut_{\!\sX}) = \frac{\alpha_{\Q}}{\alpha_{\Q} + 1}.
    \end{align*}
    To identify the exponents in the rates, we can verify that
    \begin{align*}
        s^* < r_\beta \Leftrightarrow \alpha_{\Q} < 2\beta/d, \ \text{ and } \ \gamma^* < r_\beta \Leftrightarrow \alpha_{\P} > \alpha_{\Q}/r_\beta.
    \end{align*}
    If both $\alpha_{\Q} < 2\beta/d$ and $\alpha_{\P} > \alpha_{\Q}/r_\beta,$ then the wedge rates read
    \begin{align*}
        n^{-\tfrac{\alpha_{\Q}}{\alpha_{\P} + 1}} \wedge m^{-\tfrac{\alpha_{\Q}}{\alpha_{\Q} + 1}}.
    \end{align*}
    The first necessary condition for the  accelerated rates to occur is
    \begin{align}
        \label{eq.accel.pareto}
        \underbrace{\Big(\alpha_{\Q} < \frac{2\beta}{d} \ \text{ and } \ \alpha_{\P} \leq \frac{\alpha_{\Q}}{r_\beta}\Big)}_{s^* < r_\beta < \gamma^*} \ \text{ or } \ \underbrace{\Big(\alpha_{\Q} \geq \frac{2\beta}{d} \ \text{ and } \ \alpha_{\P} > \frac{\alpha_{\Q}}{r_\beta}\Big)}_{s^* \geq r_\beta > \gamma^*}.
    \end{align}
    In other words, both distributions must be sufficiently heavy-tailed, or both must be sufficiently light-tailed. The second necessary condition for the accelerated rates to occur is
    \begin{align*}
        \underbrace{n \leq m \leq n^{\tfrac{\alpha_{\Q} + 1}{\alpha_{\P} + 1}}}_{s^* < r_\beta \leq \gamma^*} \ \text{ or } \ \underbrace{n^{\tfrac{\alpha_{\Q} + 1}{\alpha_{\P} + 1}} \leq m \leq n}_{s^* \geq r_\beta > \gamma^*}.
    \end{align*}
    In both cases, the accelerated rate is
    \begin{align*}
        n^{-\tfrac{(\alpha_{\Q} + 1)r_\beta - \alpha_{\Q}}{\alpha_{\Q} - \alpha_{\P}}}m^{-\tfrac{\alpha_{\Q} - (\alpha_{\P} + 1)r_\beta}{\alpha_{\Q} - \alpha_{\Q}}}.
    \end{align*}
\end{example}

\begin{example}[Exponential source-target pairs]
    \label{ex.exponential}
    We consider the family of exponential distributions on $\RR$ denoted by $\mE(\lambda).$ For $\lambda > 0,$ the density $r$ of $\mE(\lambda)$ is defined as $r(x) = \lambda e^{-\lambda x}\mathds{1}(x \geq 0).$ Let $\lambda_{\P}, \lambda_{\Q} > 0$ and consider $\P_{\sX} = \mE(\lambda_{\P})$ and $\Q\mathstrut_{\!\sX} = \mE(\lambda_{\Q})$ with respective densities $p$ and $q.$ For $D = \lambda_{\P}\vee \lambda_{\Q}$ and $\theta = e^{D}$ it holds that $\P_{\sX}, \Q\mathstrut_{\!\sX} \in \mP(D, \theta).$ The transferability indices are
    \begin{align*}
        \gamma^* = \gamma^*(\P_{\sX}, \Q\mathstrut_{\!\sX}) = \frac{\lambda_{\Q}}{\lambda_{\P}}, \ \text{ and } s^* = \gamma^*(\Q\mathstrut_{\!\sX}, \Q\mathstrut_{\!\sX}) = 1.
    \end{align*}
    Alternatively, we may say that $s^*$ can be taken as any value in $[0, 1).$ The wedge rates read
    \begin{align*}
        n^{-\tfrac{\lambda_{\Q}}{\lambda_{\P}}\wedge r_\beta} \wedge m^{-r_\beta}.
    \end{align*}
    The only occurrence of accelerated rates is when $\lambda_{\Q} \leq r_\beta \lambda_{\P},$ and $n^{\lambda_{\Q}/\lambda_{\P}} \leq m \leq n.$ In particular, $\Q\mathstrut_{\!\sX}$ must have sufficiently heavier tails than $\P_{\sX}.$ Then, the accelerated rate is
    \begin{align*}
        n^{-\lambda_{\Q}\tfrac{1 - r_\beta}{\lambda_{\P} - \lambda_{\Q}}}m^{-\tfrac{r_\beta\lambda_{\P} - \lambda_{\Q}}{\lambda_{\P} - \lambda_{\Q}}}.
    \end{align*}
\end{example}

\section{Transfer function, estimator, and local mass}
\label{sec.4}

This section provides technical details underlying our results. We begin by studying further properties of the transfer function introduced in Section \ref{sec.2}. We then give details on the estimator that achieves the rates in the upper bounds of Theorem \ref{th.upper.bound.two.sample} and Corollary \ref{cor.upper.bound.one.sample.P}, and we end with a discussion on the class of distributions $\mP(D, \theta).$

\subsection{Transfer function}
\label{sec.4.1}

This section exhibits several properties of the transfer function defined in Section \ref{sec.2} as well as its impact on the convergence rates of Theorems \ref{th.upper.bound.two.sample} and \ref{th.lower.bound}. We describe its behaviour as a function of $\gamma$ and explain what it measures. Then, we focus on its transferability indices. There, we link the latter to standard quantities, such as moments and R\'enyi entropies. Finally, we discuss how the blow-up of $\mT(\P, \Q, \gamma)$ close to $\gamma^*(\P, \Q)$ impacts the convergence rates in our upper bounds.\\

For this section, we consider two arbitrary distributions $\P$ and $\Q$ with respective densities $p$ and $q.$ Let us recall that the transfer function is defined as 
\begin{align*}
    \mT(\P, \Q, \gamma) = \E_{X \sim \Q} \big[p(X)^{-\gamma}\big].
\end{align*}
The transfer function satisfies $\mT(\P, \Q, 0) = 1,$ is $\log$-convex and $\gamma \mapsto \mT(\P, \Q, \gamma)^\gamma$ is non-decreasing (see Lemma \ref{lem.ratio.bound} in Appendix \ref{app.4.1}). The transferability index $\gamma^*(\P, \Q)$ is the eventual (maybe infinite) boundary point where the transfer function ceases being finite and blows up to infinity. Any $\gamma$ for which the transfer function is finite gives an approximation of the amount of $\Q$-mass that lies in low-density regions of $\P.$ Indeed, by Markov's inequality, for any $0 \leq \gamma < \gamma^*(\P, \Q),$ it holds that 
\begin{align*}
    \Q\{x: p(x) \leq t\} \leq t^{\gamma}\mT(\P, \Q, \gamma).
\end{align*}
Similarly, the transfer function controls the truncated moments of $p(X)^{-1}$. This type of control is primarily technical and has been omitted here. It is, however, the type of control that bounds the third term in \eqref{eq.three.terms} of Section \ref{sec.3.3} and yields the accelerated rate. We now focus on the transfer function's transferability indices.

First, whenever $\P = \Q$ with unbounded support, we immediately have the upper bound $\gamma^*(\P, \P) < 1.$ Additionally, lower bounds on the transferability index follow from the existence of generalised moments of the distribution.
\begin{lemma}
    \label{lem.moments.gamma}
    Let $X \sim \P \in \mM$ and $\rho, \eps > 0.$ If $\E[\| X \|^{\rho + \eps}] < \infty,$ then $\gamma^*(\P, \P) \geq \rho/(\rho + d).$
\end{lemma}
When $\P \neq \Q,$ obtaining lower bounds on $\gamma^*(\P, \Q)$ from other known quantities is more involved. Here, we obtain lower bounds in terms of indices of R\'enyi/Tsallis divergences. For $\alpha > 0, \ \alpha \neq 1,$ the $\alpha$-Tsallis divergence between $\Q$ and $\P$ in $\mM$ is defined as 
\begin{align*}
    D^T_{\alpha}(\Q\Vert \P) := \frac{1}{\alpha - 1}\Big(1 - \int q(x)^{\alpha}p(x)^{1- \alpha}\, dx\Big) = \frac{1}{\alpha - 1} \E_{X \sim \Q}\Big[\Big(\frac{q(X)}{p(X)}\Big)^\alpha\Big].
\end{align*}
The R\'enyi divergence between $\Q$ and $\P$ is defined as 
\begin{align*}
    D_{\alpha}(\Q\Vert \P) :=
        \frac{1}{\alpha - 1}\log \int_{\RR^d}q(x)^\alpha p(x)^{1 - \alpha}\, dx = \frac{1}{\alpha - 1}\log\big((\alpha - 1)(1 - D^T_{\alpha}(\Q\Vert \P))\big).
\end{align*}
Both divergences are finite for the same parameter $\alpha.$ Finiteness of $D_{\alpha}(\Q, \P)$ for some $\alpha$ combined with knowledge of $\gamma^*(\P, \P)$ is sufficient to obtain a lower bound on $\gamma^*(\P, \Q).$
\begin{lemma}
    \label{lem.renyi.to.gamma}
    Let $\P, \Q \in \mM$ and $\alpha > 0, \alpha \neq 1.$ If $D_\alpha(\P, \Q) < \infty$, then
    \begin{align*}
        \gamma^*(\P, \P)\Big(\frac{\alpha - 1}{\alpha}\Big) \leq \gamma^*(\P, \Q).
    \end{align*}
\end{lemma}
The finiteness of R\'enyi/Tsallis divergences is equivalent to the density ratio $q/p$ belonging to $L^\alpha(\Q),$ therefore, Lemma \ref{lem.renyi.to.gamma} can be re-interpreted in terms of integrability of the density ratio, combined with integrability of $p^{1-s}.$ The lower bound in Lemma \ref{lem.renyi.to.gamma} is not sharp in the sense that it can become loose for specific pairs $(\P, \Q).$ For example, we can take $\P = \Par(a, 1)$ and $\Q = \mE(1).$ In this case, $\gamma^*(\P, \Q)$ is infinite while $D_{\alpha}(\Q\Vert\P)$ is finite for all $\alpha > 1$ and $\gamma(\P, \P) = a/(a + 1).$ The bound of Lemma \ref{lem.renyi.to.gamma} is not sharp but provides a link between integrability conditions and transferability indices.\\

A final point of interest is the behaviour of the transfer function at the transferability index. First, recalling that $\gamma^*$ stands for $\gamma^*(\P, \Q),$ it is easy to see that depending on the considered pair $(\P, \Q), \ \mT(\P, \Q, \gamma^*)$ can be finite or infinite.
\begin{example}\label{lem.ratio.bound.main.text}
    Let $p$ and $q$ be the respective densities of $\P$ and $\Q.$ Let $a, b > 0$ and $c \geq 0,$ and 
    \begin{align*}
        p \propto \frac{\mathds{1}(x\geq 2)}{x^{a + 1}}\, \text{ and } \ q \propto \frac{\mathds{1}(x \geq 2)}{x^{b+1}\log(x)^{c}} \ \text{ so that } \gamma^* = \frac{b}{a + 1}.
    \end{align*}
    One can check that for all $a, b > 0,$ 
    \begin{align*}
        \text{if } c \leq 1, \ \text{ then } \ \mT(\P_{\sX}, \Q\mathstrut_{\!\sX}, \gamma^*) = \infty, \ \text{ and } \ \mT(\P_{\sX}, \Q\mathstrut_{\!\sX}, \gamma^*) < \infty \ \text{ otherwise.}
    \end{align*}
\end{example}
If the transfer function is finite at the transferability index, then the exponents in the rates given by the lower bound are achievable. In case the transfer function is infinite at the transferability index, then achievability of the exponents in the lower bounds depends on \textit{how} $\mT(\P, \Q, \cdot)$ diverges at $\gamma^*.$ Let us focus on the setting of Corollary \ref{cor.upper.bound.one.sample.P}. Assuming $\gamma^* < r_\beta,$ we can rewrite the upper bound as, with high probability, and up to a $\log n$ factor,
\begin{align*}
    \big\|\wh f - f_*\big\|_{L^2(\Q)}^2 \leq C\mT(\P, \Q, \gamma)n^{\gamma^* - \gamma}n^{-\gamma^*} = Ce^{(\gamma^* - \gamma)\log n + \log \mT(\P, \Q, \gamma)}n^{-\gamma^*}.
\end{align*}
Depending on the behaviour of $\mT(\P, \Q, \gamma)$ at $\gamma^*_-,$ a variety of possible scenarios exist. For example, for a polynomial pole, say, $\mT(\P, \Q, \gamma) = 1/(\gamma^* - \gamma)^a$ for some $a \geq 1,$ then picking $\delta_n := \gamma^* - \gamma_n = a/\log n$ leads to
\begin{align*}
    \big\|\wh f - f_*\big\|_{L^2(\Q)}^2 &\leq Ce^{a(1 - \log(a))}(\log n)^an^{-\gamma^*},
\end{align*}
which shows that when the transfer function blows up polynomially, we pay a poly-logarithmic factor in the rate, while $\gamma^*$ is still achievable as an exponent of $n$.

\subsection{Local nearest neighbours regressor}
\label{sec.4.2}

To obtain our rates, we consider a \textit{local} $k$-nearest neighbours regressor ($k$-NN). Let $\mD_{\P} := \{(X_i, Y_i)\}_{i = 1}^n$ and $\mD_{\Q} := \{(X'_i, Y'_i)\}_{i=1}^m$ be, respectively, two sets of $n$ and $m$ i.i.d.\ observations from $\P_{\sX, \sY}$ and $\Q\mathstrut_{\!\X, \sY}.$ Given two neighbour functions $k_{\P}$ and $k_{\Q},$ the \textit{local} $k$-nearest neighbours estimator is defined as
\begin{align}
    \label{eq.two.sample.local.knn}
    \wh f(x) = \frac{1}{k_{\P}(x) + k_{\Q}(x)}\Bigg[\sum_{i=1}^{k_{\P}(x)} Y_i(x) + \sum_{j=1}^{k_{\Q}(x)} Y'_j(x)\Bigg],
\end{align}
where $Y_i(x)$ and $Y_i'(x)$ denote the label of the $i$-th nearest neighbour of $x$ among $\{X_1, \dots, X_n\}$ and $\{X_1', \dots X_m'\},$ respectively. The functions $k_{\P}$ and $k_{\Q}$ are chosen to balance bias and variance, pointwise and separately, for each data set $\mD_{\P}$ and $\mD_{\Q}.$ Indeed, let us consider a $k$-NN regressor that depends only on $\mD_{\P}.$ Then, it is possible to show that, with high probability and pointwise, the squared error decomposes as
\begin{align*}
    (\wh f(x) - f_*(x))^2 &\lesssim \underbrace{R_k(x)^{2\beta}}_{\text{squared-bias}} + \underbrace{\frac{\log n}{k}}_{\text{variance}},
\end{align*}
where the $\log n$ factor follows from the high-probability nature of the statement, and $R_k(x) = \|x - X_k(x)\|_2^2$ is the distance of $x$ to its $k$-th neighbour among $\{X_1, \dots, X_n\}.$ The latter is then related to its population version $\zeta_{k}(x) := \inf\{r\geq 0: \P_{\sX}\{\B(x, r)\} \geq k/n\},$ and, using the regularity assumption of our distributions, we obtain $\zeta_k(x) \asymp (k/(np(x)))^{1/d},$ which leads to
\begin{align*}
    (\wh f(x) - f_*(x))^2 &\lesssim \bigg(\frac{k_{\P}(x)}{np(x)}\bigg)^{2\beta} + \frac{\log n}{k_{\P}}.
\end{align*}
Balancing the squared bias and variance is done by choosing $k_{\P}^*(x) \asymp \log(n)^{d/(2\beta + d)}(np(x))^{2\beta/(2\beta + d)}.$ Now, since the distribution $\P_{\sX}$ is unknown, we instead use a plug-in approximation. Finally, we need to clip $k_{\P}(x)$ to stay within $\log n,$ for the high-probability statement, and $n.$ This gives us
\begin{align*}
    k_{\P}(x) &\asymp \Big\lceil \log(n) \vee \log(n)^{d/(2\beta + d)}(n\wh p(x))^{2\beta/(2\beta + d)} \wedge n\Big\rceil,
\end{align*}
where $\wh p$ is the $\ell$-nearest neighbours density estimator of $p$. Its expression is given, at $x \in \RR^d,$ by
\begin{align}\label{eq.density.estimator}
    \wh p(x):= \frac{\P_n\{\B(x, R_{\ell}(x))\}}{R_{\ell}(x)^d} = \frac{\ell}{nR_\ell(x)^d},
\end{align}
where $\P_n$ is the empirical distribution of $\{X_1, \dots, X_n\}.$ The results in Theorem \ref{th.upper.bound.two.sample} and Corollary \ref{cor.upper.bound.one.sample.P} are obtained by choosing $\ell \asymp \log((n\vee 1)(m\vee 1)).$ This choice of $\ell$ leads to a density estimator with high variance and low bias, which is coherent with our approach since all we need is to keep the ratio $p/\wh p$ within two positive constants, uniformly over $x$ in a large subset of $\RR^d$. The same procedure can be applied to select $k_{\Q}$ from $\mD_{\Q}$ and obtain the final estimator \eqref{eq.two.sample.local.knn}. Let us now come back to its definition and emphasise that, even if the neighbour functions are chosen separately, the estimator can be argued to choose, among $\mD_{\P}$ and $\mD_{\Q},$ the most informative sample. Plugging \eqref{eq.two.sample.local.knn} in the pointwise error $(\wh f(x) - f_*(x))^2,$ bounding, and rearranging leads to
\begin{align}
   \label{eq.two.sample.risk.is.weighted.of.one.sample}
    (\wh f(x) - f_*(x))^2 &\leq \frac{k_{\P}(x)}{k_{\P}(x) + k_{\Q}(x)}\big(\wh f_{\P}(x) - f_*(x)\big)^2 + \frac{k_{\Q}(x)}{k_{\P}(x) + k_{\Q}(x)}\big(\wh f_{\Q}(x) - f_*(x)\big)^2,
\end{align}
where $\wh f_{\P}$ is a $k_{\P}$-NN estimator trained only on $\mD_{\P}$ and $\wh f_{\Q}(x)$ is a $k_{\Q}$-NN estimator trained only on $\mD_{\Q}.$ We see that the pointwise error is a convex combination of the pointwise errors of $\smash{\wh f}_{\P}$ and $\smash{\wh f}_{\Q}.$ Notice that the neighbour functions $k_{\P}$ and $k_{\Q}$ are monotone in the underlying densities. Consequently, in regions of large imbalance between $p$ and $q,$ the weights in \eqref{eq.two.sample.risk.is.weighted.of.one.sample} favour the estimator obtained from the locally denser sample.

\subsection{Local mass assumptions}
\label{sec.4.3}
In this last part, we discuss the local mass property of the distributions in $\mP(D, \theta).$ This is arguably the most restrictive assumption in our setting. Nevertheless, our constrained setting puts the proper highlight on the phenomenon.\\

One feature of our restricted setting is that it prevents minimax rates from being driven by nearly-singular distributions. Indeed, in \cite{kpotufe2021marginal, pathak2022new, trottner2024covariate} the lower bounds are derived along sequences (in $n, m$) of measures similar to that of \cite[Part III]{kahane1969trois}, which weakly converge to a limiting measure that is singular with respect to the Lebesgue measure. In contrast, our lower bounds are obtained for Pareto distributions.\\

While our main results and the discussion in Section \ref{sec.4.1} suggest that the transfer function is the natural object to consider, even in broader settings, the local mass condition is not. We now recall the findings of Section \ref{rem.synergistic}. We have seen that our results still make sense, to some extent, when we extend them formally to power law distributions. In particular, we captured the structure of the rates in the accelerated regime. However, neither the exponents in the wedge rate nor the thresholds in $n,m$ were reproduced by our results. This suggests further interaction between transferability indices and other quantities. The latter would arise from the shapes of distributions that are \textit{not} included in our setting. The derivation of a quantified weakening of \eqref{eq.minimal.maximal.mass} is out of the scope of this paper, but constitutes the natural continuation of our line of research. Another worthy direction is the case of dimensional mismatch between the supports of $\P_{\sX}$ and $\Q\mathstrut_{\!\sX}.$ This could be approached within the same general framework of our paper by weakening \eqref{eq.minimal.maximal.mass} to only require that for some $\theta > 0,$ and $0 < d' \leq d,$
\begin{align*}
    \theta^{-1}p(x)r^d \leq \P_{\sX}\{\B(x, r)\} \leq \theta p(x)r^d, \ \text{ and } \ \theta^{-1}q(x)r^{d'} \leq \Q\mathstrut_{\!\sX}\{\B(x, r)\} \leq \theta q(x) r^{d'},
\end{align*}
while keeping the upper bound $p\vee q \leq D.$ In principle, it is natural to expect that the case of dimensional mismatch could be encompassed by a more general weakening of \eqref{eq.minimal.maximal.mass}. We leave a precise treatment of this direction for future work.

\printbibliography

@article{kpotufe2021marginal,
    AUTHOR = {Kpotufe, Samory and Martinet, Guillaume},
     TITLE = {Marginal singularity and the benefits of labels in
              covariate-shift},
   JOURNAL = {Ann. Statist.},
  FJOURNAL = {The Annals of Statistics},
    VOLUME = {49},
      YEAR = {2021},
    NUMBER = {6},
     PAGES = {3299--3323},
      ISSN = {0090-5364,2168-8966},
   MRCLASS = {62G99 (62D99 68Q32 68T05)},
  MRNUMBER = {4352531},
       DOI = {10.1214/21-aos2084},
       URL = {https://doi.org/10.1214/21-aos2084},
}

@inproceedings{pathak2022new,
  title = 	 {A new similarity measure for covariate shift with applications to nonparametric regression},
  author =       {Pathak, Reese and Ma, Cong and Wainwright, Martin},
  booktitle = 	 {Proceedings of the 39th International Conference on Machine Learning},
  pages = 	 {17517--17530},
  year = 	 {2022},
  editor = 	 {Chaudhuri, Kamalika and Jegelka, Stefanie and Song, Le and Szepesvari, Csaba and Niu, Gang and Sabato, Sivan},
  volume = 	 {162},
  series = 	 {Proceedings of Machine Learning Research},
  month = 	 {07},
  publisher =    {PMLR},
  url = 	 {https://proceedings.mlr.press/v162/pathak22a.html}}

@book{Biau2015Lectures,
    AUTHOR = {Biau, G\'{e}rard and Devroye, Luc},
     TITLE = {Lectures on the nearest neighbor method},
    SERIES = {Springer Series in the Data Sciences},
 PUBLISHER = {Springer, Cham},
      YEAR = {2015},
     PAGES = {ix+290},
      ISBN = {978-3-319-25386-2},
   MRCLASS = {62G08 (60D05 60E15 62G07 62G20 62H30 68T05 68T10)},
  MRNUMBER = {3445317},
MRREVIEWER = {Christian\ Rau},
       DOI = {10.1007/978-3-319-25388-6},
       URL = {https://doi.org/10.1007/978-3-319-25388-6},
}

@book {MR1385671,
    AUTHOR = {van der Vaart, Aad W. and Wellner, Jon A.},
     TITLE = {Weak convergence and empirical processes},
    SERIES = {Springer Series in Statistics},
      NOTE = {With applications to statistics},
 PUBLISHER = {Springer-Verlag, New York},
      YEAR = {1996},
     PAGES = {xvi+508},
      ISBN = {0-387-94640-3},
   MRCLASS = {60F05 (60B12 62G30)},
  MRNUMBER = {1385671},
MRREVIEWER = {Miguel\ A.\ Arcones},
       DOI = {10.1007/978-1-4757-2545-2},
       URL = {https://doi.org/10.1007/978-1-4757-2545-2},
}

@book {MR2724368,
    AUTHOR = {Kosorok, Michael R.},
     TITLE = {Introduction to empirical processes and semiparametric
              inference},
    SERIES = {Springer Series in Statistics},
 PUBLISHER = {Springer, New York},
      YEAR = {2008},
     PAGES = {xiv+483},
      ISBN = {978-0-387-74977-8},
   MRCLASS = {62-02 (60B05 60B12 60F17 62Gxx)},
  MRNUMBER = {2724368},
MRREVIEWER = {Gutti\ J.\ Babu},
       DOI = {10.1007/978-0-387-74978-5},
       URL = {https://doi.org/10.1007/978-0-387-74978-5},
}

@article{shimodaira2000improving,
    AUTHOR = {Shimodaira, Hidetoshi},
     TITLE = {Improving predictive inference under covariate shift by
              weighting the log-likelihood function},
   JOURNAL = {J. Statist. Plann. Inference},
  FJOURNAL = {Journal of Statistical Planning and Inference},
    VOLUME = {90},
      YEAR = {2000},
    NUMBER = {2},
     PAGES = {227--244},
      ISSN = {0378-3758,1873-1171},
   MRCLASS = {62B10 (62F15)},
  MRNUMBER = {1795598},
       DOI = {10.1016/S0378-3758(00)00115-4},
       URL = {https://doi.org/10.1016/S0378-3758(00)00115-4},
}

@article{li2020transfer,
title = {Transfer learning in computer vision tasks: Remember where you come from},
journal = {Image and Vision Computing},
volume = {93},
pages = {103853},
year = {2020},
issn = {0262-8856},
doi = {https://doi.org/10.1016/j.imavis.2019.103853},
url = {https://www.sciencedirect.com/science/article/pii/S0262885619304469},
author = {Xuhong Li and Yves Grandvalet and Franck Davoine and Jingchun Cheng and Yin Cui and Hang Zhang and Serge Belongie and Yi-Hsuan Tsai and Ming-Hsuan Yang}}

@incollection{wang2022transfer,
author="Wang, Jindong
and Chen, Yiqiang",
title="Transfer Learning for Computer Vision",
bookTitle="Introduction to Transfer Learning: Algorithms and Practice",
year="2023",
publisher="Springer Nature Singapore",
address="Singapore",
pages="265--273"}

@inproceedings{ruder2019transfer,
    title = "Transfer Learning in Natural Language Processing",
    author = "Ruder, Sebastian  and
      Peters, Matthew E.  and
      Swayamdipta, Swabha  and
      Wolf, Thomas",
    editor = "Sarkar, Anoop  and
      Strube, Michael",
    booktitle = "Proceedings of the 2019 Conference of the North {A}merican Chapter of the Association for Computational Linguistics: Tutorials",
    month = {06},
    year = "2019",
    address = "Minneapolis, Minnesota",
    publisher = "Association for Computational Linguistics",
    url = "https://aclanthology.org/N19-5004",
    doi = "10.18653/v1/N19-5004",
    pages = "15--18"}

@article{ali2021enhanced,
title = {An enhanced technique of skin cancer classification using deep convolutional neural network with transfer learning models},
journal = {Machine Learning with Applications},
volume = {5},
pages = {100036},
year = {2021},
issn = {2666-8270},
doi = {https://doi.org/10.1016/j.mlwa.2021.100036},
url = {https://www.sciencedirect.com/science/article/pii/S2666827021000177},
author = {Md Shahin Ali and Md Sipon Miah and Jahurul Haque and Md Mahbubur Rahman and Md Khairul Islam}}

@article{ebbehoj2022transfer,
    doi = {10.1371/journal.pdig.0000014},
    author = {Ebbehoj, Andreas AND Thunbo, Mette {\O}stergaard AND Andersen, Ole Emil AND Glindtvad, Michala Vilstrup AND Hulman, Adam},
    journal = {PLOS Digital Health},
    publisher = {Public Library of Science},
    title = {Transfer learning for non-image data in clinical research: A scoping review},
    year = {2022},
    month = {02},
    volume = {1},
    url = {https://doi.org/10.1371/journal.pdig.0000014},
    pages = {1-22}}

@misc{wainwrighthighdimstat,
    AUTHOR = {Wainwright, Martin J.},
     TITLE = {High-Dimensional Statistics},
    SERIES = {Cambridge Series in Statistical and Probabilistic Mathematics},
    VOLUME = {48},
      NOTE = {A non-asymptotic viewpoint},
 PUBLISHER = {Cambridge University Press, Cambridge},
      YEAR = {2019},
     PAGES = {xvii+552},
      ISBN = {978-1-108-49802-9},
   MRCLASS = {62-01 (60B20 60E15 60Fxx 62Gxx 62Hxx 62Jxx)},
  MRNUMBER = {3967104},
MRREVIEWER = {Pierre Alquier},
       DOI = {10.1017/9781108627771},
       URL = {https://doi.org/10.1017/9781108627771},
}

@book {MR2724359,
    AUTHOR = {Tsybakov, Alexandre B.},
     TITLE = {Introduction to nonparametric estimation},
    SERIES = {Springer Series in Statistics},
      NOTE = {Revised and extended from the 2004 French original,
              Translated by Vladimir Zaiats},
 PUBLISHER = {Springer, New York},
      YEAR = {2009},
     PAGES = {xii+214},
      ISBN = {978-0-387-79051-0},
   MRCLASS = {62-01 (62G05 62G07 62G08 62G20)},
  MRNUMBER = {2724359},
       DOI = {10.1007/b13794},
       URL = {https://doi.org/10.1007/b13794},
}

@article{schmidt2022local,
 author = {Schmidt-Hieber, Johannes and Zamolodtchikov, Petr},
 title = {Local convergence rates of the nonparametric least squares estimator with applications to transfer learning},
 fjournal = {Bernoulli},
 journal = {Bernoulli},
 issn = {1350-7265},
 volume = {30},
 number = {3},
 pages = {1845--1877},
 year = {2024},
 language = {English},
 doi = {10.3150/23-BEJ1655},
 url = {projecteuclid.org/journals/bernoulli/volume-30/issue-3/Local-convergence-rates-of-the-nonparametric-least-squares-estimator-with/10.3150/23-BEJ1655.full},
 zbMATH = {7874403}
}

@article{zhou2025synergistic,
  title={On a synergistic learning phenomenon in nonparametric domain adaptation},
  author={Zhou, Ling and Yang, Yuhong},
  journal={arXiv preprint arXiv:2511.17009},
  year={2025}
}

@article{trottner2024covariate,
    author = {Trottner, Lukas},
    title = {Covariate shift in nonparametric regression with Markovian design},
    journal = {Information and Inference: A Journal of the IMA},
    volume = {13},
    number = {2},
    pages = {iaae011},
    year = {2024},
    month = {05},
    issn = {2049-8772},
    doi = {10.1093/imaiai/iaae011},
    url = {https://doi.org/10.1093/imaiai/iaae011},
    eprint = {https://academic.oup.com/imaiai/article-pdf/13/2/iaae011/57420104/iaae011.pdf},
}

@article{kahane1969trois,
 author = {Kahane, J.-P.},
 title = {Three notes on perfect linear sets},
 fjournal = {L'Enseignement Math{\'e}matique. 2e S{\'e}rie},
 journal = {Enseign. Math. (2)},
 issn = {0013-8584},
 volume = {15},
 pages = {185--192},
 year = {1969},
 language = {French},
 zbMATH = {3280241},
 Zbl = {0175.33902}
}

@article{zamolodtchikov2024transfer,
  title={Transfer Learning under Covariate Shift: Local $ k $-Nearest Neighbours Regression with Heavy-Tailed Design},
  author={Zamolodtchikov, Petr and Hang, Hanyuan},
  journal={arXiv preprint arXiv:2401.11554},
  year={2024}
}

@article{zhu2025,
author = {Zhu, Zhengyu and Yan, Yibo and Li, Gefei and Zhang, Riquan},
title = {Recent Developments on Statistical Transfer Learning},
journal = {International Statistical Review},
volume = {n/a},
number = {n/a},
pages = {},
year = {2025},
doi = {https://doi.org/10.1111/insr.12613},
url = {https://onlinelibrary.wiley.com/doi/abs/10.1111/insr.12613},
eprint = {https://onlinelibrary.wiley.com/doi/pdf/10.1111/insr.12613}
}

@article{ma2023optimally,
 author = {Ma, Cong and Pathak, Reese and Wainwright, Martin J.},
 title = {Optimally tackling covariate shift in {RKHS}-based nonparametric regression},
 fjournal = {The Annals of Statistics},
 journal = {Ann. Stat.},
 issn = {0090-5364},
 volume = {51},
 number = {2},
 pages = {738--761},
 year = {2023},
 language = {English},
 doi = {10.1214/23-AOS2268},
 zbMATH = {7714179},
 Zbl = {1539.62116}
}

@article{cai2024transfer,
 author = {Cai, T. Tony and Kim, Dongwoo and Pu, Hongming},
 title = {Transfer learning for functional mean estimation: phase transition and adaptive algorithms},
 fjournal = {The Annals of Statistics},
 journal = {Ann. Stat.},
 issn = {0090-5364},
 volume = {52},
 number = {2},
 pages = {654--678},
 year = {2024},
 language = {English},
 doi = {10.1214/24-AOS2362},
 zbMATH = {7860628},
 Zbl = {1539.62354}
}

@article{mokhtar2025fine,
  title = {Fine-tuning machine-learned particle-flow reconstruction for new detector geometries in future colliders},
  author = {Mokhtar, Farouk and Pata, Joosep and Garcia, Dolores and Wulff, Eric and Zhang, Mengke and Kagan, Michael and Duarte, Javier},
  journal = {Phys. Rev. D},
  volume = {111},
  issue = {9},
  pages = {092015},
  numpages = {17},
  year = {2025},
  month = {05},
  publisher = {American Physical Society},
  doi = {10.1103/PhysRevD.111.092015},
  url = {https://link.aps.org/doi/10.1103/PhysRevD.111.092015}
}

@article{camaiani2022model,
  title = {Model independent measurements of standard model cross sections with domain adaptation},
  volume = {82},
  ISSN = {1434-6052},
  url = {http://dx.doi.org/10.1140/epjc/s10052-022-10871-3},
  DOI = {10.1140/epjc/s10052-022-10871-3},
  number = {10},
  journal = {The European Physical Journal C},
  publisher = {Springer Science and Business Media LLC},
  author = {Camaiani,  Benedetta and Seidita,  Roberto and Anderlini,  Lucio and Ceccarelli,  Rudy and Ciulli,  Vitaliano and Lenzi,  Piergiulio and Lizzo,  Mattia and Viliani,  Lorenzo},
  year = {2022},
  month = {10} 
}

@inproceedings{mallinar2024minimum,
author = {Mallinar, Neil and Zane, Austin and Frei, Spencer and Yu, Bin},
title = {Minimum-norm interpolation under covariate shift},
year = {2024},
publisher = {JMLR.org},
booktitle = {Proceedings of the 41st International Conference on Machine Learning},
articleno = {1405},
numpages = {43},
location = {Vienna, Austria},
series = {ICML'24}
}

@article{gogolashvili2023importance,
  title={When is Importance Weighting Correction Needed for Covariate Shift Adaptation?},
  author={Gogolashvili, Davit and Zecchin, Matteo and Kanagawa, Motonobu and Kountouris, Marios and Filippone, Maurizio},
  journal={arXiv preprint arXiv:2303.04020},
  year={2023}
}

@article{liu2025spectral,
  title={Spectral Algorithms in Misspecified Regression: Convergence under Covariate Shift},
  author={Liu, Ren-Rui and Guo, Zheng-Chu},
  journal={arXiv preprint arXiv:2509.05106},
  year={2025}
}

@article{feng2024deep,
  author  = {Xingdong Feng and Xin He and Yuling Jiao and Lican Kang and Caixing Wang},
  title   = {Deep Nonparametric Quantile Regression under Covariate Shift},
  journal = {Journal of Machine Learning Research},
  year    = {2024},
  volume  = {25},
  number  = {385},
  pages   = {1--50},
  url     = {http://jmlr.org/papers/v25/24-0906.html}
}

@article{xu2025estimating,
  title={Estimating unbounded density ratios: Applications in error control under covariate shift},
  author={Xu, Shuntuo and Yu, Zhou and Huang, Jian},
  journal={arXiv preprint arXiv:2504.01031},
  year={2025}
}

@article{della2025computational,
  title={Computational efficiency under covariate shift in kernel ridge regression},
  author={Della Vecchia, Andrea and Watusadisi, Arnaud Mavakala and De Vito, Ernesto and Rosasco, Lorenzo},
  journal={arXiv preprint arXiv:2505.14083},
  year={2025}
}

@article{auddy2025minimax,
    author = {Auddy, Arnab and Cai, T Tony and Chakraborty, Abhinav},
    title = {Minimax and adaptive transfer learning for nonparametric classification under distributed differential privacy constraints},
    journal = {Journal of the Royal Statistical Society Series B: Statistical Methodology},
    pages = {qkaf070},
    year = {2025},
    month = {11},
    issn = {1369-7412},
    doi = {10.1093/jrsssb/qkaf070},
    url = {https://doi.org/10.1093/jrsssb/qkaf070},
    eprint = {https://academic.oup.com/jrsssb/advance-article-pdf/doi/10.1093/jrsssb/qkaf070/65272455/qkaf070.pdf},
}

@article{cai2024transferbandit,
 author = {Cai, Changxiao and Cai, T. Tony and Li, Hongzhe},
 title = {Transfer learning for contextual multi-armed bandits},
 fjournal = {The Annals of Statistics},
 journal = {Ann. Stat.},
 issn = {0090-5364},
 volume = {52},
 number = {1},
 pages = {207--232},
 year = {2024},
 language = {English},
 doi = {10.1214/23-AOS2341},
 url = {projecteuclid.org/journals/annals-of-statistics/volume-52/issue-1/Transfer-learning-for-contextual-multi-armed-bandits/10.1214/23-AOS2341.full},
 zbMATH = {7815256},
 Zbl = {1539.62097}
}

@inproceedings{chen2024high,
author = {Chen, Yihang and Liu, Fanghui and Suzuki, Taiji and Cevher, Volkan},
title = {High-dimensional kernel methods under covariate shift: data-dependent implicit regularization},
year = {2024},
publisher = {JMLR.org},
booktitle = {Proceedings of the 41st International Conference on Machine Learning},
articleno = {273},
numpages = {22},
location = {Vienna, Austria},
series = {ICML'24}
}

@article{sugiyama2007covariate,
  author  = {Masashi Sugiyama and Matthias Krauledat and Klaus-Robert M{{\"u}}ller},
  title   = {Covariate Shift Adaptation by Importance Weighted Cross Validation},
  journal = {Journal of Machine Learning Research},
  year    = {2007},
  volume  = {8},
  number  = {35},
  pages   = {985--1005},
  url     = {http://jmlr.org/papers/v8/sugiyama07a.html}
}

@article{song2024generalization,
  title={Generalization error of min-norm interpolators in transfer learning},
  author={Song, Yanke and Bhattacharya, Sohom and Sur, Pragya},
  journal={arXiv preprint arXiv:2406.13944},
  year={2024}
}

@article{blanchard2024estimation,
  title={Estimation of multiple mean vectors in high dimension},
  author={Blanchard, Gilles and Fermanian, Jean-Baptiste and Marienwald, Hannah},
  journal={arXiv preprint arXiv:2403.15038},
  year={2024}
}

@article{blanchard2021domain,
author = {Blanchard, Gilles and Deshmukh, Aniket Anand and Dogan, \"{U}run and Lee, Gyemin and Scott, Clayton},
title = {Domain generalization by marginal transfer learning},
year = {2021},
issue_date = {January 2021},
publisher = {JMLR.org},
volume = {22},
number = {1},
issn = {1532-4435},
journal = {Journal of Machine Learning Research},
month = {01},
articleno = {2},
numpages = {55},
url = {http://jmlr.org/papers/v22/17-679.html}
}

\newpage

\appendix

\section{Preliminary Results} 
\label{app.prelim}

\subsection{Uniform control on neighbours}

\begin{proposition}\label{prop.covering.number}
Let $\mB := \{\B(x, r) : x \in \RR^d, r > 0 \}$. 
The VC dimension of  $\mB$ is $2d + 1$. 
Moreover, there exists a universal constant $K$ such that for any $\eps \in (0, 1)$, 
\begin{align*}
\mathcal{N}(\mathds{1}_{\mB}, \|\cdot\|_{L_1(\Q)}, \eps)
\leq K (2d+1) (4e)^{2d+1} \eps^{-2d}
\end{align*}
holds for any probability measure $\Q,$ where $\mathds{1}_{\mB} := \{\mathds{1}(\cdot \in B), B \in \mB\}.$
\end{proposition}

\begin{proof}
The first result of VC dimension follows from Example 2.6.1 in \cite{MR1385671}. The second result on the covering number follows from Theorem 9.2 in \cite{MR2724368}.
\end{proof}

\begin{lemma}[Bernstein's Inequality]
    Let $X_1, \dots, X_n$ be $n$ i.i.d.\ random variables with values in $\RR.$ If $\PP\{|X_1| \leq C\} = 1$ for some $C > 0,$ and $\E[X_1] = \mu,$ then, for all $t > 0,$
    \begin{align}
        \label{eq.Bernstein}
        \PP\bigg\{\Big|\frac 1n \sum_{i=1}^n X_i - \mu\Big| \geq t\bigg\} \leq 2\exp\bigg(-\frac{nt^2}{2\sigma^2 + 2Ct/3}\bigg),
    \end{align}
    where $\sigma^2 = \Var(X_1).$
\end{lemma}
\begin{definition}
    For $\P \in \mM,$ and $h > 0,$ we define the function $\zeta_h^{\P}$ as 
    \begin{align}
        \label{eq.zeta.definition}
        \zeta_h^{\P}\colon x\mapsto \inf\Big\{r > 0: \P\{\B(x, r)\} \geq h\Big\}.
    \end{align}
    When the considered distribution $\P$ is clear from the context, we will drop the $\P$ exponent in the notations and write $\zeta_h$ instead.
\end{definition}
Next, we prove a uniform bound for the distances to $k$-th neighbours. Let $\P \in \mM, \ n \geq 2d, \ X_1, \dots, X_n$ be $n$ i.i.d.\ random variables with distribution $\P.$ Denote by $R_k(x)$ the distance from $x$ to its $k$-th nearest neighbour among $\{X_1, \dots X_n\}.$ Let $\delta, h \in (0, 1).$
\begin{proposition}
    \label{prop.Bernstein.general}
    Let $\delta, h > 0,$ and $K_0 = 2K(2d + 1)(4e)^{2d+1},$ where $K$ is the absolute constant defined in Proposition \ref{prop.covering.number}. The two following assertions hold.
    \begin{enumerate}
        \item [$(i)$] If $\lambda \in (0, 1)$ and the integer $n$ satisfies
            \begin{align}
                \label{eq.Bernstein.nk.condition}
                n &\geq \frac{4d(1 + \lambda/3)}{\lambda^2h}\log\bigg(\Big(\frac{K_0}{\delta}\Big)^{1/(2d)}\frac{2}{(1 - \lambda)h}\bigg),
            \end{align}
            then
            \begin{align*}
                \PP\Bigg\{\inf_{x \in \RR^d}\frac{1}{nh}\sum_{i=1}^n \mathds{1}(X_i \in \B(x, \zeta_h(x))) < \frac{(1 - \lambda)}{2}\Bigg\} \leq \delta.
            \end{align*}
        \item [$(ii)$] If $\lambda \in \RR_+^*$ and the integer $n$ satisfies
            \begin{align}
                \label{eq.Bernstein.nk.lower.condition}
                n &\geq \frac{4d(1 + \lambda/3)}{\lambda^2h}\log\bigg(\Big(\frac{K_0}{\delta}\Big)^{1/(2d)}\frac{1}{(1 + \lambda)h}\bigg),
            \end{align}
            then,
            \begin{align*}
                \PP\Bigg\{\sup_{x \in \RR^d}\frac{1}{nh}\sum_{i=1}^n \mathds{1}(X_i \in \B(x, \zeta_h(x))) > 2(1 + \lambda)\Bigg\} \leq \delta.
            \end{align*}
    \end{enumerate}
\end{proposition}

\begin{proof}
    Let $z^m := (z_1, \dots, z_m) \in \RR^{d\times m}$ be an arbitrary vector of $m$ different points, denote $S(x, h) := \B(x, \zeta_h(x)),$ and define the set
    \begin{align*}
        \mI(h) := \bigg\{\mathds{1}(\cdot \in S(x, h))\bigg\}.
    \end{align*}
    Consider the event
    \begin{align*}
        A_n(t, z^m) := \bigcap_{j=1}^m\bigg\{\Big|\frac 1n \sum_{i=1}^n \mathds{1}\big(X_i \in S(z_j, h)\big) - h \Big|\leq t\bigg\}.
    \end{align*}
    Let $X \sim \P.$ For all $x \in \RR^d,$ we have
    \begin{enumerate}
        \item [$(i)$] $\E[\mathds{1}(X \in S(x, h))] = \P_{\sX}\{S(x, h)\} = \P_{\sX}\{\B(x, \zeta_h(x))\} = h$
        \item [$(ii)$] $\Var(\mathds{1}(X \in S(x, h))) = h(1 - h) \leq h.$
        \item [$(iii)$] $0 \leq \mathds{1}(X \in S(x, h)) \leq 1.$
    \end{enumerate}
    By the union bound and Bernstein's inequality \eqref{eq.Bernstein}, 
    \begin{align}
        \label{eq.Bernstein.union.bound}
        \PP\big\{A_n(t, z^m)^c\big\} &\leq \sum_{j=1}^m \PP\bigg\{\Big|\frac 1n \sum_{i=1}^n \mathds{1}\big(X_i \in S(z_j, h)\big) - h \Big|\geq t\bigg\} \leq 2m\exp\bigg(-\frac{nt^2}{2h + 2t/3}\bigg).
    \end{align}
    Let $\P_n := n^{-1}\sum_{i=1}^n \delta_{X_i}$ be the empirical distribution of the sample $X_1, \dots, X_n,$ and let $\eps \in (0, 1).$ Applying Proposition \ref{prop.covering.number} ensures that there exists an $\eps$-net over $(\mB, \|\cdot\|_{L_1(\P_n)})$ with finite cardinality $M = M(X^n),$ with $\|M\|_\infty \leq K_0\eps^{-2d}.$ As $\mI(h) \subseteq \mB,$ there exists an $\eps$-net over $(\mI(h), \|\cdot\|_{L_1(\P_n)})$ with finite (random) cardinality $\ol m = \ol m(X^n)$ upper bounded by $\|M\|_{\infty}.$  Let $Z_1, \dots Z_{\ol m}$ be such that this $\eps$-net is the set $\big\{S(Z_j, h), 1\leq j\leq \ol m\big\}.$ Since the upper bound on $m$ does not depend on the sample, the upper bound \eqref{eq.Bernstein.union.bound} is also valid when $z^m$ is replaced by $Z^{\ol m}.$ Hence, using the covering bound of Proposition \ref{prop.covering.number}, yields
    \begin{align*}
        \PP\big\{A_n(t, Z^{\ol m})^c\big\}&\leq K_0\eps^{-2d}\exp\bigg(-\frac{nt^2}{2h + 2t/3}\bigg).
    \end{align*}
    We now work on $A_n(t, Z^{\ol m}).$ The $\eps$-net property of the $Z_j$'s ensures that for all $x \in \RR^d,$ there exists $j \in [\ol m],$ such that
    \begin{align*}
        \big\|\mathds{1}(\cdot \in S(x, h)) - \mathds{1}(\cdot \in S(Z_j, h))\big\|_{L_1(\P_n)} \leq \eps,
    \end{align*}
    and, on $A_n(t, Z^{\ol m}),$ it follows that
    \begin{align*}
        \Big|\frac 1n \sum_{i=1}^n \mathds{1}(X_i \in S(x, h)) - h\Big| \leq \big\|\mathds{1}&(\cdot \in S(x, h)) - \mathds{1}(\cdot \in S(Z_j, h))\big\|_{L_1(\P_n)}\\
        &+ \Big|\frac 1n \sum_{i=1}^n \mathds{1}(X_i \in S(Z_j, h)) - h\Big|,
    \end{align*}
    where the first term is upper bounded by $\eps,$ while the definition of $A_n(t, Z^{\ol m})$ ensures that the second term is upper bounded by $t.$ The point $x \in \RR^d$ has been taken arbitrarily in $\RR^d,$ hence,
    \begin{align}
        \label{eq.two.sided.uniform.Bernstein}
        \sup_{x \in \RR^d}\Big|\frac 1n \sum_{i=1}^n \mathds{1}(X_i \in S(x, h)) - h\Big| \leq \eps + t.
    \end{align}
    We now proceed to prove the first claim. In particular, from \eqref{eq.two.sided.uniform.Bernstein}, it holds that
    \begin{align}
        \label{eq.0000}
        \inf_{x \in \RR^d}\frac 1n \sum_{i=1}^n \mathds{1}(X_i \in S(x, h)) \geq h -
        \eps - t.
    \end{align}
    Let $\delta, \lambda \in (0, 1).$ We pick $t = \lambda h,$ and
    \begin{align*}
        \eps = \delta^{-1/(2d)}K_0^{1/(2d)}\exp\bigg(-\frac{n\lambda^2h}{4d(1 + \lambda/3)}\bigg).
    \end{align*}
    One can check that for all $n$ satisfying
    \begin{align*}
        n \geq \frac{4d(1 + \lambda/3)}{\lambda^2h}\log\bigg(\frac{2}{(1 - \lambda)h}\Big(\frac{K_0}{\delta}\Big)^{1/(2d)}\bigg),
    \end{align*}
    it holds that $\eps \leq (1 - \lambda)h/2.$ Finally, plugging these values in \eqref{eq.0000} shows that, on $A_n(t, Z^{\ol m}),$
    \begin{align*}
        \inf_{x \in \RR^d}\frac 1{n} \sum_{i=1}^n \mathds{1}(X_i \in S(x, h)) \geq \frac{(1 - \lambda)h}2.
    \end{align*}
    This proves that
    \begin{align*}
        \PP\Bigg\{\inf_{x \in \RR^d}\frac 1{nh} \sum_{i=1}^n \mathds{1}(X_i \in S(x, h)) < \frac{(1 - \lambda)}2\Bigg\} < \PP\big\{A_n(t, Z^{\ol m})^c\big\},
    \end{align*}
    and it can be checked that $\PP\big\{A_n(t, Z^{\ol m})^c\big\} \leq \delta,$ which finishes the proof of the first claim. We now prove the second claim, which is a slight variation of the proof of the first one. A consequence of \eqref{eq.two.sided.uniform.Bernstein} is that
    \begin{align*}
        \sup_{x \in \RR^d}\frac 1n \sum_{i=1}^n \mathds{1}(X_i \in S(x, h)) \leq h + \eps + t.
    \end{align*}
    We then pick $t = \lambda h$ with $\lambda \in (0, \infty)$ and $\eps$ as previously. In particular, it holds that for all 
    \begin{align*}
        n \geq \frac{4d(1 + \lambda/3)}{\lambda^2h}\log\Bigg(\frac{1}{(1 + \lambda) h}\bigg(\frac{K_0}{\delta}\bigg)^{1/(2d)}\Bigg),
    \end{align*}
    we have $\eps \leq (1 + \lambda)h.$ Hence, for such $n$,
    \begin{align*}
        \sup_{x \in \RR^d}\frac 1n \sum_{i=1}^n \mathds{1}(X_i \in S(x, h)) \leq 2(1 + \lambda)h,
    \end{align*}
    which proves that
    \begin{align*}
        \PP\Bigg\{\sup_{x \in \RR^d}\frac 1{nh} \sum_{i=1}^n \mathds{1}(X_i \in S(x, h)) < 2(1 + \lambda)\Bigg\} < \PP\big\{A_n(t, Z^{\ol m})^c\big\} \leq \delta.
    \end{align*}
    This finishes the proof.
\end{proof}

One can check that both the conclusions of Proposition \ref{prop.Bernstein.general} are true whenever \eqref{eq.Bernstein.nk.condition} is satisfied. From now on, for technical reasons, we assume that the constant $K$ in Proposition \ref{prop.covering.number} is greater than $1.$ This assumption only potentially worsens the covering number bound. We now consider an integer $k \in \{1, \dots, n\}.$

\begin{proposition}
    \label{prop.neighbours.distance.bound.zeta}
    Let $\tau > 0.$ If $k \geq \lceil \log n\rceil, \ n \geq 2d,$ and $h_+ = h_+(k) := 2k/((1 - \lambda^+_\tau)n)$ with
    \begin{align*}
        \lambda_\tau^+ := \frac{4\big(2d + \tau + \log(K_0)/\log(2d))\big) + 9}{4\big(2d + \tau + \log(K_0)/\log(2d))\big) + 12},
    \end{align*} 
    then,
    \begin{align*}
        \PP\Bigg\{\sup_{x \in \RR^d}\frac{R_k(x)}{\zeta_{h_+}(x)} \geq 1\Bigg\} \leq n^{-\tau}.
    \end{align*}
\end{proposition}
We further define 
\begin{align}
    \label{eq.def.btau}
    B_\tau := \frac{2}{1 - \lambda_\tau^+} = \frac{8(2d + 3 + \log(K_0)/\log(2d) + \tau)}{3},
\end{align}
so that $h_+ = B_\tau k/n.$

\begin{proof}
    Applying Proposition \ref{prop.Bernstein.general} $(i)$ with $\delta = n^{-\tau},$ and $h = 2k/((1 - \lambda)n)$ with arbitrary $\lambda \in (0, 1),$ yields that for all $n$ satisfying
    \begin{align}
        \label{eq.n.condition.nearest.neighbours.bound.zeta}
        n \geq \frac{2d(1 + \lambda/3)(1 - \lambda)n}{\lambda^2k}\log\Big(K_0^{1/(2d)}n^{1 + \tau/(2d)}k^{-1}\Big),
    \end{align}
    it holds that
    \begin{align}
        \label{eq.bound.zeta.neighbour.hp}
        \PP\Bigg\{\inf_{x \in \RR^d}\frac{1}{n}\sum_{i=1}^n \mathds{1}(X_i \in \B(x, \zeta_{h_+}(x))) <\frac kn\Bigg\} \leq n^{-\tau},
    \end{align}
    where $K_0$ is defined in Proposition \ref{prop.Bernstein.general}. The high-probability bound \eqref{eq.bound.zeta.neighbour.hp} implies that
    \begin{align*}
        \PP\Bigg\{\sup_{x \in \RR^d}\frac{R_k(x)}{\zeta_{h_+}(x)} \geq 1\Bigg\} \leq n^{-\tau}.
    \end{align*}
    We now check that the right-hand side of \eqref{eq.n.condition.nearest.neighbours.bound.zeta} is always smaller than $2d$. Inequality \eqref{eq.n.condition.nearest.neighbours.bound.zeta} is equivalent to
    \begin{align}
        \label{eq.1111}
        \frac{(1 + \lambda/3)(1-\lambda)}{\lambda^2k}\log\Big(K_0n^{2d + \tau}k^{-2d}\Big) \leq 1.
    \end{align}
    For $n \geq k \geq \log n,$ using the expression of $K_0,$ the assumption that $K \geq 1,$ and $d\geq 1,$ we can check that for $n^\tau \geq (384e^2)^{-1},$ it holds that $\log(K_0n^{2d + \tau}k^{-2d})\geq 0.$ The latter then holds for any $\tau > 0$ and any $n \geq 1.$ Hence, we can assert that for \eqref{eq.1111} to hold, it is sufficient to have
    \begin{align*}
        \frac{(1 + \lambda /3)(1 - \lambda)}{\lambda^2}\log\Big(K_0n^{2d + \tau}\Big) \leq \log n.
    \end{align*}
    Using $a\log(b) = \log(b^a),$ exponentiating on both sides and rearranging the terms leads to 
    \begin{align*}
        1 \leq n^{\tfrac{\lambda^2}{(1 + \lambda/3)(1 - \lambda)} - 2d - \tau - \log(K_0)/\log(n)}.
    \end{align*}
    For the latter to hold, it is sufficient that the exponent be nonnegative. Since $n \geq 2d,$ a lower bound for the exponent in the above display is $\lambda^2/((1 + \lambda/3)(1 - \lambda)) - 2d - \tau - \log(K_0)/\log(2d).$ Rearranging the expression of the derived lower bound gives the sufficient condition under the form of a quadratic inequation
    \begin{align*}
        \lambda^2\Big(1 + \frac{a}{3}\Big) + \frac{2a}{3}\lambda - a.
    \end{align*}
    where $a:= 2d + \tau + \log(K_0)/\log(2d).$ It is straightforward to find that the positive root $z_+$ satisfies
    \begin{align*}
        z_+ = \frac{a}{a + 3}\bigg(\sqrt{4 + \frac 9a} - 1\bigg) \leq \frac{4a + 9}{4a + 12} = \lambda_\tau^+ < 1,
    \end{align*}
    where the inequality follows Bernoulli's inequality. This shows that with our choice of $\lambda_\tau^+,$ \eqref{eq.n.condition.nearest.neighbours.bound.zeta} holds for all $n \geq 2d.$ Hence, for all $k \geq \log n,$
    \begin{align*}
        \PP\Bigg\{\sup_{x \in \RR^d}\frac{R_k(x)}{\zeta_{h_+}(x)} \geq 1\Bigg\} \leq n^{-\tau}.
    \end{align*}
\end{proof}

Let $\P \in \mP.$ Let $X \sim \P,$ and  $Y = f(X) + \eps,$ where $\eps$ is an $\RR$-valued random variable that is independent from $X.$ Assume as well that $X$ is sub-exponential with parameters $\alpha, \nu > 0.$ Let $(X_1, Y_1), \dots, (X_n, Y_n)$ be $n$ i.i.d.\ copies of $(X, Y).$ For $k \in \NN$ and $x \in \RR^d,$ we denote by $R_k^{\P}(x) := \|x - X_k(x)\|.$ We define the events
\begin{align*}
    U_n^{\P}(k) := \Bigg\{\sup_{x \in \RR^d} \frac{R_k^{\P}(x)}{\zeta^{\P_{\sX}}_{h_+(k)}(x)} \leq 1\Bigg\}, \text{ and }U_n^{\P} := \bigcap_{k = \lceil \log n\rceil}^n U^{\P}_n(k).
\end{align*}

\begin{remark}
    \label{rem.bound.variable.k}
    While Proposition \ref{prop.neighbours.distance.bound.zeta} shows that $R_k(x) \leq \zeta_{h_+}(x)$ uniformly with high probability for each individual choice of $k \geq \log n,$ it can easily be extended to the case where $k = k(x) \in \{1, \dots, n\}$ depends on $x \in \RR^d.$ The fact that $k$ can now change values depending on $x$ demands one to pay the price that the $\sup$ in the probability is not on the whole $\RR^d$ anymore, but only on a subset $\Theta_n^*(k)$ defined as follows
    \begin{align*}
        \Theta_n(k) := \bigg\{x \in \RR^d: k(x) \geq \lceil\log n \rceil\bigg\}.
    \end{align*}
    Similarly, we have to pay the price of a factor $n$ in the probability, as can be seen from the union bound
    \begin{align*}
        \PP\Bigg\{\sup_{x \in \Theta_n^*}\frac{R_{k(x)}(x)}{\zeta_{h_+}} \geq 1\Bigg\} &\leq \PP\Bigg\{\bigcup_{\ell = \lceil \log n \rceil}^n\Bigg\{\sup_{x \in \RR^d}\frac{R_{\ell}(x)}{\zeta_{h_+}(x)} \geq 1\Bigg\} \Bigg\}\\
        &\leq \sum_{\ell = \lceil \log n \rceil}^n\PP\Bigg\{\sup_{x \in \RR^d}\frac{R_{\ell}(x)}{\zeta_{h_+}(x)}\geq 1\Bigg\}\\
        &\leq n^{1 - \tau}.
    \end{align*}
\end{remark}
\begin{proposition}
    \label{prop.neighbours.distance.lower.bound.zeta}
    Let $\tau > 0$ and $V_\tau := K_0^{1/(2d)}\vee(8d + 4\tau).$ If $k \geq \lceil V_\tau\log n\rceil, \ n\geq 2d,$ and $h_- = h_-(k) := k/(2n(1 + \lambda_\tau^-)),$ with 
    \begin{align*}
        \lambda_\tau^- := 1 + \sqrt{\frac 52},
    \end{align*}
    then,
    \begin{align*}
        \PP\Bigg\{\inf_{x \in \RR^d}\frac{R_k(x)}{\zeta_{h_-}(x)} \leq 1\Bigg\} \leq n^{-\tau}.
    \end{align*}
\end{proposition}
We further define
\begin{align}
    \label{eq.def.dtau}
    c_1 := \frac{1}{2 + \sqrt{10}},
\end{align}
so that $h_- = c_1 k/n.$

\begin{proof}
    Applying Proposition \ref{prop.Bernstein.general} $(ii)$ with $\delta = n^{-\tau}$ and $h = k/(2(1 + \lambda)n)$ with arbitrary $\lambda \in \RR_+^*,$ yields that for all $n$ satisfying
    \begin{align}
        \label{eq.n.condition.nearest.neighbours.bound.zeta.lower}
        n \geq \frac{8d(1 + \lambda/3)(1 + \lambda)n}{\lambda^2k}\log\Big(K_0^{1/(2d)}n^{1 + \tau/(2d)}k^{-1}\Big),
    \end{align}
    it holds that
    \begin{align}
        \label{eq.bound.zeta.neighbour.lower.hp}
        \PP\Bigg\{\sup_{x \in \RR^d}\frac{1}{n}\sum_{i=1}^n \mathds{1}(X_i \in \B(x, \zeta_{h_-}(x))) \geq \frac kn\Bigg\} \leq n^{-\tau},
    \end{align}
    where $K_0$ is defined in Proposition \ref{prop.Bernstein.general}. The high-probability bound \eqref{eq.bound.zeta.neighbour.lower.hp} implies that
    \begin{align*}
        \PP\Bigg\{\inf_{x \in \RR^d}\frac{R_k(x)}{\zeta_{h_-}(x)} \geq 1\Bigg\} \leq n^{-\tau}.
    \end{align*}
    Taking $\lambda_\tau^- = 1 + \sqrt{5}/2$ in \eqref{eq.n.condition.nearest.neighbours.bound.zeta.lower} leads to
    \begin{align*}
        k \geq 4\log\big(K_0n^{2d + \tau}k^{-2d}\big),
    \end{align*}
    and it is easy to check that this inequality is satisfied for all $n \geq 2d$ and $k \geq \lceil V_\tau \log n\rceil.$
\end{proof}
We define the events
\begin{align*}
    L_n^{\P}(k) := \Bigg\{\inf_{x \in \RR^d} \frac{R_k^{\P}(x)}{\zeta^{\P}_{h_+(k)}(x)}\geq 1\Bigg\},\text{ and }L_n^{\P} &:= \bigcap_{k= \lceil V_\tau\log n\rceil}^n L_n^{\P}(k).
\end{align*}
We also define the set
\begin{align*}
    \Delta_n^{\P}(c, k) := \bigg\{x \in \RR^d: p(x) \geq \frac{ck}{n}\bigg\}.
\end{align*}

\begin{lemma}
    \label{lem.bound.neighbours.density}
    Let $\tau > 0$ and $p$ be the density of $\P.$ If $k \geq \lceil\log n\rceil,$ then, for all $n \geq 2d,$
    \begin{align*}
         U_n^{\P}(k) \subseteq \Bigg\{\sup_{x \in \Delta_n^{\P}(S_{\tau}, k)}p(x)R_k(x)^d \leq \frac{S_\tau k}{n}\Bigg\},
    \end{align*}
    with $S_{\tau} := \theta B_\tau.$
\end{lemma}

\begin{proof}
    Since $\P \in \mP,$ we can apply \eqref{eq.minimal.maximal.mass} to obtain that for all $x \in \RR^d,$
    \begin{align*}
        h_+ = \P\{\B(x, \zeta_{h_+}(x))\} \geq \theta^{-1}\big(\zeta_{h_+}(x)\wedge 1\big)^dp(x).
    \end{align*}
    By rearranging the previously displayed equation, we obtain
    \begin{align*}
        \zeta_{h_+}(x)\wedge 1 \leq \bigg(\frac{\theta h_+}{p(x)}\bigg)^{1/d}.
    \end{align*}
    We observe that if $x \in \Delta_n^{\P}(S_\tau, k),$ then $\P\{\B(x, 1)\} \geq \theta^{-1}p(x) \geq B_\tau k/n = h_+.$ Hence, $\zeta_{h_+}(x) \leq 1$ and for all $x \in \Delta_n^{\P}(S_\tau, k),$ on $U_n^{\P}(k),$
    \begin{align*}
        R_k(x)^d \leq \frac{S_\tau k}{np(x)},
    \end{align*}
    which implies the result.
\end{proof}

\begin{remark}
    \label{rem.bound.variable.k.density}
    Given a function $k \colon \RR^d \to \{1, \dots, n\}$ and a constant $c>0,$ we define the set
    \begin{align*}
        \Theta_n^{\P}(c, k) := \Bigg\{x \in \RR^d: p(x)\geq \frac{ck(x)}{n}\Bigg\}.
    \end{align*}
    As a direct implication of Lemma \ref{lem.bound.neighbours.density}, we obtain,
    \begin{align*}
        \Bigg\{\sup_{x \in \Theta_n^{\P}(c, k)}\frac{p(x)R_{k(x)}(x)^d}{k(x)} \geq \frac{S_\tau}{n}\Bigg\} &= \bigcup_{i=\lceil\log n\rceil}^n\Bigg\{\sup_{x \in k^{-1}(i)\cap \Delta_n^{\P}(c, i)}p(x)R_i(x)^d \geq \frac{S_\tau i}{n}\Bigg\}\\
        &\subseteq \bigcup_{i=\lceil \log n \rceil}^n\Bigg\{\sup_{x \in \Delta_n^{\P}(c, i)} p(x)R_i(x)^d \geq \frac{S_\tau i}{n}\Bigg\}\\
        &= \bigg(\bigcap_{i=\lceil \log n\rceil}^n U_n^{\P}(i)\bigg)^c = U_n^{\P c}.
    \end{align*}
    Hence,
    \begin{align*}
        \PP\Bigg\{\sup_{x \in \Theta_n^{\P}(c, k)}\frac{p(x)R_{k(x)}(x)^d}{k(x)} \leq \frac{S_\tau}{n}\Bigg\} \geq 1 - n^{1-\tau}.
    \end{align*}
\end{remark}

\begin{lemma}
    \label{lem.lower.bound.neighbours.density}
    Let $\tau > 0.$ If $k \geq \lceil V_\tau\log n\rceil,$ and $n \geq 2d,$ then
    \begin{align*}
        \Bigg\{\inf_{x \in \Delta_n^{\P}(c_2, k)}p(x)R_k(x)^d \leq \frac{c_1 k}{\theta n}\Bigg\} \subseteq L_n^{\P}(k)^c,
    \end{align*}
    where $c_2 := \theta c_1.$
\end{lemma}

\begin{proof}
    We work on the event $L_n^{\P}(k).$ For all $x \in \supp(\P)$ such that $\zeta_{h_-}(x) \leq 1,$ the local mass property \eqref{eq.minimal.maximal.mass} implies that
    \begin{align*}
        h_- = \P\{\B(x, \zeta_{h_-}(x))\} \leq \theta p(x)\zeta_{h_-}(x)^d.
    \end{align*}
    The latter, once rearranged, yields
    \begin{align*}
        \zeta_{h_-}(x) \geq \bigg(\frac{c_1 k}{\theta np(x)}\bigg)^{1/d}.
    \end{align*}
    It remains to prove that if $x \in \Delta_n^{\P}(c_2, k),$ then $\zeta_{h_-}(x) \leq 1.$ Let $x \in \Delta_n^{\P}(c_2, k)$ and assume that $\zeta_{h_-}(x) > 1.$ The local mass property \eqref{eq.minimal.maximal.mass} implies that
    \begin{align*}
        h_- = \P\{\B(x, \zeta_{h_-}(x))\} \overset{(i)}{>} \P\{\B(x, 1)\} \geq \theta^{-1}p(x) \geq \frac{c_1 k}{n} = h_-,
    \end{align*}
    which is a contradiction. Hence, for all $x \in \Delta_n^{\P}(c_2, k),$ it holds that $\zeta_{h_-}(x) \leq 1,$ and a fortiori, that
    \begin{align*}
        L_n^{\P}(k) \subseteq \Bigg\{\inf_{x \in \Delta_n^{\P}(c_2, k)} p(x) R_k(x)^d \geq \frac{c_1 k}{\theta n}\Bigg\}.
    \end{align*}
\end{proof}

\subsection{Uniform pointwise absolute bias and standard deviation bounds}
We now work within Model \eqref{eq.model}, where $f_* \in \mH(L, \beta)$ with $L >0$ and $\beta \in (0, 1].$ Let $\wh f$ be a local $k$-NN regression estimator with neighbour function $k \colon \RR^d \to \{1, \dots, n\}$ independent from $Y_i|X_i$ for $1 \leq i \leq n.$ That is, for all $x \in \RR^d,$
\begin{align}
    \label{eq.local.knn}
    \wh f(x) = \frac{1}{k(x)}\sum_{i=1}^n Y_i(x),
\end{align}
where $Y_i(x), \ 1\leq i \leq k(x)$ is the label associated to the $k$-th nearest neighbour of $x$ among $\{X_1, \dots, X_n\}.$ We define the associated pointwise average function 
\begin{align}
    \label{eq.average.knn}
    \ol f\colon x \mapsto \E\big[\wh f(x)|X_1, \dots, X_n\big].
\end{align}
The previous sub-section leads to the following bound on the pointwise bias.
\begin{lemma}
    \label{lem.bias.term}
    If $n \geq \|k(x)\|_{\infty},$ then, for all $x \in \RR^d,$
        \begin{align*}
            |\ol f(x) - f_*(x)| \leq 2L \wedge L \sum_{i=1}^{k(x)} k(x)^{-1} R_i(x)^\beta,
        \end{align*}
        Additionally, if $\P \in \mP,$ and $k(x) \geq \lceil \log n\rceil$ for all $x \in \RR^d,$ then
        \begin{align*}
            \bigcap_{i\in k(\RR^d)}U_n^{\P}(i) \subseteq \Bigg\{\sup_{x \in \Theta_n^{\P}(S_\tau, k)}\frac{p(x)^{\beta/d}|\ol f(x) - f_*(x)|}{k(x)^{\beta/d}} \leq \bigg(\frac{S_\tau}{ n}\bigg)^{\beta/d}\Bigg\}.
        \end{align*}
\end{lemma}
\begin{proof}
    Since $f_* \in \mH(L, \beta),$ we get for all $x \in \RR^d,$ and all $i \in \{1, \dots, k(x)\},$
    \begin{align*}
        \big|f(X_i(x)) - f_*(x)\big| &\leq \big(f(X_i(x)) + f_*(x)\big) \wedge L \|X_i(x) - x\|^{\beta}
        \nonumber\\
        &\leq 2L \wedge L R_i(x)^\beta.
    \end{align*}
    This together with the definition of $\ol f(x)$ and $\sum_{i=1}^{k(x)} k(x)^{-1} = 1$ yields 
    \begin{align*}
        \big|\ol f(x) - f_*(x) \big|
        &= \bigg|\sum_{i=1}^{k(x)} k(x)^{-1} f(X_i(x)) - f_*(x) \bigg|\\
        &\leq \sum_{i=1}^{k(x)} k(x)^{-1} \big|f_*(X_i(x)) - f_*(x) \big|\\
        &\leq \sum_{i=1}^{k(x)} k(x)^{-1}\big(2F \wedge L R_i(x)^\beta\big)\\
        &\leq 2L \wedge L \sum_{i=1}^{k(x)} k(x)^{-1} R_i(x)^\beta.
    \end{align*}
    This proves the first statement. The second statement is obtained by further upper bounding the previously derived inequality by $|\ol f(x) - f_*(x)| \leq LR_k(x)^\beta$ and applying Lemma \ref{lem.bound.neighbours.density}.
\end{proof}

\begin{proposition}\label{prop.variance.bound.general.case}
Let $\wh f$ be as in \eqref{eq.local.knn}, with neighbour function $k \colon \RR^d \to \{1, \dots, n\}.$ Let $\ol f$ be defined by \eqref{eq.average.knn}.
Then, for all $n\geq 2d,$ it holds that
\begin{align*}
    \PP\Bigg\{\sup_{x \in \RR^d}k(x)^{1/2}\big|\wh f(x) - \ol f(x)\big| \geq \eta\Bigg\} &\leq K_dn^{2d + 1}e^{-H \eta^2},
\end{align*}
where $K_d = 2(25/d)^d,$ and $H = \lambda_0^2/(8\sigma_0)$ for arbitrary $\lambda_0, \sigma_0 > 0$ such that $\E[\exp(\lambda|\eps|)] \leq \sigma_0.$
\end{proposition}

\begin{proof}
For any $x \in \RR^d$, we have
\begin{align*}
\wh f(x)-\ol f(x)=\sum_{i=1}^nk(x)^{-1}[Y_i(x)-f(X_i(x))]\mathds{1}(i \leq k(x)).
\end{align*}
Let $Z_i(x) := Y_i(x) - f(X_i(x))$ for all $1 \leq i \leq n$ and all $x \in \RR^d.$ By Proposition 8.1 in \cite{Biau2015Lectures} and the regression model \eqref{eq.model}, conditionally on $X_1,\ldots,X_n$, the random variables $\{Z_i(x)\}_{i=1}^n$ are mutually independent with zero mean. By assumption, they are sub-exponential, and therefore, they satisfy the uniform noise condition; there exists constants $\lambda_0, \sigma_0 > 0$ such that
\begin{align}\label{equ::uniformnoise}
\sup_{x^n\in \RR^{d\times n}} 
\E \Big[ \exp\big(\lambda_0|Z_i(x)|\big) \big| X^n = x^n \Big]
\leq \sigma_0 
< \infty,
\end{align}
where $x^n \in \RR^{d\times n}.$ 
By Lemma 12.1 in \cite{Biau2015Lectures}, 
for any $x \in \RR^d,$ and any $\eta \leq \min(1, 2\sigma_0) / \lambda_0$, it holds that
\begin{align}
    \label{eq.Chernoff.bound.knn}
    \PP \Big\{\big|\wh f(x) - \ol f(x)\big| \geq \eta \Big| X^n = x^n \Big\}
    \leq 2 \exp \bigg( - \frac{\eta^2 \lambda_0^2k(x)}{8 \sigma_0} \bigg).
\end{align}
To derive a high-probability bound that is uniform over $x \in \RR^d,$ we introduce the following set of permutations
\begin{align*}
    \mW(x^n) := \bigg\{\sigma \in \mS_n: \exists x \in \RR^d,\ X_i(x) = X_{\sigma(i)},\ \forall i \in \{1, \dots, n\}\bigg\},
\end{align*}
where $\mS_n$ is the symmetric group over $\{1, \dots, n\}.$ Theorem 12.2 in \cite{Biau2015Lectures} states that for all $n \geq 2d,$ and all $x^n \in \RR^{d\times n},$ the cardinality of $\mW(x^n)$ is bounded by
\begin{align*}
    \ell(x^n) := \card(\mW(x^n)) \leq \Big(\frac{25}{d}\Big)^dn^{2d}.
\end{align*}
In particular, given $x^n \in \RR^{d\times n},$ we can define the equivalence relationship $\sim_{x^n}$ on $\RR^d$ as, for all $z_1, z_2 \in \RR^d,$
\begin{align*}
    \big(z_1 \sim_{x^n} z_2\big) \Longleftrightarrow \Big(\forall i \in \{1, \dots, n\}: x_i(z_1) = x_i(z_2)\Big),
\end{align*}
where $x_i(z)$ is the $i$-th nearest neighbour of $z$ among $\{x_1, \dots, x_n\}.$ Further, $\RR^d/\sim_{x^n}$ has cardinality at most $\card(\mW(x^n)),$ and is in bijection with $\mW(x^n).$ Let then $\{A_\sigma\}_{\sigma \in \mW(x^n)}$ be the elements of $\RR^d/\sim_{x^n}.$ We can now derive a high probability uniform upper bound by partitioning $\RR^d$ according to $\sim_{x^n},$ using the union bound, and applying \eqref{eq.Chernoff.bound.knn} as follows 
\begin{align*}
    \PP\Bigg\{\sup_{x \in \RR^d}&k(x)^{1/2}\big|\wh f(x) - \ol f(x)\big| \geq \eta\Big| X^n = x^n\Bigg\} \\
    &\leq \PP\Bigg\{\bigcup_{\sigma \in \mW(X^n)}\Bigg\{\sup_{x \in A_\sigma}\Big|\sum_{i=1}^n \frac{k(x)^{1/2}}{k(x)}\big(Y_{\sigma(i)} - f_*(X_{\sigma(i)})\big)\Big| \geq \eta \bigg| X^n = x^n\Bigg\}\Bigg\}\\
    &\leq \sum_{\sigma \in \mW(x^n)}\PP\Bigg\{\sup_{\ell \in \{1, \dots, n\}}\Big|\sum_{i=1}^\ell \ell^{-1/2}\big(Y_{\sigma(i)} - f_*(X_{\sigma(i)})\big)\Big| \geq \eta \bigg| X^n = x^n\Bigg\}\\
    &\leq \sum_{\sigma \in \mW(x^n)}\sum_{\ell = 1}^n\PP\Bigg\{\Big|\sum_{i=1}^\ell \ell^{-1/2}\big(Y_{\sigma(i)} - f_*(X_{\sigma(i)})\big)\Big| \geq \eta \bigg| X^n = x^n\Bigg\}\\
    &\leq \sum_{\sigma \in \mW(x^n)}\sum_{\ell = 1}^n 2\exp\bigg(-\frac{\eta^2\lambda_0^2}{8c_0}\bigg)\\
    &\leq 2\Big(\frac{25}d\Big)^dn^{2d+1} \exp\bigg(-\frac{\eta^2\lambda_0^2}{8c_0}\bigg).
\end{align*}
We finish the proof by taking the expectation with regards to $X_1, \dots, X_n$ on both sides.
\end{proof}

\begin{corollary}
    \label{cor.variance.bound.nice.version}
    Under the conditions of Proposition \ref{prop.variance.bound.general.case}, we have for all $\tau > 0,$ all $n \geq 3 \vee 2d,$
    \begin{align*}
        \PP\Bigg\{\sup_{x \in \RR^d}k(x)^{1/2}\big|\wh f(x) - \ol f(x)\big|\geq \sqrt{C_\tau \log n}\Bigg\} \leq n^{-\tau},
    \end{align*}
    with 
    \begin{align*}
        C_\tau := \frac{\log K_d}{H}\bigg(1 + \frac{2d + \tau + 1 }{\log(K_d))}\bigg)
    \end{align*} where $H$ and $K_d$ are defined in Proposition \ref{prop.variance.bound.general.case}.
\end{corollary}

\begin{proof}
    Once rearranged, the conclusion of Proposition \ref{prop.variance.bound.general.case} reads
    \begin{align*}
        \PP\Bigg\{\forall x\in\RR^d, \ \big|\wh f(x) - \ol f(x)\big| \leq \sqrt{\frac{1}{Hk(x)}\log\Big(\frac{K_dn^{2d + 1}}{\delta}\Big)}\Bigg\} \geq 1 - \delta.
    \end{align*}
    We now apply this result with 
    \begin{align*}
        \eta_- \geq \sqrt{\frac{\log K_d}H\bigg(1 + \frac{2d + \tau + 1}{\log K_d}\log n\bigg)}
    \end{align*}
    to obtain
    \begin{align*}
        \PP\Bigg\{\sup_{x \in \RR^d}k(x)^{1/2}|\wh f(x) - \ol f(x)| \geq \eta_-\Bigg\} \leq n^{-\tau}.
    \end{align*}
    Finally, since $n\geq 3, \ \log n \geq 1$ and we can pick
    \begin{align*}
        \eta = \sqrt{C_\tau\log n} = \sqrt{\frac{\log K_d}{H}\bigg(1 + \frac{2d + \tau + 1}{\log K_d}\bigg)\log n} \geq \eta_-,
    \end{align*}
    and obtain the claimed result.
\end{proof}
We define the event
\begin{align*}
    V_n^{\P}(\wh f) := \Bigg\{\sup_{x \in \RR^d}k(x)^{1/2}\big|\wh f(x) - \ol f(x)\big| \leq \sqrt{C_\tau\log n}\Bigg\}.
\end{align*}

\begin{proposition}[Bias-Variance decomposition]
    \label{prop.decomp.knn}
    Under the assumptions of model \eqref{eq.model}, if $\wh f$ is a local $k$-NN with neighbour function $k\colon \RR^d \to \{\lceil \log n\rceil, \dots, n\},$ and if $\P \in \mP,$ then, for all $n \geq 2d \vee 3$ and all $\Delta \subseteq \Theta_n^{\P}(S_\tau, k),$
    \begin{multline*}
        \Bigg\{\forall x \in \RR^d,\ \big(\wh f(x) - f(x)\big)^2 \leq 2C_\tau\frac{\log n}{k(x)} + 8L^2\bigg(\frac{S_\tau k(x)}{np(x)}\bigg)^{2\beta/d}\mathds{1}(x \in \Delta) + 8L^2\mathds{1}(x \in \Delta^c)\Bigg\}\\
        \supseteq \bigg(\bigcap_{i \in k(\RR^d)}U_n^{\P}(i)\bigg)\bigcap V_n^{\P}(\wh f).
    \end{multline*}
\end{proposition}

\begin{proof}
    Lemmas \ref{lem.bias.term} and \ref{cor.variance.bound.nice.version} imply that, on the event
    \begin{align*}
        \bigg(\bigcap_{i \in k(\RR^d)} U_n^{\P}(i)\bigg)\bigcap V_n^{\P}(\wh f),
    \end{align*}
    for all $x \in \RR^d,$
    \begin{align*}
        \big(\wh f(x) - f(x)\big)^2 &\leq 2\big|\wh f(x) - \ol f(x)\big|^2 + 2|\ol f(x) - f(x)\big|^2\\
        &\leq \frac{2C_\tau \log n}{k(x)} + 8L^2\bigg(\frac{S_\tau k(x)}{np(x)}\bigg)^{2\beta/d}\mathds{1}(x \in \Delta) + 8L^2\mathds{1}(x \in \Delta^c).
    \end{align*}
    This finishes the proof.
\end{proof}

\subsection{Proofs for density estimator}

In the proofs of our upper bounds, we choose a neighbour function $k(x)$ that depends on an $\ell$-NN density estimator of the corresponding covariates' density. The definition of a $\ell$-nearest neighbour density estimator $\wh p$ is given by \eqref{eq.density.estimator}. We begin with a preliminary result showing that $\wh p \asymp p$ in a high-density region.

\begin{lemma}
    \label{lem.density.ratio}
    Let $\P \in \mP(D, \theta)$ with $D, \theta > 0.$ Denote by $p$ the density of $\P$, and let $\wh p$ be the $\ell$-NN density estimator obtained from $n$ i.i.d.\ samples from $\P.$ Work with $n \geq 3\vee 2d.$
    \begin{enumerate}
        \item [$(i)$] If $\ell \geq \lceil V_\tau\log n\rceil,$ then
        \begin{align*}
            L_n^{\P}(\ell) \subseteq \Bigg\{\sup_{x \in \Delta_n^{\P}(c_2, \ell)} \frac{\wh p(x)}{p(x)} \leq \frac{\theta}{c_1}\Bigg\}.
        \end{align*}
        \item [$(ii)$] If $\ell \geq \lceil \log n \rceil,$ then
        \begin{align*}
            U_n^{\P}(\ell) \subseteq \Bigg\{\inf_{x \in \Delta_n^{\P}(S_\tau, \ell)} \frac{\wh p(x)}{p(x)} \geq \frac{1}{S_\tau}\Bigg\}
        \end{align*}
    \end{enumerate}
\end{lemma}

\begin{proof}
    $(i)$ We work on the event $L_n^{\P}(\ell).$ For all $x \in \Delta_n^{\P}(c_2, \ell)$ we have
    \begin{align*}
        \frac{\wh p(x)}{p(x)} = \frac{\ell}{np(x)R_{\ell}(x)^d} \leq \frac{\theta}{c_1},
    \end{align*}
    where the last inequality comes from Lemma \ref{lem.lower.bound.neighbours.density}. This shows that 
    \begin{align*}
        L_n^{\P}(\ell) \subseteq \Bigg\{\sup_{x \in \Delta_n^{\P}(c_2, \ell)} \frac{\wh p(x)}{p(x)} \leq \frac{\theta}{c_1}\Bigg\}.
    \end{align*}
    $(ii)$ We work on the event $U_n^{\P}(\ell).$ For all $x \in \Delta_n^{\P}(S_\tau, \ell),$ we have
    \begin{align*}
        \frac{\wh p(x)}{p(x)} = \frac{\ell}{np(x)R_\ell(x)^d} \geq \frac{1}{S_\tau},
    \end{align*}
    where the last inequality comes from Lemma \ref{lem.bound.neighbours.density}. This shows that 
    \begin{align*}
        U_n^{\P}(\ell) \subseteq \Bigg\{\inf_{x \in \Delta_n^{\P}(c_2, \ell)} \frac{\wh p(x)}{p(x)} \geq \frac{1}{S_\tau}\Bigg\}.
    \end{align*}
\end{proof}

\begin{remark}
    This shows that, if $\ell \geq \lceil V_\tau\log n\rceil,$ and $n\geq 3\vee 2d,$ then, on $L_n^{\P}(\ell) \cap U_n^{\P}(\ell),$ for all $x \in \Delta_n^{\P}(S_\tau\vee c_2, \ell),$
    \begin{align*}
        \frac{1}{S_\tau} \leq \frac{\wh p(x)}{p(x)}\leq \frac{\theta}{c_1}.
    \end{align*}
\end{remark}
\begin{lemma}
    \label{lem.inclusions}
    Let $\P \in \mP,\ \tau > 0$ and $\k > 0.$ Consider the function
    \begin{align*}
        k(x) = n \wedge \Big\lceil\k\log(n)^{d/(2\beta + d)}\big(n\wh p(x)\big)^{2\beta/(2\beta + d)} \Big\rceil \vee \lceil \log n\rceil, 
    \end{align*}
    where $\wh p$ is a $\ell$-NN density estimator for $\P,$ with $\ell \geq \lceil V_\tau\log n\rceil.$ Let $N = N(\k, \theta, D) \geq 3\vee 2d$ satisfy
    \begin{align*}
        \frac N{\log N} \geq \k^{(2\beta + d)/d}\bigg(\frac{\theta D}{c_1}\bigg)^{2\beta/d},
    \end{align*}
    and define the constant
    \begin{align*}
        C := S_\tau\vee c_2 \vee \frac{S_\tau}{\k^{(2\beta + d)/(2\beta)}V_\tau} \vee \frac{(2S_\tau\k)^{(2\beta + d)/d}}{V_\tau}\bigg(\frac{\theta}{c_1}\bigg)^{2\beta/d}.
    \end{align*}
    Then, the following assertions hold
    \begin{enumerate}
        \item [$(i)$] For all $n \geq N,$ 
            \begin{align*}
                L_n^{\P}(\ell)\cap U_n^{\P}(\ell) \subseteq \Bigg\{\forall x \in \Delta_n^{\P}(C, \ell), \ k(x) = \Big\lceil\k\log(n)^{d/(2\beta + d)}\big(n\wh p(x)\big)^{2\beta/(2\beta + d)}\Big\rceil\Bigg\}.
            \end{align*}
        \item [$(ii)$] For all $n \geq N,$
            \begin{align*}
                 L_n^{\P}(\ell)\cap U_n^{\P}(\ell) \subseteq \Bigg\{\Delta_n^{\P}(C, \ell) \subseteq \Theta_n^{\P}(S_\tau, k)\Bigg\}.
            \end{align*}
    \end{enumerate}
\end{lemma}

\begin{proof}
    Work on $L_n^{\P}(\ell)\cap U_n^{\P}(\ell).$ We obtain $(i)$ by showing that for all $x \in \Delta_n^{\P}(C, \ell),$
    \begin{align*}
        \log n \leq \k\log(n)^{d/(2\beta + d)}\big(n\wh p(x)\big)^{2\beta/(2\beta + d)} \leq n.
    \end{align*}
    Let $x \in \Delta_n^{\P}(C, \ell),$ and $u(x) := \k\log(n)^{d/(2\beta + d)}(n\wh p(x))^{2\beta/(2\beta + d)}.$ We begin by applying Lemma \ref{lem.density.ratio} and obtain, since $\Delta_n^{\P}(C, \ell) \subseteq \Delta_n^{\P}(S_\tau \vee c_2, \ell),$
    \begin{align}
        \label{eq.double.ineq}
        \k\log(n)^{d/(2\beta + d)}\bigg(\frac{np(x)}{S_\tau}\bigg)^{2\beta/(2\beta + d)} \leq u(x) \leq \k\log(n)^{d/(2\beta + d)}\bigg(\frac{\theta n D}{c_1}\bigg)^{2\beta/(2\beta + d)}.
    \end{align}
    It is immediate to check that the leftmost term in \eqref{eq.double.ineq} is lower bounded by $\log(n)$ if
    \begin{align*}
        p(x) \geq \frac{S_\tau}{\k^{(2\beta + d)/(2\beta)}}\frac{\log n}{n},
    \end{align*}
    which is true a fortiori, for all $x \in \Delta_n^{\P}(C, \ell)$ since
    \begin{align*}
        p(x) \geq \frac{C\ell}{n} \geq \frac{S_\tau}{\k^{(2\beta + d)/(2\beta)}V_\tau}\frac{\ell}{n} \geq \frac{S_\tau}{\k^{(2\beta + d)/(2\beta)}}\frac{\log n}{n}.
    \end{align*}
    The rightmost term in \eqref{eq.double.ineq} is upper bounded by $n$ if
    \begin{align*}
        \frac n{\log n} \geq \k^{(2\beta + d)/d}\bigg(\frac{\theta D}{c_1}\bigg)^{2\beta/d},
    \end{align*}
    which is true by assumption. We have proven the two previous facts for arbitrary $x \in \Delta_n^{\P}(C, \ell),$ therefore,
    \begin{align*}
        L_n^{\P}(\ell)\cap U_n^{\P}(\ell) \subseteq \Bigg\{\forall x \in \Delta_n^{\P}(C, \ell), \ k(x) = \Big\lceil\k\log(n)^{d/(2\beta + d)}\big(n\wh p(x)\big)^{2\beta/(2\beta + d)}\Big\rceil\Bigg\}.
    \end{align*}
    We now prove $(ii).$ Let $x \in \Delta_n^{\P}(C, \ell).$ For $n \geq e \vee N, \ (i)$ implies that $k(x) \geq \lceil \log n\rceil \geq 1.$ Hence, 
\begin{align*}
    \frac {S_\tau k(x)}n &\leq \frac{2S_\tau\k\log(n)^{d/(2\beta + d)}(n\wh p(x))^{2\beta/(2\beta + d)}}{n}\\
    &\leq 2S_\tau\k \bigg(\frac{\log n}{n}\bigg)^{d/(2\beta + d)}\bigg(\frac{\theta p(x)}{c_1}\bigg)^{2\beta/(2\beta + d)},
\end{align*}
and the latter is upper-bounded by $p(x)$ since
\begin{align*}
    p(x) \geq \frac{C\ell}{n} \geq \frac{(2S_\tau\k)^{(2\beta + d)/d}}{V_\tau}\bigg(\frac{\theta}{c_1}\bigg)^{2\beta/d}\frac{\ell}{n} \geq (2S_\tau\k)^{(2\beta + d)/d}\bigg(\frac{\theta}{c_1}\bigg)^{2\beta/d}\frac{\log n}{n}.
\end{align*}
We have proven that
\begin{align*}
    p(x) \geq \frac {C\ell}{n} \Rightarrow p(x) \geq \frac{S_\tau k(x)}{n},
\end{align*}
and, since $x \in \Delta_n^{\P}(C, \ell)$ was arbitrary,
\begin{align*}
    L_n^{\P}(\ell)\cap U_n^{\P}(\ell) \subseteq \Bigg\{\Delta_n^{\P}(C, \ell) \subseteq \Theta_n^{\P}(S_\tau, k)\Bigg\}.
\end{align*}
This concludes the proof.
\end{proof}

\section{Proofs of the main results}
\label{app.main.results}

\subsection{Proof of Corollary \ref{cor.upper.bound.one.sample.P}}

Let $\P_{\sX}, \Q\mathstrut_{\!\sX} \in \mP(D, \theta).$ Let $X \sim \P_{\sX},\ X' \sim \Q\mathstrut_{\!\sX},\ Y = f_*(X) + \eps,$ and $Y' \sim f_*(X') + \eps',$ where $\eps, \eps'$ are two i.i.d.\ random variables that are independent from $X$ and $X'$ and sub-exponential with parameters $\alpha, \nu > 0.$ Let $(X_1, Y_1), \dots, (X_n, Y_n)$ be $n$ i.i.d.\ copies of $(X, Y)$ and $(X_1', Y_1'), \dots, (X_m', Y_m')$ be $m$ i.i.d.\ copies of $(X', Y').$ For $k \in \NN$ and $x \in \RR^d,$ whenever those quantities make sense, we denote by $R_k^{\P}(x) := \|x - X_k(x)\|$ and by $R_k^{\Q}(x) := \|x - X_k'(x)\|.$ Consider the local $k$-NN estimator
\begin{align}
    \label{eq.weighted.knn}
    \wh f(x) = \frac{1}{k(x)} \sum_{i=1}^n Y_i(x),
\end{align}
where $k(x) \colon \RR^d \to \{1, \dots, n\}$ and $Y_i(x)$ is the label of the $i$-th nearest neighbour of $x$ among $\{X_1, \dots, X_n\}.$ Our next result concerns the local $k$-NN estimator. Our choice of the neighbour function $k$ involves a density-estimation step.

\begin{theorem}
    \label{th.transfer.local.knn.proof}
    Let $\gamma > 0, \ \kappa > 0.$ Let $C$ be the constant defined in Lemma \ref{lem.inclusions}, $r := \gamma\wedge 2\beta/(2\beta + d),$ and $\mT(r) := \int q(x)/p(x)^r\, dx.$ If $\wh p$ is a $\ell$-NN density estimator for $p$ with $\ell \geq \lceil V_\tau \log n\rceil,$ and $\wh f$ is a local $k$-NN estimator defined in \eqref{eq.weighted.knn} with neighbour function $k$ given by
    \begin{align*}
        k(x) := n\wedge \Big\lceil \kappa\log(n)^{d/(2\beta + d)}(n\wh p(x))^{2\beta/(2\beta + d)}\Big\rceil \vee \big\lceil \log n \big\rceil.
    \end{align*}
    then, there exists an $N > 2d\vee e$ such that for all $n \geq N,$ with probability at least $1 - 3n^{1- \tau},$
    \begin{multline*}
        \big\|\wh f - f_*\big\|_{L^2(\Q\mathstrut_{\!\sX})}^2 \leq \bigg(\frac{V_\tau}{C}\bigg)^{r - 2\beta/(2\beta +d)}\Bigg[\frac{2C_\tau}{\kappa}\big(S_\tau\big)^{2\beta/(2\beta + d)} + A\Bigg]\mT_{\P_{\sX}}(r)\bigg(\frac{\log n}{n}\bigg)^{r}\\
        + \Bigg[2C_\tau + 8L^2\Bigg]\mT_{\P_{\sX}}(r)\bigg(\frac{C\ell}{n}\bigg)^{r}.
    \end{multline*}
    with
    \begin{align*}
        A := 8L^2\big(2\k S_\tau\big)^{2\beta/d}\bigg(\frac{\theta}{c_1}\bigg)^{4\beta^2/(d(2\beta + d))}.
    \end{align*}
\end{theorem}

\begin{remark}
    In particular, the choice $\ell = \big\lceil V_\tau \log n\big\rceil$ in Theorem \ref{th.transfer.local.knn.proof} leads to
    \begin{align*}
        \big\|\wh f - f_*\big\|_{L^2(\Q\mathstrut_{\!\sX})}^2 \leq V_\tau^r\mT_{\P_{\sX}}(r)\Bigg[\frac{2C_\tau}{\kappa}\bigg(\frac{S_\tau C}{V_\tau}\bigg)^{r_\beta} + A\bigg(\frac{C}{C_\tau}\bigg)^{r_\beta} + 2^r\bigg[2C_\tau + 8L^2\bigg]\Bigg]\bigg(\frac{\log n}n\bigg)^{r},
    \end{align*}
    with probability at least $1 - 3n^{1-\tau}.$ Additionally, the statement of Corollary \ref{cor.upper.bound.one.sample.P} is obtained after choosing $\k = 1.$
\end{remark}

\begin{proof}
    Let $X \sim \Q\mathstrut_{\!\sX}, \ C$ be the constant defined in Lemma \ref{lem.inclusions}, and denote $\Delta = \Delta_n^{\P_{\sX}}(C, \ell).$ We work on the event 
    \begin{align}
        \label{eq.event.local.knn}
        L_n^{\P_{\sX}}\cap U_n^{\P_{\sX}}(\ell)\cap V_n^{\P_{\sX}}(\wh f).
    \end{align}
    By Lemma \ref{lem.inclusions} $(ii),$ $\Delta \subseteq \Theta_n^{\P_{\sX}}(S_\tau, k),$ and, by definition, $k(x) \geq \lceil \log n\rceil$. Hence, we can apply Proposition \ref{prop.decomp.knn} to obtain
    \begin{align}
        \label{eq.local.k.proof.decomp}
        \big\|\wh f - f_*\big\|_{L^2(\Q\mathstrut_{\!\sX})}^2 &\leq \E\Bigg[\frac{2C_\tau\log n}{k(X)} + 16L^2\bigg(\frac{S_\tau k(X)}{np(X)}\bigg)^{2\beta/d}\mathds{1}(X \in \Delta) + 8L^2\mathds{1}(X \in \Delta^c)\Bigg].
    \end{align}
    We first derive pointwise bounds on the three quantities occurring inside the expectation operator in \eqref{eq.local.k.proof.decomp}. For the first term, using the definition of $k,$ Lemma \ref{lem.inclusions}, and Lemma \ref{lem.density.ratio} $(ii)$, we get
    \begin{align}
        \label{eq.local.one.sample.pointwise.variance.bound}
        \frac{2C_\tau\log n}{k(x)} &\leq  \frac{2C_\tau\log n}{\kappa \log(n)^{d/(2\beta + d)} (n \wh p(x))^{2\beta/(2\beta + d)}}\mathds{1}(x \in \Delta) +  \frac{2C_\tau\log n}{\lceil \log n\rceil}\mathds{1}(x \in \Delta^c)\nonumber\\
        &\leq \frac{2C_\tau}{\k}\bigg(\frac{S_\tau \log n}{np(x)}\bigg)^{2\beta/(2\beta + d)}\mathds{1}(x \in \Delta) + 2C_\tau\mathds{1}(x \in \Delta^c).
    \end{align}
    Using Lemma \ref{lem.inclusions} $(i)$ and the condition $n \geq N\vee e,$ we have $k(x) \geq \lceil \log n\rceil \geq 1.$ This allows to use $\lceil a \rceil \leq 2a$ with $a$ being the expression inside the ceiling operator in the expression of $k(x).$ Additionally, we use Lemma \ref{lem.density.ratio} $(i)$ to bound the second term in \eqref{eq.local.k.proof.decomp} as follows
    \begin{align}
        \label{eq.local.one.sample.pointwise.bias.hd.bound}
        8L^2\bigg(\frac{S_\tau k(x)}{np(x)}\bigg)^{2\beta/d} & \leq 8L^2\big(S_\tau\big)^{2\beta/d}\bigg(\frac{2\k\log(n)^{d/(2\beta + d)}(n\wh p(x))^{2\beta/ (2\beta + d)}}{np(x)}\bigg)^{2\beta/d} \nonumber\\
        &\leq 8L^2\big(S_\tau\big)^{2\beta/d}\Bigg(2\k\bigg(\frac{\theta}{c_1}\bigg)^{2\beta/(2\beta + d)}\bigg(\frac{\log(n)}{np(x)}\bigg)^{d/(2\beta + d)}\Bigg)^{2\beta/d} \nonumber\\
        &\leq \underbrace{8L^2\big(2\k S_\tau\big)^{2\beta/d}\bigg(\frac{\theta}{c_1}\bigg)^{4\beta^2/(d(2\beta + d))}}_{= A}\bigg(\frac{\log(n)}{np(x)}\bigg)^{2\beta/(2\beta + d)}.
    \end{align}
    We now integrate the previously derived bounds with respect to $\Q\mathstrut_{\!\sX}.$ We begin by integrating the first term of \eqref{eq.local.one.sample.pointwise.variance.bound}. This gives,
    \begin{align*}
        \int_{\Delta}\frac{2C_\tau \log n}{k(X)}\, d\Q\mathstrut_{\!\sX}(x) &\leq \frac{2C_\tau}{\kappa}\bigg(\frac{S_\tau\log n}{n}\bigg)^{2\beta/(2\beta + d)}\int_{\Delta}\frac{q(x)}{p(x)^{r}}p(x)^{r - 2\beta/(2\beta + d)}\, dx\\
        &\leq \frac{2C_\tau}{\kappa}\big(S_\tau\big)^{2\beta/(2\beta + d)}\mT_{\P_{\sX}}(r)\bigg(\frac{\log n}{n}\bigg)^{2\beta/(2\beta + d)}\bigg(\frac{C\ell}{n}\bigg)^{r - 2\beta/(2\beta + d)}\\
        &\overset{(i)}{\leq} \frac{2C_\tau}{\kappa}\big(S_\tau\big)^{2\beta/(2\beta + d)}\mT_{\P_{\sX}}(r)\bigg(\frac{V_\tau}{C}\bigg)^{r - 2\beta/(2\beta + d)}\bigg(\frac{\log n}{n}\bigg)^{r},
    \end{align*}
    where step $(i)$ follows from the fact that $r - 2\beta/(2\beta + d) \leq 0.$ For the second term in \eqref{eq.local.one.sample.pointwise.variance.bound}, we obtain
    \begin{align*}
        \int_{\Delta^c}\frac{2C_\tau\log n}{k(x)}\, d\Q\mathstrut_{\!\sX}(x) &\leq 2C_\tau\mT_{\P_{\sX}}(r)\bigg(\frac{C\ell}{n}\bigg)^r.
    \end{align*}
    Integrating \eqref{eq.local.one.sample.pointwise.bias.hd.bound} with respect to $\Q\mathstrut_{\!\sX}$ leads to 
    \begin{align*}
        \int_{\Delta}8L^2\bigg(\frac{S_\tau k(x)}{np(x)}\bigg)^{2\beta/d}\, d\Q\mathstrut_{\!\sX}(x) &\leq A\bigg(\frac{\log n}n\bigg)^{2\beta/(2\beta + d)}\int_{\Delta}\frac{q(x)}{p(x)^{r}}p(x)^{r - 2\beta/(2\beta + d)}\, dx\\
        &\leq A\mT_{\P_{\sX}}(r)\bigg(\frac{\log n}n\bigg)^{2\beta/(2\beta + d)}\bigg(\frac{C\ell}n\bigg)^{r - 2\beta/(2\beta + d)}\\
        &\leq A\mT_{\P_{\sX}}(r)\bigg(\frac{V_\tau}{C}\bigg)^{r - 2\beta/(2\beta + d)}\bigg(\frac{\log n}n\bigg)^{r}
    \end{align*}
    We now bound the last term in \eqref{eq.local.k.proof.decomp}. We use the fact that on $\Delta^c, \ p(x) < C\ell/n$ to obtain
    \begin{align}
        \label{eq.local.one.sample.bias.lp.bound}
        8L^2\Q\mathstrut_{\!\sX}\{\Delta^c\} = 8L^2\int_{\Delta^c}\frac{q(x)}{p(x)^r}p(x)^{r}\, dx \leq 8L^2\bigg(\frac{C\ell}n\bigg)^r\mT_{\P_{\sX}}(r).
    \end{align}
    Putting everything together yields
    \begin{multline*}
        \big\|\wh f - f_*\big\|_{L^2(\Q\mathstrut_{\!\sX})}^2 \leq \bigg(\frac{V_\tau}{C}\bigg)^{r - 2\beta/(2\beta +d)}\Bigg[\frac{2C_\tau}{\kappa}\big(S_\tau\big)^{2\beta/(2\beta + d)} + A\Bigg]\mT_{\P_{\sX}}(r)\bigg(\frac{\log n}{n}\bigg)^{r}\\
        + \Bigg[2C_\tau + 8L^2\Bigg]\mT_{\P_{\sX}}(r)\bigg(\frac{C\ell}{n}\bigg)^{r}.
    \end{multline*}
    Finally, Proposition \ref{prop.neighbours.distance.bound.zeta}, Proposition \ref{prop.neighbours.distance.lower.bound.zeta}, Corollary \ref{cor.variance.bound.nice.version}, and Remark \ref{rem.bound.variable.k} together with the union bound ensure that
    \begin{align*}
        \PP\Bigg\{L_n^{\P_{\sX}}(\ell)\cap U_n^{\P_{\sX}}(\ell)\cap V_n^{\P_{\sX}}(\wh f)\Bigg\} \geq 1 - 2n^{-\tau} - n^{1-\tau} \geq 1 - 3n^{1-\tau}.
    \end{align*}
\end{proof}

\subsection{Proof of Theorem \ref{th.upper.bound.two.sample}}

To obtain our main result, we need an additional lemma that allows us to lower bound $\zeta_{h_-}(x)$ for $x$ in a low-density region.
\begin{lemma}
    \label{lem.bounded.below.zeta.low.density}
    Let $\P \in \mP(D, \theta)$ with $D, \theta > 0,$ and $k \in \{1, \dots, n\}.$ If $x \in \supp(\P)$ satisfies $p(x) \leq \alpha k/n,$ with $\alpha > 0,$ then, 
    \begin{align*}
         \zeta_{h_-}(x) \geq \bigg(\frac{c_1}{\alpha \theta}\bigg)^{1/d} \wedge 1.
    \end{align*}
\end{lemma}

\begin{proof}
    If $\zeta_{h_-}(x) > 1,$ then, the result is immediate. By contradiction, we show the claim in the case $\zeta_{h_-}(x) \leq 1$. Let $t \in (0, 1]$ be arbitrary and assume that $\zeta_{h_-}(x) \leq t^{1/d} \leq 1.$ Then, by \eqref{eq.minimal.maximal.mass}, we have $h_- = c_1 k/n \leq \theta tp(x).$ Rearranging the latter and using the assumption leads to
    \begin{align*}
        \frac{c_1 k}{\theta t n} \leq p(x) \leq \frac{\alpha k}n,
    \end{align*}
    which implies $t \geq c_1/(\alpha \theta).$ Hence
    \begin{align*}
         \zeta_{h_-}(x) \geq \bigg(\frac{c_1}{\theta\alpha}\bigg)^{1/d} \wedge 1.
    \end{align*}
\end{proof}

Before proving our main result, we recall the setting. Let $D, \theta > 0.$ For $\P_{\sX}, \Q\mathstrut_{\!\sX} \in \mP(D, \theta)$, we assume Model \eqref{eq.model}, under \eqref{eq.covariate.shift} where the unknown function $f_*$ lies in $\mH(L, \beta).$ We have access to
\begin{enumerate}
    \item $\{(X_1, Y_1), \dots, (X_n, Y_n)\},$ an i.i.d.\ $n$-sample from Model \eqref{eq.model} with covariates' distribution $\P_{\sX},$ and
    \item $\{(X_1', Y_1'), \dots, (X_n', Y_n')\},$ an i.i.d.\ $m$-sample from Model \eqref{eq.model} with covariates' distribution $\Q\mathstrut_{\!\sX}.$
\end{enumerate}
In what follows, we use the following estimator
\begin{align}
    \label{eq.two.sample.local.knn.appendix}
    \wh f(x) = \frac{1}{k_{\P}(x) + k_{\Q}(x)}\bigg(\sum_{i=1}^{k_{\P}(x)} Y_i(x) + \sum_{j=1}^{k_{\Q}(x)}Y_j'(x)\bigg),
\end{align}
where $k_{\P} \colon \RR^d \to \{1, \dots n\}, \ k_{\Q} \colon \RR^d \to \{1, \dots, m\},$ and $Y_i(x)$ (resp.\ $Y_j'(x)$) denotes the label associated to the $i$-th (resp.\ $j$-th) nearest neighbour of $x$ among $\{X_1, \dots, X_n\}$ (resp.\ $\{X_1', \dots, X_m'\}$).

\begin{theorem}[Rates for two-sample local nearest neighbours]
\label{th.rates.local.k.two.sample.proof}
    Let $\tau > 0, \k_{\P}, \k_{\Q} >0,$ and define
    \begin{align*}
        r_S := 2\beta/(2\beta + d) \wedge \gamma,\text{ and }r_T := s\wedge 2\beta/(2\beta + d).
    \end{align*}
    Let $\wh p$ and $\wh q$ be, respectively, the $\ell_{\P}$ and $\ell_{\Q}$-NN density estimators of $p$ and $q$ with
    \begin{align*}
        \ell := \ell_{\P} = \ell_{\Q} = \big\lceil V_\tau\log nm\big\rceil.
    \end{align*}
    Finally, consider $\wh f$ to be the local $(k_{\P}, k_{\Q})$-NN estimator with neighbour functions given by
    \begin{align*}
        k_{\P}(x) &:= n\wedge \Big\lceil \log(nm)^{d/(2\beta + d)}(n\wh p(x))^{2\beta/(2\beta + d)}\Big\rceil \vee \big\lceil \log (nm) \big\rceil\\
        k_{\Q}(x) &:= m\wedge \Big\lceil \log(nm)^{d/(2\beta + d)}(m\wh q(x))^{2\beta/(2\beta + d)}\Big\rceil \vee \big\lceil \log (nm) \big\rceil.
    \end{align*}
    Then, there exists $N \geq 2d$ and $M \geq 2d$ such that for all $n \geq N, m \geq M,$ it holds, with probability at least $1 - 3n^{1-\tau} - 3m^{1-\tau},$
    \begin{align*}
        \sup_{f \in \mH(L, \beta)}\PP\Big\{\mE_{\Q}(\wh f, f) \geq C(\tau)R(n,m, \gamma, s)\Big\} \leq 3(n^{1 - \tau} + m^{1 - \tau})
    \end{align*}
    with, if $(\gamma - r_\beta)(s - r_\beta) <0$ and $m \in [n, n^{\gamma/s}],$ then
    \begin{align*}
        R(n,m, \gamma, s) &= \tilde C_0\Big(\frac{\log(nm)}{n}\Big)^{\gamma a}\Big(\frac{\log(nm)}{m}\Big)^{s(1- a)}\mT_{\P}(\gamma)^{a}\mT_{\Q}(s)^{1-a},
    \end{align*}
    where $\tilde C_0$ is a constant independent of $n,m, \gamma, s, \tau$ and, if $(\gamma - r_\beta)(s - r_\beta) \geq 0$ or $m \notin [n, n^{\gamma/s}],$ then
    \begin{align*}
         R &= \ol C_0\Big[\mT_{\P}(r_S)\Big(\frac{\log(nm)}{n}\Big)^{r_S}\Big]\wedge \Big[\mT_{\Q}(r_T)\Big(\frac{\log(nm)}{m}\Big)^{r_T}\Big],
    \end{align*}
    where $\ol C_0$ is a constant independent of $n,m, \gamma, s, \tau$ and $C_p := C(\k_{\P}), \ C_q := C(\k_{\Q}),$ where
    \begin{align*}
        C(t) := S_\tau\vee c_2 \vee \frac{S_\tau}{t^{(2\beta + d)/(2\beta)}V_\tau} \vee \frac{(2S_\tau t)^{(2\beta + d)/d}}{V_\tau}\bigg(\frac{\theta}{c_1}\bigg)^{2\beta/d}.
    \end{align*}
    
\end{theorem}
\begin{proof}
    Let 
    \begin{align*}
        C(t) := S_\tau\vee c_2 \vee \frac{S_\tau}{t^{(2\beta + d)/(2\beta)}V_\tau} \vee \frac{(2S_\tau t)^{(2\beta + d)/d}}{V_\tau}\bigg(\frac{\theta}{c_1}\bigg)^{2\beta/d}.
    \end{align*}
    and define $C_p := C(\k_{\P}), \ C_q := C(\k_{\Q}).$ For $R \in \{\P, \Q\},$ Define the event
    \begin{align*}
        E_n^{R_{\sX}}(\ell, \wh f_R) := L_n^{R_{\sX}} \cap \U_n^{R_{\sX}}(\ell) \cap V_n^{R_{\sX}}(\wh f_R).
    \end{align*}
    We work on the event
    \begin{align}
        \label{eq.event.two.sample.local}
        E_n^{\P_{\sX}}(\ell_{\P}, \wh f_{\P}) \cap E_m^{\Q\mathstrut_{\!\sX}}(\ell_{\Q}, \wh f_{\Q}).
    \end{align}
   Let $X \sim \Q\mathstrut_{\!\sX}$ be independent of the samples. We start by decomposing the risk as follows.
    \begin{align*}
        \big\|\wh f - f_*\big\|_{L^2(\Q\mathstrut_{\!\sX})}^2 &= \E_X\bigg[\bigg(\frac{1}{k_{\P}(X) + k_{\Q}(X)}\bigg(\sum_{i=1}^{k_{\P}(X)}\big(Y_i(X) - f_*(X)\big) + \sum_{j=1}^{k_{\Q}(X)} \big(Y_j'(X) - f_*(X)\big)\bigg)\bigg)^2\bigg]\\
        &= \E\bigg[\bigg(\frac{k_{\P}(X)}{k_{\P}(X) + k_{\Q}(X)}\big(\wh f_{\P}(X) - f_*(X)\big) + \frac{k_{\Q}}{k_{\P}(X) + k_{\Q}(X)}\big(\wh f_{\Q}(X) - f_*(X)\big)\bigg)^2\bigg],
    \end{align*}
    where $\wh f_{\P}(x) := k_{\P}(x)^{-1}\sum_{i=1}^{k_{\P}(x)} Y_i(x)$ and $\wh f_{\Q}(x) := k_{\Q}(x)^{-1}\sum_{j=1}^{k_{\Q}(x)} Y_j'(x).$ Denote by $w_{\P}(x) := k_{\P}(x)/(k_{\P}(x) + k_{\Q}(x))$ and by $w_{\Q}(x) := k_{\Q}(x)/(k_{\P}(x) + k_{\Q}(x)).$ A further application of Jensen's inequality yields
    \begin{align}
        \label{eq.decomposition.expectation.two.sample}
        \big\|\wh f - f_*\big\|_{L^2(\Q\mathstrut_{\!\sX})}^2 &\leq \E_X\Big[w_{\P}(X)\big(\wh f_{\P}(X) - f_*(X)\big)^2 + w_{\Q}(X) \big(\wh f_{\Q}(X) - f_*(X)\big)^2\Big].
    \end{align}
    Let $\Delta_{\P} := \Delta_n^{\P_{\sX}}(C_p, \ell_{\P})$ and $\Delta_{\Q} := \Delta_m^{\Q\mathstrut_{\!\sX}}(C_q, \ell_{\Q}).$ We apply Proposition \ref{prop.decomp.knn} twice to obtain bounds on $(\wh f_{\P}(X) - f_*(X))^2$ and $(\wh f_{\Q}(X) - f_*(X))^2.$ On the event \eqref{eq.event.two.sample.local}, we can decompose the error inside the expectation in \eqref{eq.decomposition.expectation.two.sample} in four different
    ways, depending on the region. The result is as follows
    \begin{align}
        \label{eq.first.bv}
        \begin{split}
        \big(\wh f_{\P}(x) - f_*(x)\big)^2 &\leq \begin{dcases}
            b^2_{\P}(x) + V_{\P}(x) & \text{ if } x \in \Delta_{\P}\\
            8L^2 + V_{\P}(x) &\text{ if } x \in \Delta_{\P}^c
        \end{dcases}\\
        \big(\wh f_{\Q}(x) - f_*(x)\big)^2 &\leq \begin{dcases}
            b^2_{\Q}(x) + V_{\Q}(x) & \text{ if } x \in \Delta_{\Q}\\
            8L^2 + V_{\Q}(x) &\text{ if } x \in \Delta_{\Q}^c,
        \end{dcases}
        \end{split}
    \end{align}
    where
    \begin{align*}
        b_{\P}^2(x) &:= 8L^2\bigg(\frac{S_\tau k_{\P}(x)}{np(x)}\bigg)^{2\beta/d}, \ V_{\P}(x) := 2C_\tau \frac{\log n}{k_{\P}(x)},\\
        b_{\Q}^2(x) &:= 8L^2\bigg(\frac{S_\tau k_{\Q}(x)}{mq(x)}\bigg)^{2\beta/d}, \ V_{\Q}(x) := 2C_\tau \frac{\log m}{k_{\Q}(x)},
    \end{align*}
    In order to derive bounds on the excess risk, we further subdivide $\RR^d$ into four regions. Let 
    \begin{align*}
        \Delta := \Delta_{\P}\cap \Delta_{\Q}, \ \Gamma_{\P} := \Delta_{\P}\cap \Delta_{\Q}^c,\ \Gamma_{\Q} := \Delta_{\P}^c\cap \Delta_{\Q},\text{ and } \Theta := \Delta_{\P}^c\cap \Delta_{\Q}^c.
    \end{align*}
    We now bound the pointwise risk of $\wh f$ on each of these regions.\\
    \noindent\textbf{Region $\Delta$.} For $x \in \Delta,$ the total pointwise risk is
    \begin{align}
        \label{eq.pointwise.loss.bulk}
        \big(\wh f(x) - f_*(x)\big)^2 &\leq w_{\P}(x)\big(b_{\P}(x)
        ^2 + V_{\P}(x)\big) + w_{\Q}(x)\big(b_{\Q}(x)
        ^2 + V_{\Q}(x)\big).
    \end{align}
    We now denote $r_\beta := 2\beta/(2\beta + d), \ \mT_{\P}(u) := \int q(x)/p(x)^{u}\, dx$ and $\mT_{\Q}(v) := \int q(x)^{1 -v}\, dx.$ From the definitions of $k_{\P}, k_{\Q},$ Lemma \ref{lem.inclusions} and Lemma \ref{lem.density.ratio} $(ii)$, we obtain
    \begin{align}
        \label{eq.two.sample.local.variance.bound}
        w_{\P}(x)V_{\P}(x) + w_{\Q}(x)V_{\Q}(x) &= 2C_\tau\Big(\frac{\log n}{k_{\P}(x) + k_{\Q}(x)} + \frac{\log m}{k_{\P}(x) + k_{\Q}(x)}\Big)\nonumber\\
        &= 2C_\tau\frac{\log (nm)}{k_{\P}(x) + k_{\Q}(x)}\\
        & = 2C_\tau \Big(\frac{k_{\P}(x)}{\log (nm)} + \frac{k_{\Q}(x)}{\log (nm)}\Big)^{-1}\nonumber\\
        &= 2C_\tau \Bigg(\k_{\P}\frac{\log(nm)^{1 - r_\beta}(n\wh p(x))^{r_\beta}}{\log(nm)} + \k_{\Q}\frac{\log(nm)^{1 - r_\beta}(m\wh q(x))^{r_\beta}}{\log(nm)}\Bigg)^{-1}\nonumber\\
        &\leq 2C_\tau \Bigg(\k_{\P}\bigg(\frac{np(x)}{S_\tau\log(nm)}\bigg)^{r_\beta} + \k_{\Q}\bigg(\frac{m q(x)}{S_\tau \log(nm)}\bigg)^{r_\beta}\Bigg)^{-1}\nonumber\\
        &= 2C_\tau S_\tau^{r_\beta} \Bigg(\k_{\P}\bigg(\frac{np(x)}{\log(nm)}\bigg)^{r_\beta} + \k_{\Q}\bigg(\frac{m q(x)}{\log(nm)}\bigg)^{r_\beta}\Bigg)^{-1}
    \end{align}
    We now bound the remaining terms in \eqref{eq.pointwise.loss.bulk}. From the definition of $k_{\P},$ and Lemma \ref{lem.density.ratio}, we get
    \begin{align*}
        k_{\P}(x)b_{\P}^2(x) = k_{\P}(x)\bigg(\frac{S_\tau k_{\P}(x)}{np(x)}\bigg)^{2\beta/d} &= S_\tau^{2\beta/d}\frac{k_{\P}(x)^{(2\beta + d)/d}}{(np(x))^{2\beta/d}}\\
        &=S_\tau^{2\beta/d}\k_{\P}^{(2\beta + d)/d}\log(nm)\bigg(\frac{\wh p(x)}{p(x)}\bigg)^{2\beta/d}\\
        &\leq \k_{\P}^{(2\beta + d)/d}\bigg(\frac{\theta S_\tau}{c_1}\bigg)^{2\beta/d}\log(nm) := K_{\P}\log(nm).
    \end{align*}
    The same analysis for the second term in the expression of $b^2(x)$ yields
    \begin{align*}
        k_{\Q}(x)b_{\Q}^2(x) = k_{\Q}(x)\bigg(\frac{S_\tau k_{\Q}(x)}{mq(x)}\bigg)^{2\beta/d} \leq \k_{\Q}^{(2\beta + d)/d}\bigg(\frac{\theta S_\tau}{c_1}\bigg)^{2\beta/d}\log(nm) =: K_{\Q}\log(nm).
    \end{align*}
    Hence,
    \begin{align*}
        8L^2(w_{\P}(x)b^2_{\P}(x) + w_{\Q}(x)b^2_{\Q}(x)) &\leq 8L^2(K_{\P} + K_{\Q})\frac{\log(nm)}{k_{\P}(x) + k_{\Q}(x)} =: \ol K\frac{\log(nm)}{k_{\P}(x) + k_{\Q}(x)},
    \end{align*}
    where $\ol K := K_{\P} + K_{\Q}.$ The latter quantity can be bounded in the exact same way as in \eqref{eq.two.sample.local.variance.bound}, which shows that, for $x \in \Delta,$
    \begin{align*}
        \big(\wh f(x) - f_*(x)\big)^2 &\leq (2C_\tau + 8L^2\ol K)S_\tau^{r_\beta} \Bigg(\k_{\P}\bigg(\frac{np(x)}{\log(nm)}\bigg)^{r_\beta} + \k_{\Q}\bigg(\frac{m q(x)}{\log(nm)}\bigg)^{r_\beta}\Bigg)^{-1}\\
        &:= C_{\Delta}\Bigg(\k_{\P}\bigg(\frac{np(x)}{\log(nm)}\bigg)^{r_\beta} + \k_{\Q}\bigg(\frac{m q(x)}{\log(nm)}\bigg)^{r_\beta}\Bigg)^{-1}\\
        &\leq C_{\Delta}\bigg(\k_{\P}^{-1}\Big(\frac{\log(nm)}{np(x)}\Big)^{r_\beta}\wedge\k_{\Q}^{-1}\Big(\frac{\log(nm)}{mq(x)}\Big)^{r_\beta}\bigg)
    \end{align*}
    \noindent\textbf{Regions $\Gamma_{\P}$ and $\Gamma_{\Q}.$} For $x \in \Gamma_{\P},$ we have
    \begin{align*}
        \big(\wh f(x) - f_*(x)\big)^2 &\leq w_{\P}\big(b_{\P}^2(x) + V_{\P}(x)\big)^2 + w_{\Q}(x)\big(8L^2 + V_{\Q}(x)\big).
    \end{align*}
    Since $x \in \Delta_{\P},$ the bound on $w_{\P}(x)(\wh f_{\P}(x) - f_*(x))^2$ from the previous case still holds, and we obtain
    \begin{align*}
        \big(\wh f(x) - f_*(x)\big)^2 &\leq \frac{(K_{\P} + 2C_\tau) \log(nm) + 8L^2k_{\Q}(x)}{k_{\P}(x) + k_{\Q}(x)}.
    \end{align*}
    Recall that $x \in \Delta_{\Q}^c,$ meaning that $q(x) \leq C_q\ell_{\Q}/m,$ hence, by virtue of Lemma \ref{lem.bounded.below.zeta.low.density}, it holds that
    \begin{align}
        \label{eq.ub.wh.q.low.density}
        \wh q(x) = \frac{\ell_{\Q}}{mR_{\ell_{\Q}, \Q}(x)^d} \leq \frac{\ell_{\Q}}{m\zeta_{h_-, \Q}(x)^d} \leq \frac{a\theta}{c_1}\frac{\ell_{\Q}}{m}.
    \end{align}
    Hence, recalling that $\ell_{\P} = \ell_{\Q} = \lceil V_\tau \log(nm)\rceil,$
    \begin{align*}
        \big(\wh f(x) - f_*(x)\big)^2 &\leq \frac{(K_{\P} + 2C_\tau) \log(nm) + 16L^2\k_{\Q}(a\theta/c_1)^{r_\beta}\log(nm)^{1-r_\beta}\ell_{\Q}^{r_\beta}}{k_{\P}(x) + k_{\Q}(x)}\\
        &\leq (K_{\P} + 2C_\tau + 16L^2V_\tau^{r_\beta}\k_{\Q}(a\theta/c_1)^{r_\beta}) \frac{\log(nm)}{k_{\P}(x) + k_{\Q}(x)}\\
        &\leq (K_{\P} + 2C_\tau + 16L^2V_\tau^{r_\beta}\k_{\Q}(a\theta/c_1)^{r_\beta})\frac{\log(nm)}{k_{\P}(x)}\\
        &\leq \frac{K_{\P} + 2C_\tau + 16L^2V_\tau^{r_\beta}\k_{\Q}(a\theta/c_1)^{r_\beta}}{\k_{\P}}\bigg(\frac{\log(nm)}{n\wh p(x)}\bigg)^{r_\beta}\\
        &\leq \frac{(K_{\P} + 2C_\tau + 16L^2V_\tau^{r_\beta}\k_{\Q}(a\theta/c_1)^{r_\beta})S_\tau^{r_\beta}}{\k_{\P}}\bigg(\frac{\log(nm)}{n p(x)}\bigg)^{r_\beta}=: C_{\Gamma_{\P}}\bigg(\frac{\log(nm)}{np(x)}\bigg)^{r_\beta}.
    \end{align*}
    The analysis on the region $\Gamma_{\Q}$ is symmetric and we obtain, for $x \in \Gamma_{\Q},$
    \begin{align*}
        \big(\wh f(x) - f_*(x)\big)^2 &\leq C_{\Gamma_{\Q}}\bigg(\frac{\log(nm)}{mq(x)}\bigg)^{r_\beta}, \ \text{ with } C_{\Gamma_{\Q}} := \frac{(K_{\Q} + 2C_\tau + 16L^2V_\tau^{r_\beta}\k_{\P}(a\theta/c_1)^{r_\beta})S_\tau^{r_\beta}}{\k_{\Q}}.
    \end{align*}
    \noindent\textbf{Region $\Theta.$} For $x \in \Theta,$ we use the fact that, by definition, $k_{\P}(x)\wedge k_{\Q}(x) \geq \log(nm)$ to obtain
    \begin{align*}
        \big(\wh f(x) - f_*(x)\big)^2 &\leq 8L^2 + w_{\P}(x) V_{\P}(x) + w_{\Q}(x)V_{\Q}(x)\\
        &\leq 8L^2 + 2C_\tau \frac{\log (nm)}{k_{\P}(x) + k_{\Q}(x)}\\
        &\leq 8L^2 + C_\tau =: C_{\theta}
    \end{align*}
    This finishes the pointwise bounds on our partition of $\RR^d.$ It remains to integrate the latter against $\Q\mathstrut_{\!\sX}.$ We will proceed to integrate in each region. For each region, we will distinguish cases between the \textbf{wedge regime} and the \textbf{accelerated regime}. We define
    \begin{align*}
        I_{\Delta} &:= \int_{\Delta}\bigg(\k_{\P}^{-1}\Big(\frac{1}{np(x)}\Big)^{r_\beta}\wedge\k_{\Q}^{-1}\Big(\frac{1}{mq(x)}\Big)^{r_\beta}\bigg)\, d\Q\mathstrut_{\!\sX}(x)\\
        I_{\Gamma_{\P}} &:= \int_{\Gamma_{\P}}\bigg(\frac{1}{np(x)}\bigg)^{r_\beta}\, d\Q\mathstrut_{\!\sX}\\
        I_{\Gamma_{\Q}} &:= \int_{\Gamma_{\Q}}\bigg(\frac{1}{mq(x)}\bigg)^{r_\beta}\, d\Q\mathstrut_{\!\sX}(x)\\
        I_{\Theta} &:= \int_{\Theta}\, d\Q\mathstrut_{\!\sX}(x).
    \end{align*}
    With these notations, the final bound on the excess risk reads
    \begin{align*}
        \big\|\wh f - f_*\big\|_{L^2(\Q\mathstrut_{\!\sX})}^2 &\leq \log(nm)^{r_\beta}\big(C_{\Delta}I_{\Delta} + C_{\Gamma_{\P}}I_{\Gamma_{\P}} + C_{\Gamma_{\Q}}I_{\Gamma_{\Q}}) + C_{\Theta}I_{\Theta}.
    \end{align*}
    
    \noindent\textbf{Bound on $\Delta.$}
    We write
    \begin{align*}
        I_{\Delta} &\leq \inf_{\theta \in [0, 1]}\int_{\Delta}(\k_{\P}(np(x))^{r_\beta})^{-\theta}(\k_{\Q}(mq(x))^{r_\beta})^{1 - \theta}\, d\Q\mathstrut_{\!\sX}(x)\\
        &\leq \inf_{\theta \in [0, 1]}(\k_{\P}n^{r_\beta})^{-\theta}(\k_{\Q}m^{r_\beta})^{-(1 - \theta)}\int_{\Delta}p(x)^{-r_\beta \theta}q(x)^{-r_\beta(1 - \theta)}\, d\Q\mathstrut_{\!\sX}\\
        &\leq \inf_{\theta \in [0, 1]}(\k_{\P}n^{r_\beta})^{-\theta}(\k_{\Q}m^{r_\beta})^{-(1 - \theta)}\E_{X \sim \Q\mathstrut_{\!\sX}}\big[p(X)^{-r_\beta \theta}q(X)^{-r_\beta(1 - \theta)}\big]
    \end{align*}
    In the \textbf{wedge regime,} it is sufficient to use the facts that $r_S - r_\beta \leq 0, \ r_M - r_\beta \leq 0$ and that, on $\Delta, \ p(x) \geq C_p\ell_{\P}/n$ and $q(x) \geq C_q \ell_{\Q}/m$ and 
    \begin{align*}
        I_{\Delta} &\leq \bigg[\frac{1}{\k_{\P}n^{r_\beta}}\int_{\Delta}\frac{q(x)}{p(x)^{r_\beta}}\, dx\bigg]\wedge \bigg[\frac{1}{\k_{\Q}m^{r_\beta}}\int_{\Delta}q(x)^{1 - r_\beta}\, dx\bigg]\\
        &\leq \bigg[\frac{1}{\k_{\P}n^{r_\beta}}\int_{\Delta}\frac{q(x)}{p(x)^{r_S}}p(x)^{r_S - r_\beta}\, dx\bigg]\wedge \bigg[\frac{1}{\k_{\Q}m^{r_\beta}}\int_{\Delta}q(x)^{1 - r_\beta}\, dx\bigg]\\
        &\leq \bigg[\frac{1}{\k_{\P}n^{r_S}}(C_p\ell_{\P})^{r_S - r_\beta}\mT_{\P}(r_S)\bigg]\wedge\bigg[\frac{1}{\k_{\Q}m^{r_T}}(C_q\ell_{\Q})^{r_T - r_\beta}\mT_{\Q}(r_T)\bigg].
    \end{align*}
    In the \textbf{acceleration regime}, we can instead take $\theta = \gamma(s - r)/(r_\beta(s - \gamma))$ to obtain
    \begin{align*}
        I_{\Delta} &\leq (\k_{\P}n^{r_\beta})^{-\theta}(\k_{\Q}m^{r_\beta})^{-(1 - \theta)}\E_{X \sim \Q\mathstrut_{\!\sX}}\big[p(X)^{-r_\beta \theta}q(X)^{-r_\beta(1 - \theta)}\big]\\
        &\leq \k_{\P}^{-\theta}\k_{\Q}^{-(1 - \theta)}n^{-\tfrac{\gamma(r_\beta - s)}{\gamma - s}}m^{-\tfrac{s(\gamma - r)}{\gamma - s}}\E_{X \sim \Q\mathstrut_{\!\sX}}\big[p(X)^{-r_\beta \theta}q(X)^{-r_\beta(1 - \theta)}\big].
    \end{align*}
    It remains to verify that the expectation above is finite. We apply H\"older's inequality with the conjugates $a := \gamma/(\theta r_\beta),$ and $b := s/((1 - \theta)r_\beta),$ which leads to
    \begin{align*}
        \E_{X \sim \Q\mathstrut_{\!\sX}}\big[p(X)^{-r_\beta \theta}q(X)^{-r_\beta(1 - \theta)}\big] &\leq \Big(\E_{X \sim \Q\mathstrut_{\!\sX}}\big[p(X)^{-r_\beta \theta a}\big]\Big)^{1/a}\Big(\E_{X \sim \Q\mathstrut_{\!\sX}}\big[q(X)^{-r_\beta(1 - \theta)b}\big]\Big)^{1/b},
    \end{align*}
    and, since $r_\beta\theta a = \gamma$ and $r_\beta(1 - \theta)b = s,$ we conclude that
    \begin{align*}
        I_{\Delta} &\leq \k_{\P}^{-\theta}\k_{\Q}^{-(1 - \theta)}n^{-\tfrac{\gamma(r_\beta - s)}{\gamma - s}}m^{-\tfrac{s(\gamma - r)}{\gamma - s}}\mT_{\P}(\gamma)^{\theta r_\beta/\gamma} \mT_{\Q}(s)^{(1 - \theta)r_\beta/s}.
    \end{align*}
    
    \noindent\textbf{Bound on $\Gamma_{\P}.$} On $\Gamma_{\P},$ by definition, we have
    \begin{align*}
        np(x) \geq C_p\ell_{\P}.
    \end{align*}
    Therefore, in the \textbf{wedge regime}, it holds that
    \begin{align*}
        I_{\Gamma_{\P}} &= \int_{\Gamma_{\P}}(np(x))^{-r_\beta}\, d\Q\mathstrut_{\!\sX} \leq \int_{\Gamma_{\P}}\big[(C_p\ell_{\P})^{-r_\beta} \wedge (np(x))^{-r_\beta}\big]\, d\Q\mathstrut_{\!\sX}.
    \end{align*}
    Since for $x \in \Gamma_{\P}, \ q(x) \leq C_q\ell_{\Q}/m,$ and $p(x) \geq C_p\ell_{\P}/n,$ it holds that 
    \begin{align*}
        I_{\Gamma_{\P}} &\leq \bigg[\frac{(C_p\ell_{\P})^{r_S - r_\beta}}{n^{r_S}}\int_{\Gamma_{\P}}p(x)^{-r_S}\, d\Q\mathstrut_{\!\sX}(x)\bigg] \wedge \bigg[(C_p\ell_{\P})^{-r_\beta}\int_{\Gamma_{\P}}\frac{q(x)}{q(x)^{r_T}}q(x)^{r_T}\, dx\bigg]\\
        &\leq \bigg[\frac{(C_p\ell_{\P})^{r_S - r_\beta}}{n^{r_S}}\mT_{\P}(r_S)\bigg] \wedge \bigg[(C_p\ell_{\P})^{-r_\beta}\int_{\Gamma_{\P}}\frac{q(x)}{q(x)^{r_T}}q(x)^{r_T}\, dx\bigg]\\
        &\leq \bigg[\frac{(C_p\ell_{\P})^{r_S - r_\beta}}{n^{r_S}}\mT_{\P}(r_S)\bigg] \wedge \bigg[(C_p\ell_{\P})^{-r_\beta}\Big(\frac{C_q\ell_{\Q}}{m}\Big)^{r_T}\mT_{\Q}(r_T)\bigg],
    \end{align*}
    which gives the bound in the \textbf{wedge regime}. Next, in the \textbf{accelerated regime} where either $s < r_\beta < \gamma,$ or $\gamma < r_\beta < s,$ we instead write
    \begin{align*}
        I_{\Gamma_{\P}} &= n^{-r_\beta}\int_{\Gamma_{\P}}\frac{q(x)}{p(x)^{r_\beta}}\, dx \\
        &= n^{-r_\beta}\int_{\Gamma_{\P}}\Big(\frac{q(x)}{p(x)^{\gamma}}\Big)^aq(x)^{(1 - s)(1 - a)}p(x)^{-(r_\beta - \gamma a)}q(x)^{s(1 - a)}\, dx,
    \end{align*}
    with 
    \begin{align}
        \label{eq.a}
        a := \tfrac{r_\beta - s}{\gamma - s} \in (0, 1).    
    \end{align}
    We check that
    \begin{align*}
        r_\beta - \gamma a = \frac{r_\beta(\gamma - s) - \gamma(r_\beta -s)}{\gamma - s} = s\frac{\gamma - r_\beta}{\gamma - s}  = s(1 - a) > 0.
    \end{align*}
    Hence, using the definition of $\Gamma_{\P}$ before applying H\"older's inequality with conjugates $1/a$ and $1/(1 - a)$, we get
    \begin{align*}
        I_{\Gamma_{\P}} &\leq (\ell_{\P}C_{\P})^{\gamma a - r_\beta}(\ell_{\Q}C_q)^{s(1 - a)}n^{-\gamma a}m^{-s(1 - a)}\int_{\Gamma_{\P}}\Big(\frac{q(x)}{p(x)^{\gamma}}\Big)^aq(x)^{(1 - s)(1 - a)}\, dx\\
        &\leq (\ell_{\P}C_{p})^{\gamma a - r_\beta}(\ell_{\Q}C_q)^{s(1 - a)}n^{-\gamma a}m^{-s(1 - a)}\mT_{\P}(\gamma)^{a}\mT_{\Q}(s)^{1-a},
    \end{align*}
    which gives us the bound in the accelerated regime.\\
    
    \noindent\textbf{Bound on $\Gamma_{\Q}$.} This case is symmetric to the bound on $\Gamma_{\P}.$ Following the exact same steps and exchanging the roles of $\P_{\sX}$ and $\Q\mathstrut_{\!\sX}$ leads to the following. In the \textbf{wedge regime}, we obtain the bound
    \begin{align*}
        I_{\Gamma_{\Q}} &\leq \bigg[\frac{(C_q\ell_{\Q})^{r_T - r_\beta}}{m^{r_T}}\mT_{\Q}(r_T)\bigg] \wedge \bigg[(C_q\ell_{\Q})^{-r_\beta}\Big(\frac{C_p\ell_{\P}}{n}\Big)^{r_S}\mT_{\P}(r_S)\bigg],
    \end{align*}
    and, in the \textbf{accelerated regime}, we instead rewrite
    \begin{align*}
        I_{\Gamma_{\Q}} &= m^{-r_\beta}\int_{\Gamma_{\Q}}\frac{q(x)}{q(x)^{r\beta}}\, dx\\
        &= m^{-r_\beta}\int_{\Gamma_{\Q}}\Big(\frac{q(x)}{p(x)^\gamma}\Big)^aq(x)^{(1 - s)(1 - a)}p(x)^{\gamma a}q(x)^{-(r_\beta - s(1 - a))}\, dx.
    \end{align*}
    Then, we already know that $r_\beta - s(1 - a) = \gamma a >0,$ therefore, on $\Gamma_{\Q},$
    \begin{align*}
    q(x)^{-(r_\beta - s(1 - a))} \leq (m/(C_q\ell_{\Q}))^{r_\beta - s(1 - a)},
    \end{align*}
    and this enables to conclude
    \begin{align*}
        I_{\Gamma_{\Q}} &\leq (C_q\ell_{\Q})^{s(1 - a) - r_\beta}(C_p\ell_{\P})^{\gamma a}n^{-\gamma a}m^{-s(1 - a)}\mT_{\P}(\gamma)^{a}\mT_{\Q}(s)^{1 - a}.
    \end{align*}
    
    \noindent\textbf{Bound on $\Theta.$} We aim to bound
    \begin{align*}
        I_{\Theta} = \int_{\Theta}q(x)\, dx.
    \end{align*}
    We again distinguish the two regimes. In the \textbf{wedge regime}, using the fact that if $\gamma > 0$ then $\P_{\sX} \ll \Q\mathstrut_{\!\sX},$ it holds that 
    \begin{align*}
        I_{\Theta} = \int_{\Theta}q(x)\, dx &= \int_{\Theta}\Big(p(x)^{r_S}\frac{q(x)}{p(x)^{r_S}}\Big)\wedge \Big(q(x)^{r_T}\frac{q(x)}{q(x)^{r_T}}\Big)\, dx\\
        &\leq \bigg(\int_{\Theta}p(x)^{r_S}\frac{q(x)}{p(x)^{r_S}}\, dx\bigg)\wedge\bigg(q(x)^{r_T}\frac{q(x)}{q(x)^{r_T}}\, dx\bigg)\\
        &\leq \Big[\Big(\frac{C_p\ell_{\P}}{n}\Big)^{r_S}\mT_{\P}(r_S)\Big]\wedge \Big[\Big(\frac{C_q\ell_{\Q}}{m}\Big)^{r_T}\mT_{\Q}(r_T)\Big],
    \end{align*}
    where we also used that by definition, if $x \in \Theta,$ then $p(x) < C_p\ell_{\P}/n$ and $q(x) < C_q\ell_{\Q}/m.$ In the \textbf{accelerated regime}, we instead proceed as follows
    \begin{align*}
        I_{\Theta} = \int_{\Theta}q(x)\, dx &\leq \int_{\Theta} \Big(\frac{q(x)}{p(x)^{\gamma}}\Big)^a\Big(\frac{q(x)}{q(x)^s}\Big)^{1 - a}p(x)^{\gamma a}q(x)^{s(1 - a)}\, dx\\
        &\leq (C_p\ell_{\P})^{\gamma a}(C_q\ell_{\Q})^{s(1 - a)}n^{-\gamma a}m^{-s(1 - a)}\mT_{\P}(\gamma)^{a}\mT_{\Q}(s)^{1-a},
    \end{align*}
    while choosing $a$ as in \eqref{eq.a}.\\
    
    \noindent\textbf{Conclusion.} It remains to gather the terms to obtain, in the \textbf{wedge regime},
    \begin{align*}
        \big\|\wh f - f_*\big\|^2_{L^2(\Q\mathstrut_{\!\sX})} &\leq C_{\Delta}\log(nm)^{r_\beta}\bigg[\frac{1}{\k_{\P}n^{r_S}}(C_p\ell_{\P})^{r_S - r_\beta}\mT_{\P}(r_S)\bigg]\wedge\bigg[\frac{1}{\k_{\Q}m^{r_T}}(C_q\ell_{\Q})^{r_T - r_\beta}\mT_{\Q}(r_T)\bigg]\\
        &\qquad + C_{\Gamma_{\P}}\log(nm)^{r_\beta}\bigg[\frac{(C_p\ell_{\P})^{r_S - r_\beta}}{n^{r_S}}\mT_{\P}(r_S)\bigg] \wedge \bigg[(C_p\ell_{\P})^{-r_\beta}\Big(\frac{C_q\ell_{\Q}}{m}\Big)^{r_T}\mT_{\Q}(r_T)\bigg]\\
        &\qquad + C_{\Gamma_{\Q}}\log(nm)^{r_\beta}\bigg[(C_q\ell_{\Q})^{-r_\beta}\Big(\frac{C_p\ell_{\P}}{n}\Big)^{r_S}\mT_{\P}(r_S)\bigg]\wedge \bigg[\frac{(C_q\ell_{\Q})^{r_T - r_\beta}}{m^{r_T}}\mT_{\Q}(r_T)\bigg]\\
        &\qquad + C_{\Theta}\Big[\Big(\frac{C_p\ell_{\P}}{n}\Big)^{r_S}\mT_{\P}(r_S)\Big]\wedge \Big[\Big(\frac{C_q\ell_{\Q}}{m}\Big)^{r_T}\mT_{\Q}(r_T)\Big].
    \end{align*}
    Taking $\k_{\P} = \k_{\Q} = 1$ and recalling that $\ell_{\P} = \ell_{\Q} = \lceil V_\tau\log(nm)\rceil$ leads to first obtain $C_p = C_q =: C_0$ and
    \begin{align*}
        \big\|\wh f - f_*\big\|^2_{L^2(\Q\mathstrut_{\!\sX})} &\leq C_{\Delta}\bigg[(C_0V_\tau)^{r_S - r_\beta}\mT_{\P}(r_S)\Big(\frac{\log(nm)}{n}\Big)^{r_S}\bigg]\wedge\bigg[(C_0V_\tau)^{r_T - r_\beta}\mT_{\Q}(r_T)\Big(\frac{\log nm}{m}\Big)^{r_T}\bigg]\\
        &\qquad + C_{\Gamma_{\P}}\bigg[(C_0V_\tau)^{r_S - r_\beta}\mT_{\P}(r_S)\Big(\frac{\log (nm)}{n}\Big)^{r_S}\bigg] \wedge \bigg[(C_0V_\tau)^{r_T-r_\beta}\mT_{\Q}(r_T)\Big(\frac{\log(nm)}{m}\Big)^{r_T}\bigg]\\
        &\qquad + C_{\Gamma_{\Q}}\bigg[(C_0V_\tau)^{r_S-r_\beta}\mT_{\P}(r_S)\Big(\frac{\log(nm)}{n}\Big)^{r_S}\bigg]\wedge \bigg[(C_0V_\tau)^{r_T - r_\beta}\mT_{\Q}(r_T)\Big(\frac{\log(nm)}{m}\Big)^{r_T}\bigg]\\
        &\qquad + C_{\Theta}\Big[(C_0V_\tau)^{r_S}\mT_{\P}(r_S)\Big(\frac{\log(nm)}{n}\Big)^{r_S}\Big]\wedge \Big[(C_0V_\tau)^{r_T}\mT_{\Q}(r_T)\Big(\frac{\log(nm)}{m}\Big)^{r_T}\Big]\\
        &\leq \ol C_0\Big[\mT_{\P}(r_S)\Big(\frac{\log(nm)}{n}\Big)^{r_S}\Big]\wedge \Big[\mT_{\Q}(r_T)\Big(\frac{\log(nm)}{m}\Big)^{r_T}\Big],
    \end{align*}
    for $\ol C_0$ being the sum of all the constants above.\\
    \noindent In the \textbf{accelerated regime}, collecting the terms leads to
    \begin{align*}
        \big\|\wh f - f_*\big\|_{L^2(\Q\mathstrut_{\!\sX})}^2 &\leq C_{\Delta}\log(nm)^{r_\beta}\k_{\P}^{-\gamma a/r_{\beta}}\k_{\Q}^{-s(1 - a)/r_\beta}n^{-\gamma a}m^{-s(1- a)}\mT_{\P}(\gamma)^{a} \mT_{\Q}(s)^{1 - a}\\
        &\quad + C_{\Gamma_{\P}}\log(nm)^{r_\beta}(\ell_{\P}C_{p})^{\gamma a - r_\beta}(\ell_{\Q}C_q)^{s(1 - a)}n^{-\gamma a}m^{-s(1 - a)}\mT_{\P}(\gamma)^{a}\mT_{\Q}(s)^{1-a}\\
        &\quad + C_{\Gamma_{\Q}}\log(nm)^{r_\beta}(C_q\ell_{\Q})^{s(1 - a) - r_\beta}(C_p\ell_{\P})^{\gamma a}n^{-\gamma a}m^{-s(1 - a)}\mT_{\P}(\gamma)^{a}\mT_{\Q}(s)^{1 - a}\\
        &\quad + C_{\Theta}(C_p\ell_{\P})^{\gamma a}(C_q\ell_{\Q})^{s(1 - a)}n^{-\gamma a}m^{-s(1 - a)}\mT_{\P}(\gamma)^{a}\mT_{\Q}(s)^{1-a}.
    \end{align*}
    Picking $\k_{\P} = \k_{\Q} = 1$ and recalling that $\ell_{\P} = \ell_{\Q} = \lceil V_\tau \log(nm)\rceil$ leads to $C_p = C_q$ and
    \begin{align*}
        \big\|\wh f - f_*\big\|_{L^2(\Q\mathstrut_{\!\sX})}^2 &\leq C_{\Delta}\Big(\frac{\log(nm)}{n}\Big)^{\gamma a}\Big(\frac{\log(nm)}{m}\Big)^{s(1- a)}\mT_{\P}(\gamma)^{a} \mT_{\Q}(s)^{1 - a}\\
        &\quad + C_{\Gamma_{\P}}\Big(\frac{\log(nm)}{n}\Big)^{\gamma a}\Big(\frac{\log(nm)}{m}\Big)^{s(1- a)}\mT_{\P}(\gamma)^{a}\mT_{\Q}(s)^{1-a}\\
        &\quad + C_{\Gamma_{\Q}}\Big(\frac{\log(nm)}{n}\Big)^{\gamma a}\Big(\frac{\log(nm)}{m}\Big)^{s(1- a)}\mT_{\P}(\gamma)^{a}\mT_{\Q}(s)^{1 - a}\\
        &\quad + C_{\Theta}(C_0V_\tau)^{r_\beta}\Big(\frac{\log(nm)}{n}\Big)^{\gamma a}\Big(\frac{\log(nm)}{m}\Big)^{s(1- a)}\mT_{\P}(\gamma)^{a}\mT_{\Q}(s)^{1-a}\\
        &= \tilde C_0\Big(\frac{\log(nm)}{n}\Big)^{\gamma a}\Big(\frac{\log(nm)}{m}\Big)^{s(1- a)}\mT_{\P}(\gamma)^{a}\mT_{\Q}(s)^{1-a},
    \end{align*}
    which yields the claimed upper bound after replacing $a = \tfrac{r_\beta - s}{\gamma - s}.$ Recall that we were working in the event $E_n^{\P_{\sX}}(\ell_{\P}, \wh f_{\P}) \cap E_m^{\Q\mathstrut_{\!\sX}}(\ell_{\Q}, \wh f_{\Q}).$ It therefore remains to use Propositions \ref{prop.neighbours.distance.bound.zeta} and \ref{prop.neighbours.distance.lower.bound.zeta}, Corollary \ref{cor.variance.bound.nice.version}, and Remark \ref{rem.bound.variable.k} together with the union bound to obtain
    \begin{align*}
        \PP\bigg\{E_n^{\P_{\sX}}(\ell_{\P}, \wh f_{\P}) \cap E_m^{\Q\mathstrut_{\!\sX}}(\ell_{\Q}, \wh f_{\Q})\bigg\} \geq 1 - 2n^{-\tau} - n^{1-\tau} - 2m^{-\tau} - m^{1-\tau} \geq 1 - 3n^{1-\tau} - 3m^{1-\tau}.
    \end{align*}
\end{proof}

\subsection{Proof of Theorem \ref{th.lower.bound}}
\label{app.lower.bound}

\begin{theorem}
    \label{th.lower.bound.app}
    Work under \eqref{eq.covariate.shift} and within the regression model \eqref{eq.model} with Gaussian noise $\eps \sim \mN(0, \sigma_e^2)$ with $\sigma_e^2 > 0.$ Let $L, D > 0, \ \theta \geq \Leb(\B(0, 1)), \beta \in (0, 1]$ and $(\gamma, s) \in (0, \infty)\times (0, 1).$ Then, there exists a constant $C_0 > 0$ independent of $n,m$ such that for all $n\vee m >1$, 
    \begin{align*}
        \inf_{\wh f}\sup_{\substack{f_* \in \mH(L, \beta)\\ \P_{\sX}, \Q\mathstrut_{\!\sX} \in \ul\mP(D, \theta, \gamma, s)}} \E\Big[\|\wh f - f_*\|_{L^2(\Q\mathstrut_{\!\sX})}^2\Big] \geq  \begin{dcases}
            C_0m^{- s\tfrac{\gamma - r_\beta}{\gamma - s}}n^{-\gamma\tfrac{r_\beta - s}{\gamma - s}} & \vcenter{\hbox{\shortstack[l]{if $(\gamma-r_\beta)(s-r_\beta)<0,$\\and $m\in[n,n^{\gamma/s}]$}}}\\
            C_0\Big[n^{-(\gamma\wedge r_\beta)} \wedge m^{-(s \wedge r_\beta)}\Big] & \text{ else. }
        \end{dcases}
    \end{align*}
    where the $\inf$ is taken over all estimators obtained from $n$ samples from $\P_{\sX, \sY}$ and $m$ samples from $\Q\mathstrut_{\!\X, \sY}.$
\end{theorem}

\begin{proof}[Proof of Theorem \ref{th.lower.bound}]
    Let $\alpha_{\P}, \alpha_{\Q} >0$ and define
    \begin{align*}
        \gamma = \frac{\alpha_{\Q}}{\alpha_{\P} + 1}, \ \text{ and } \ s = \frac{\alpha_{\Q}}{\alpha_{\Q} + 1}.
    \end{align*}
    Consider the distributions $\P_{\sX}$ and $\Q\mathstrut_{\!\sX}$ with respective densities 
    \begin{align*}
        p(x) &:= \Big(\frac{\alpha_{\P}}{\sigma}\Big)^d\prod_{i=1}^d\Big(1 + \frac{x_i}{\sigma}\Big)^{-(\alpha_{\P} + 1)}\mathds{1}(x_i \geq 0)\\
        q(x) &:= \Big(\frac{\alpha_{\Q}}{\sigma}\Big)^d\prod_{i=1}^d\Big(1 + \frac{x_i}{\sigma}\Big)^{-(\alpha_{\Q} + 1)}\mathds{1}(x_i \geq 0).
    \end{align*}
    As seen in Lemma \ref{lem.pareto}, it holds that $\gamma = \gamma^*(\P_{\sX}, \Q\mathstrut_{\!\sX})$ and $s = \gamma^*(\Q\mathstrut_{\!\sX}, \Q\mathstrut_{\!\sX}).$ Additionally, Lemma \ref{lem.pareto} ensures that taking $\sigma$ large enough leads to $(\P_{\sX}, \Q\mathstrut_{\!\sX}) \in \mP(D, \theta).$
    Let $a \geq 0$ and $z_1, \dots, z_M$ be the centres of a (maximal) $2h$-packing of $[a, 2a]^d$ in infinity norm. It holds that
    \begin{align*}
        \Big(\frac{a}{2h}\Big)^d \leq M \leq \Big(\frac{3a}{2h}\Big)^d.
    \end{align*}
    Consider $\Phi \in \mH(L, \beta)$ to be a Kernel supported on $[-1, 1]^d$ such that $\|\Phi\|_{\infty} <\infty.$ Let $\phi_j(x) := Lh^\beta \Phi(|x - z_j|/h).$ Varshamov-Gilbert's bound \cite[Lemma 2.9]{MR2724359} ensures that, provided $M\geq 8,$ there exists a subset $B$ of $\{0, 1\}^M$ such that 
    \begin{enumerate}
        \item [$(i)$] $\card(B) \geq 2^{M/8}$ and
        \item [$(ii)$] $|b - b'|_1 \geq M/8$ for all $b, b' \in B$ and $b \neq b'.$
    \end{enumerate}
    For $b \in B$ define $f_b := \sum_{i=1}^M b_i \phi_i$ and let $\P_b = \P_{f_b},$ the joint distribution of $n$ i.i.d.\ pairs $(X_i, Y_i)$ where $X_i \sim \P_{\sX}$ and $Y_i$ is distributed according to Model \eqref{eq.model}. Next,
    \begin{align}
    \label{eq.L2Q.bound}
    \begin{split}
        \|f_{b} - f_{b'}\|^2_{L^2(\Q\mathstrut_{\!\sX})} &= \int \sum_{i=1}^M (b_i - b_i')^2\phi_i(x)^2\, d\Q(x)\\
        &= \sum_{i=1}^M|b_i - b_i'|\int_{\B(z_i, h)}\phi_i(x)^2\, d\Q(x)\\
        &\geq \Big(\frac{\alpha_{\Q}}{\sigma}\Big)^d\frac{L^2h^{2\beta}}{(1 + 2a/\sigma)^{d(\alpha_{\Q} + 1)}}\sum_{i=1}^M|b_i - b_i'|\int_{\B(z_i, h)}\Phi\bigg(\frac{x - z_i}{h}\bigg)^2\, dx\\
        &= \Big(\frac{\alpha_{\Q}}{\sigma}\Big)^d\frac{L^2h^{2\beta + d}\|\Phi\|_2^2}{(1 + 2a/\sigma)^{d(\alpha_{\Q} + 1)}}\sum_{i=1}^M|b_i - b_i'|\\
        &\geq \Big(\frac{\alpha_{\Q}}{\sigma}\Big)^d\frac{L^2h^{2\beta+d}M\|\Phi\|_{2}^2}{8(1 + 2a/\sigma)^{d(\alpha_{\Q} + 1)}} =: \frac{C_{q, 2}h^{2\beta + d}M}{(1 + 2a/\sigma)^{d(\alpha_{\Q} + 1)}}.
        \end{split}
    \end{align}
    Similarly, we can derive the following upper bounds.
    \begin{align*}
        \|f_{b} - f_{b'}\|^2_{L^2(\Q\mathstrut_{\!\sX})} &= \int \sum_{i=1}^M (b_i - b_i')^2\phi_i(x)^2\, d\Q(x)\\
        &\leq \Big(\frac{\alpha_{\Q}}{\sigma}\Big)^d\frac{L^2h^{2\beta}}{(1 + a/\sigma)^{d(\alpha_{\Q} + 1)}}\sum_{i=1}^M|b_i - b_i'|\int_{\B(z_i, h)}\Phi\bigg(\frac{x - z_i}{h}\bigg)^2\, dx\\
        &\leq \Big(\frac{\alpha_{\Q}}{\sigma}\Big)^d\frac{L^2h^{2\beta+d}M\|\Phi\|_{2}^2}{8(1 + a/\sigma)^{d(\alpha_{\Q} + 1)}} =: \frac{C_{q, 1}h^{2\beta + d}M}{(1 + a/\sigma)^{d(\alpha_{\Q} + 1)}},
    \end{align*}
    and 
    \begin{align*}
        \|f_{b} - f_{b'}\|^2_{L^2(\P_{\sX})} \leq \Big(\frac{\alpha_{\P}}{\sigma}\Big)^d\frac{L^2h^{2\beta+d}M\|\Phi\|_{2}^2}{8(1 + a/\sigma)^{d(\alpha_{\P} + 1)}} =: \frac{C_{p, 1}h^{2\beta + d}M}{(1 + a/\sigma)^{d(\alpha_{\P} + 1)}}.
    \end{align*}
    Fano's inequality \cite[Proposition 15.2]{wainwrighthighdimstat} combined with an elementary upper-bound on the mutual information leads to
    \begin{align}
        \label{eq.fano}
        \begin{split}
        \inf_{\wh f} \sup_{f \in \mH(L, \beta)}\E_{f}&\big[\|\wh f - f\|_{L^2(\Q\mathstrut_{\!\sX})}^2\big]\\
        &\geq \inf_{b\neq b'}\|f_b - f_{b'}\|_{L^2(\Q\mathstrut_{\!\sX})}\bigg(1 - \frac{\log(2) + \max_{b, b'}\KL(\P_b, \P_{b'})}{\log(\card(B))}\bigg).
        \end{split}
    \end{align}
    Denote by $r_{\beta} := 2\beta/(2\beta + d)$ for short. We first assume that $M > 8,$ which we will check in all the latter cases. Then we have
    \begin{align}
    \label{eq.bound.log2.fano.heavy.tail}
    \begin{split}
        \frac{\log(2)}{\log(\card(B))} \leq \frac{8\log(2)}{M} \leq \frac14,
    \end{split}
    \end{align}
    as long as $M \geq 32\log(2),$ which we verify later. Before splitting into cases, we summarise the constraints derived by Fano's method in our setting. The packing in $\KL$ requires
    \begin{align}
    \label{eq.KL.bound.heavy.tail}
    \begin{split}
        \frac{\max_{b, b'}\KL(\P_b, \P_{b'})}{\log(\card(B))} &\leq \frac{8\max_{b, b'}\KL(\P_b, \P_{b'})}{M}\\
        &\leq \frac{8}{M\sigma_e^2}\max_{b, b'}\Big[n\|f_b - f_{b'}\|_{L^2(\P_{\sX})}^2 + m\|f_b - f_{b'}\|_{L^2(\Q\mathstrut_{\!\sX})}^2\Big]\\
        &\leq \frac{16}{M\sigma_e^2} \bigg(\frac{C_{p, 1}nMh^{2\beta+d}}{(1 + a/\sigma)^{d(\alpha_{\P} + 1)}}\vee\frac{C_{q, 1}mMh^{2\beta+d}}{(1 + a/\sigma)^{d(\alpha_{\Q} + 1)}}\bigg)\\
        &= \frac{16}{M\sigma_e^2}h^{2\beta + d}\bigg(\frac{C_{p, 1}n}{(1 + a/\sigma)^{d(\alpha_{\P} + 1)}}\vee\frac{C_{q, 1}m}{(1 + a/\sigma)^{d(\alpha_{\Q} + 1)}}\bigg)\\
        &\leq \frac{16}{\sigma_e^2}h^{2\beta + d}\bigg(\frac{C_{p, 1}n}{(a/\sigma)^{d(\alpha_{\P} + 1)}}\vee\frac{C_{q, 1}m}{(a/\sigma)^{d(\alpha_{\Q} + 1)}}\bigg)\\
        &\leq h^{2\beta + d}\tilde C\bigg(\frac{n}{(a/\sigma)^{d(\alpha_{\P} + 1)}}\vee\frac{m}{(a/\sigma)^{d(\alpha_{\Q} + 1)}}\bigg),
    \end{split}
    \end{align}
    where $\tilde C := 16\big(C_{p, 1}\vee C_{q, 1}\big)/\sigma_e^2,$ and the latter upper bound is less than $1/4$ if $\mu\geq 1$ and $a$ is chosen as
    \begin{align}
        \label{eq.fano.support.diameter}
        a = 4\sigma\mu\bigg[\Big(\tilde C^{\tfrac{1}{d(\alpha_{\P} + 1)}}n^{\tfrac{1}{d(\alpha_{\P} + 1)}}h^{\tfrac{2\beta + d}{d(\alpha_{\P} + 1)}}\Big)\vee \Big(\tilde C^{\tfrac{1}{d(\alpha_{\Q} + 1)}}m^{\tfrac{1}{d(\alpha_{\Q} + 1)}}h^{\tfrac{2\beta + d}{d(\alpha_{\Q} + 1)}}\Big)\bigg].
    \end{align}
    With this choice of $a$ we obtain, after applying Fano's method, and assuming that $a \geq \sigma,$ which will be checked later,
    \begin{align}
        \label{eq.fano1}
        \begin{split}
        \inf_{\wh f} \sup_{f \in \mH(L, \beta)}\E_{f}\big[\|\wh f - f\|_{L^2(\Q\mathstrut_{\!\sX})}^2\big] &\geq \frac 12\inf_{b\neq b'}\|f_b - f_{b'}\|_{L^2(\Q\mathstrut_{\!\sX})}\\
        &\geq \frac{C_{q, 2}}{2}\frac{h^{2\beta + d}M}{(1 + 2a/\sigma)^{d(\alpha_{\Q} + 1)}}\\
        &\geq \frac{C_{q, 2}}{2}\frac{h^{2\beta}a^d}{(4a/\sigma)^{d(\alpha_{\Q} + 1)}}\\
        &=: C_1 \frac{h^{2\beta}}{a^{d\alpha_{\Q}}}\\
        &= \frac{C_1}{(4\sigma\mu)^{d\alpha_{\Q}}}\bigg(\frac{h^{2\beta - (2\beta + d)\gamma}}{\tilde C^{\gamma}n^{\gamma}} \wedge \frac{h^{2\beta - (2\beta + d)s}}{\tilde C^s m^s}\bigg)\\
        &= \frac{C_1}{(4\sigma\mu)^{d\alpha_{\Q}}}\bigg(\frac{h^{(2\beta + d)(r_\beta - \gamma)}}{\tilde C^{\gamma}n^{\gamma}} \wedge \frac{h^{(2\beta + d)(r_\beta - s)}}{\tilde C^s m^s}\bigg).
        \end{split}
    \end{align}
    We now consider different cases depending on $(\gamma, s).$ \\
    \noindent\textbf{Case 1: $\gamma, s \leq r_\beta$.} Then, we can choose $h = 1,$ which leads to
    \begin{align*}
        a = 4\sigma\mu\bigg[\Big(\tilde C^{\tfrac{1}{d(\alpha_{\P} + 1)}}n^{\tfrac{1}{d(\alpha_{\P} + 1)}}\Big)\vee \Big(\tilde C^{\tfrac{1}{d(\alpha_{\Q} + 1)}}m^{\tfrac{1}{d(\alpha_{\Q} + 1)}}\Big)\bigg].
    \end{align*}
    We notice that $a \to \infty$ as either $n\to \infty$ or $m\to \infty.$ This, and picking $\mu$ large enough, leads to that $M \geq 8\vee 32\log(2)$ and $a \geq \sigma$ for all $n \vee m \geq 1.$ Hence, Fano's inequality \eqref{eq.fano1} yields
    \begin{align*}
        \inf_{\wh f} \sup_{f \in \mH(L, \beta)}\E_{f}\big[\|\wh f - f\|_{L^2(\Q\mathstrut_{\!\sX})}^2\big] &\geq \frac{C_1}{(4\sigma\mu)^{d\alpha_{\Q}}}\bigg(\frac{h^{(2\beta + d)(r_\beta - \gamma)}}{\tilde C^{\gamma}n^{\gamma}} \wedge \frac{h^{(2\beta + d)(r_\beta - s)}}{\tilde C^s m^s}\bigg)\geq \frac{C}{n^{\gamma}\vee m^s}.
    \end{align*}
    This finishes the proof in the case $\gamma, s \leq r_\beta.$\\
    \noindent\textbf{Case 2: $\gamma, s \geq r_\beta.$} In this case, we take $a$ defined as before and pick 
    \begin{align*}
        h = (n\vee m)^{-1/(2\beta + d)}.
    \end{align*}
    This leads to 
    \begin{align*}
        a &= 4\sigma\mu\Big(\frac{\tilde C^{1/(d(\alpha_{\P} + 1))}n^{1/(d(\alpha_{\P} + 1))}}{(n\vee m)^{1/(d(\alpha_{\P} + 1))}}\vee \frac{\tilde C^{1/(d(\alpha_{\Q} + 1))}m^{1/(d(\alpha_{\Q} + 1))}}{(n\vee m)^{1/(d(\alpha_{\Q} + 1))}}\Big) =: \mu C,
    \end{align*}
    which is a constant, while $h$ is a null sequence in $n\vee m.$ Hence, by picking $\mu$ large enough, we have $M \asymp (a/h)^d \geq 8\vee (32\log(2))$ and  $a \geq \sigma$ for all $n\vee m \geq 1.$ Consequently, the lower bound \eqref{eq.fano1} becomes
    \begin{align*}
        \inf_{\wh f} \sup_{f \in \mH(L, \beta)}\E_{f}\big[\|\wh f - f\|_{L^2(\Q\mathstrut_{\!\sX})}^2\big] &\geq C_1 \frac{h^{2\beta}}{a^{d\alpha_{\Q}}}\\
        &= \frac{C_1}{C^{d\alpha_{\Q}}} (n \vee m)^{2\beta/(2\beta + d)},
    \end{align*}
    as claimed. This shows the lower bound for all $(\gamma, s) \in (0, \infty)\times (0, 1)$ in positive and null configurations. Next, we show the lower bounds for negative configurations.\\
    \noindent\textbf{Case 3: $s < r_\beta < \gamma.$} This case is further subdivided into three sub-cases.\\
    
    \noindent\textbf{Case 3.1: $m^s \leq n^s.$} Taking 
    \begin{align*}
        h = n^{-1/(2\beta + d)}
    \end{align*}
    leads to
    \begin{align*}
        a &= 4\sigma\mu \Big(\tilde C^{1/(d(\alpha_{\P} + 1))} \vee \tilde C^{1/(d(\alpha_{\Q} + 1))}m^{s/(d\alpha_{\Q})}n^{-s/(d\alpha_{\Q})}\Big) \in \big[4\sigma\mu \tilde C^{1/(d(\alpha_{\P} + 1))}, 4\sigma\mu \tilde C^{1/(d(\alpha_{\Q} + 1))}\big],
    \end{align*}
    which is of constant order, while $h$ goes to $0$ as $n$ (or $m,$ since $m^s \leq n^s$) grows. Hence, picking $\mu$ large enough leads to $M \geq 8\vee(32\log(2))$ and $a \geq \sigma$ for all $0 \leq m \leq n,\ m\vee n \geq 1$ and the lower bound becomes
    \begin{align*}
        \inf_{\wh f} \sup_{f \in \mH(L, \beta)}\E_{f}\big[\|\wh f - f\|_{L^2(\Q\mathstrut_{\!\sX})}^2\big] &\geq C_1 \frac{h^{2\beta}}{a^{d\alpha_{\Q}}}\\
        &\geq C n^{-2\beta/(2\beta + d)} = C (n^{r_\beta} \vee m^s)^{-1},
    \end{align*}
    as claimed.\\
    
    \noindent\textbf{Case 3.2: $m^s \geq n^\gamma.$} In this case, we take $h = 1$ and we obtain
    \begin{align*}
        a &= 4\sigma\mu\bigg[\Big(\tilde C^{\tfrac{1}{d(\alpha_{\P} + 1)}}n^{\tfrac{1}{d\alpha_{\P} + 1}}\Big)\vee \Big(\tilde C^{\tfrac{1}{d(\alpha_{\Q} + 1)}}m^{\tfrac{1}{d(\alpha_{\Q} + 1)}}\Big)\bigg],
    \end{align*}
    which grows with $n$ and/or $m,$ hence, picking $\mu$ large enough ensures that$M \geq 8 \vee 32\log(2)$ and $a \geq \sigma$ for all $0 \leq n^\gamma \leq m^s, \ n \vee m \geq 1$ Consequently, 
    \begin{align*}
        \inf_{\wh f} \sup_{f \in \mH(L, \beta)}\E_{f}\big[\|\wh f - f\|_{L^2(\Q\mathstrut_{\!\sX})}^2\big] &\geq C_1\frac{h^{2\beta}}{a^{d\alpha_{\Q}}}\\
        &\geq C \big(n^{\gamma}\vee m^{s})^{-1}.
    \end{align*}
    This case is proven.\\

    \noindent\textbf{Case 3.3: $n^s < m^s < n^\gamma.$} In this case, assuming $M \geq 8\vee 32\log(2)$ and $a > \sigma,$ we notice that the lower bound obtained by Fano's method is essentially of the form
    \begin{align}
        \label{eq.3.3}
        \inf_{\wh f} \sup_{f \in \mH(L, \beta)}\E_{f}\big[\|\wh f - f\|_{L^2(\Q\mathstrut_{\!\sX})}^2\big] &\geq \frac{C_1}{(4\sigma\mu)^{d\alpha_{\Q}}}\bigg(\frac{h^{(2\beta + d)(r_\beta - \gamma)}}{\tilde C^{\gamma}n^{\gamma}} \wedge \frac{h^{(2\beta + d)(r_\beta - s)}}{\tilde C^s m^s}\bigg),
    \end{align}
    and because $r_\beta - \gamma < 0 < r_\beta - s,$ the first term in the minimum above is a decreasing function of $h$ while the second term is an increasing function of $h.$ We find that those two terms are (approximately) balanced by picking 
    \begin{align*}
        h := \Big(\frac{m^s}{n^\gamma}\Big)^{\tfrac{1}{(\gamma - s)(2\beta + d)}},
    \end{align*}
    and the lower bound becomes
    \begin{align*}
        \inf_{\wh f} \sup_{f \in \mH(L, \beta)}\E_{f}\big[\|\wh f - f\|_{L^2(\Q\mathstrut_{\!\sX})}^2\big] &\geq \frac{C_1 C(\gamma, s)}{(4\sigma\mu)^{d\alpha_{\Q}}}m^{-s}\Big(\frac{m^s}{n^\gamma}\Big)^{\tfrac{r_\beta - s}{\gamma - s}}\\
        &= \frac{C_1 C(\gamma, s)}{(4\sigma\mu)^{d\alpha_{\Q}}}m^{-s\tfrac{\gamma - r_\beta}{\gamma - s}}n^{-\gamma\tfrac{r_\beta - s}{\gamma - s}}.
    \end{align*}
    It remains to check that $M \geq 8\vee (32\log(2))$ with this choice. A useful reparametrisation here is to take $m^s = n^{\lambda \gamma + (1 - \lambda)s}$ for $\lambda \in (0, 1).$ With this in mind, we can rewrite
    \begin{align*}
        h = n^{\tfrac{(\gamma - s)(\lambda - 1)}{(\gamma - s)(2\beta + d)}} = n^{\tfrac{\lambda - 1}{2\beta + d}},
    \end{align*}
    and
    \begin{align*}
        M^{1/d} \asymp \frac{a}{h} &\asymp h^{-1} \Big(n^{\gamma/(d\alpha_{\Q})}h^{(2\beta + d)\gamma/(d\alpha_{\Q})} \vee m^{s/(d\alpha_{\Q})}h^{(2\beta + d)s/(d\alpha_{\Q})}\Big)\\
        &= \frac 1h\Big(n^{\gamma}h^{(2\beta + d)\gamma} \vee m^{s}h^{(2\beta + d)s}\Big)^{1/(d\alpha_{\Q})}\\
        &= \frac 1h\big(n^{\gamma\lambda} \vee n^{\gamma\lambda}\big)^{1/(d\alpha_{\Q})} = n^{\tfrac{\gamma\lambda}{d\alpha_{\Q}} + \tfrac{1 - \lambda}{2\beta + d}},
    \end{align*}
    and the exponent of $n$ above is strictly positive. Additionally, $n$ is necessarily greater than one, since otherwise there is no regime $n^s < m^s < n^\gamma.$ This, combined with choosing $\mu$ large enough, leads to $M \geq 8\vee (32\log(2))$ and $a\geq \sigma$ for all $n^s < m^s < n^\gamma.$ As a consequence, we obtain the lower bound
    \begin{align*}
        \inf_{\wh f} \sup_{f \in \mH(L, \beta)}\E_{f}\big[\|\wh f - f\|_{L^2(\Q\mathstrut_{\!\sX})}^2\big] &\geq \frac{C_1 C(\gamma, s)}{(4\sigma\mu)^{d\alpha_{\Q}}}m^{-s\tfrac{\gamma - r_\beta}{\gamma - s}}n^{-\gamma\tfrac{r_\beta - s}{\gamma - s}}.
    \end{align*}
    This finishes case 3.3.\ and Case 3.\\
    
    \noindent\textbf{Case 4: $\gamma < r_\beta < s.$} Again this case is split into three cases.\\
    
    \noindent\textbf{Case 4.1: $n^{\gamma} \leq m^\gamma.$} This case is symmetric to case 3.1. The lower bound is found by taking $h = m^{-1/(2\beta + d)}$ and following similar steps to obtain
    \begin{align*}
        \inf_{\wh f} \sup_{f \in \mH(L, \beta)}\E_{f}\big[\|\wh f - f\|_{L^2(\Q\mathstrut_{\!\sX})}^2\big] &\geq C m^{-2\beta/(2\beta + d)} = C (n^{\gamma} \vee m^{r_\beta})^{-1}.
    \end{align*}
    Hence, the claim holds in Case 4.1.\\
    
    \noindent\textbf{Case 4.2: $n^\gamma \geq m^s.$} This case is symmetric to case 3.2. The lower bound is obtained by taking $h = 1$ and following similar steps to obtain
    \begin{align*}
        \inf_{\wh f} \sup_{f \in \mH(L, \beta)}\E_{f}\big[\|\wh f - f\|_{L^2(\Q\mathstrut_{\!\sX})}^2\big] &\geq C_1\frac{h^{2\beta}}{a^{d\alpha_{\Q}}}\\
        &\geq C \big(n^{\gamma}\vee m^{s})^{-1} = Cn^{-\gamma}.
    \end{align*}
    Hence, the claim holds in Case 4.2.\\
    
    \noindent\textbf{Case 4.3: $m^\gamma < n^\gamma < m^s.$} This case is again symmetric to case 3.3. We detail the steps for completeness. Following the same steps as in case 3.3.\ we can see that this time the lower bound reads
    \begin{align*}
        \inf_{\wh f} \sup_{f \in \mH(L, \beta)}\E_{f}\big[\|\wh f - f\|_{L^2(\Q\mathstrut_{\!\sX})}^2\big] &= \frac{C_1}{(4\sigma\mu)^{d\alpha_{\Q}}}\bigg(\frac{h^{(2\beta + d)(r_\beta - \gamma)}}{\tilde C^{\gamma}n^{\gamma}} \wedge \frac{h^{(2\beta + d)(r_\beta - s)}}{\tilde C^s m^s}\bigg),
    \end{align*}
    and because $r_\beta - s < 0 < r_\beta - \gamma,$ the first term in the minimum above is an increasing function of $h$ while the second term is a decreasing function of $h.$ We find that those two terms are (approximately) balanced by picking 
    \begin{align*}
        h := \Big(\frac{n^\gamma}{m^s}\Big)^{\tfrac{1}{(s - \gamma)(2\beta + d)}},
    \end{align*}
    and the lower bound becomes
    \begin{align*}
        \inf_{\wh f} \sup_{f \in \mH(L, \beta)}\E_{f}\big[\|\wh f - f\|_{L^2(\Q\mathstrut_{\!\sX})}^2\big] &\geq \frac{C_1 C(\gamma, s)}{(4\sigma\mu)^{d\alpha_{\Q}}}n^{-\gamma}\Big(\frac{n^\gamma}{m^s}\Big)^{\tfrac{r_\beta - \gamma}{s - \gamma}}\\
        &= \frac{C_1 C(\gamma, s)}{(4\sigma\mu)^{d\alpha_{\Q}}}m^{-s\tfrac{r_\beta - \gamma}{s - \gamma}}n^{-\gamma\tfrac{s - r_\beta}{s - \gamma}}.
    \end{align*}
    We again can parametrise $n^{\gamma} = m^{\lambda s + (1 - \lambda)\gamma}$ and obtain that
    \begin{align*}
        \frac{a}{h} &= h^{-1} \Big(n^{\gamma/(d\alpha_{\Q})}h^{(2\beta + d)\gamma/(d\alpha_{\Q})} \vee m^{s/(d\alpha_{\Q})}h^{(2\beta + d)s/(d\alpha_{\Q})}\Big)\\
        &= \frac 1h\Big(n^{\gamma}h^{(2\beta + d)\gamma} \vee m^{s}h^{(2\beta + d)s}\Big)^{1/(d\alpha_{\Q})}\\
        &= \frac 1h\big(m^{s\lambda} \vee m^{s\lambda}\big)^{1/(d\alpha_{\Q})}\\
        &= m^{\tfrac{s\lambda}{d\alpha_{\Q}} + \tfrac{(1 - \lambda)}{2\beta + d}},
    \end{align*}
    and the latter is $m$ taken to a strictly positive power, which shows that $M$ diverges. This, and picking $\mu$ large enough ensures that $M \geq 8\vee (32\log(2)),$ and $a \geq \sigma$ for all $m^\gamma < n^\gamma < m^s.$ This proves the lower bound
    \begin{align*}
        \inf_{\wh f} \sup_{f \in \mH(L, \beta)}\E_{f}\big[\|\wh f - f\|_{L^2(\Q\mathstrut_{\!\sX})}^2\big] &\geq \frac{C_1 C(\gamma, s)}{(4\sigma\mu)^{d\alpha_{\Q}}}n^{-\gamma}\Big(\frac{n^\gamma}{m^s}\Big)^{\tfrac{r_\beta - \gamma}{s - \gamma}}\\
        &= \frac{C_1 C(\gamma, s)}{(4\sigma\mu)^{d\alpha_{\Q}}}m^{-s\tfrac{r_\beta - \gamma}{s - \gamma}}n^{-\gamma\tfrac{s - r_\beta}{s - \gamma}}.
    \end{align*}
    This finishes case 4.3, Case 4, and the overall proof.
\end{proof}

\section{Proofs for the examples}
\label{app.examples}
\begin{lemma}
    \label{lem.example.exponential}
    Let $\P_{\sX}  = \mE(\lambda),$ then, $\P_{\sX} \in \mP(\lambda, e^\lambda)$ and $\gamma^*(\P_{\sX}, \P_{\sX}) = 1.$ Additionally, if $\P_{\sX} = \mE(\lambda_{\P}), \ \Q\mathstrut_{\!\sX} = \mE(\lambda_{\Q}),$ and $\lambda = \lambda_{\P} \vee \lambda_{\Q},$ then, $\P_{\sX}$ and $\Q\mathstrut_{\!\sX}$ are both in $\mP(\lambda, \exp(\lambda)).$ Additionally, $\gamma^{*}(\P_{\sX}, \Q\mathstrut_{\!\sX}) = \lambda_{\Q}/\lambda_{\P}.$
\end{lemma}

\begin{proof}
    We first check the assumptions in Definition \ref{def.class.proba}. By definition, the density of $\P_{\sX}$ is upper bounded by $\lambda.$ Next, for any $x > 0$ and $0 < t \leq 1,$
    \begin{align*}
        \P_{\sX}\{\B(x, t)\} = \int_{x - t}^{x + t} \lambda\exp(-\lambda u)\mathds{1}(u \geq 0)\, du &\geq \int_x^{x + t}\lambda\exp(-\lambda (x + t))\, du\\
        &= t\exp(-\lambda t)p(x) \geq e^{-\lambda}tp(x),
    \end{align*}
    This proves that the minimal mass property is satisfied with $\theta = e^\lambda.$ Next, we prove the maximal mass property. Let $t \in (0, 1].$ If $0 \leq x \leq t,$ then $\P_{\sX}\{\B(x, t)\} = 1 - \exp(-\lambda(x + t)),$ and
    \begin{align*}
        \frac{\P_{\sX}\{\B(x, t)\}}{tp(x)} = \frac{1 - \exp(-\lambda(x + t))}{t\lambda\exp(-\lambda x)} &= \frac{\exp(\lambda x) - \exp(-\lambda t)}{t\lambda}\\
        &\leq \frac{\exp(\lambda t) - \exp(-\lambda t)}{t\lambda} = \frac{2\sinh(\lambda t)}{\lambda t} \leq \exp(\lambda t).
    \end{align*}
    The function $g \colon t \mapsto 2\sinh(\lambda t)/(\lambda t)$ continuously increases on $\RR_+^*.$ Moreover, $g$ admits a limit as $t \to 0,$ which is equal to $2.$ We have proven that for $0 < t \leq 1$ and $0\leq x \leq t,$ it holds that $\P_{\sX}\{\B(x, t)\} \leq \theta p(x)t.$ We now consider $t < x,$ in which case we have
    \begin{align*}
        \frac{\P_{\sX}\{\B(x, t)\}}{tp(x)} = \frac{\exp(-\lambda(x - t)) - \exp(-\lambda(x + t))}{\lambda t \exp(-\lambda x)} = \frac{2\sinh(\lambda t)}{\lambda t} \leq \exp(\lambda t),
    \end{align*}
    as previously, for all $0 < t \leq 1.$ We have shown that $\P_{\sX} \in \mP(\lambda, \exp(\lambda)).$ Next, for $s \in [0, 1),$ we have
    \begin{align*}
        \mT(\P_{\sX}, \P_{\sX}) = \int p(x)^{1-s}\, dx = \int \lambda^{1-s} \exp(-\lambda x (1 - s))\, dx  < \infty,
    \end{align*}
    while for all $s \geq 1, \mT(\P_{\sX}, \P_{\sX}) = \infty.$ Hence, $\gamma^*(\P_{\sX}, \P_{\sX}) = 1.$ Finally, we show the second claim. Let $\gamma \in [0, \lambda_{\Q}/\lambda_{\P}).$ Then,
    \begin{align*}
        \int q(x)p(x)^{-\gamma}\, dx &= \frac{\lambda_{\Q}}{\lambda_{\P}^{\gamma}} \int \exp((\gamma\lambda_{\P} - \lambda_{\Q})x)\, dx < \infty,
    \end{align*}
    while $\mT(\P_{\sX}, \Q\mathstrut_{\!\sX}, \gamma) = \infty$ for all $\gamma \geq \lambda_{\Q}/\lambda_{\P}.$ Hence $\gamma^*(\P_{\sX}, \Q\mathstrut_{\!\sX}) = \lambda_{\Q}/\lambda_{\P}.$
\end{proof}

\begin{definition}[Pareto distribution]
    For $\alpha, \sigma > 0,$ we call Pareto distribution $\Par(\alpha, \sigma)$ the distribution on $\RR$ with density given by
    \begin{align*}
    p(x) = \frac{\alpha}{\sigma}(1 + x/\sigma)^{-(\alpha+1)}\mathds{1}(x\geq 0).
\end{align*}
\end{definition}

\begin{lemma}
    \label{lem.pareto}
    Let $\alpha, \sigma > 0$ and $\P_{\sX} = \Par(\alpha, \sigma),$ then $\P_{\sX} \in \mP(\alpha/\sigma, 2(1 + \sigma^{-1})^{\alpha +1})$ and $\gamma^*(\P_{\sX}, \P_{\sX}) = \alpha/(\alpha + 1).$ Additionally, if $\P_{\sX} = \Par(\alpha_{\P}, \sigma_{\P})$ and $\Q\mathstrut_{\!\sX} = \Par(\alpha_{\Q}, \sigma_{\Q})$ then $\P_{\sX}$ and $\Q\mathstrut_{\!\sX}$ are both in $\mP(D, \theta)$ with
    \begin{align*}
        D = \frac{\alpha_{\P}}{\sigma_{\P}}\vee\frac{\alpha_{\Q}}{\sigma_{\Q}}, \ \text{ and } \ \theta = 2\Big(1 + \frac{1}{\sigma_{\P}\wedge\sigma_{\Q}}\Big)^{\alpha_{\P}\wedge\alpha_{\Q} + 1}.
    \end{align*}
    Additionally, $\gamma^*(\P_{\sX}, \Q\mathstrut_{\!\sX}) = \alpha_{\Q}/(\alpha_{\P} + 1).$
\end{lemma}

\begin{proof}
    Let $p$ be the density of $\P_{\sX} = \Par(\alpha, \sigma).$ We have, for all $ 0 < r \leq 1,$
    \begin{align*}
        (rp(x))^{-1}\int_{x-r}^{x + r}p(t)\, dt \leq \frac{2r\sup_{|t - x| \leq r}p(t)}{rp(x)} &= \frac{2p((x - r)\vee 0)}{p(x)}\\
        &= 2\Big(1 + \frac{x}{\sigma}\Big)^{\alpha + 1}\mathds{1}(x \leq r) + 2\Big(\frac{\sigma + x}{\sigma + x - r}\Big)^{\alpha + 1}\mathds{1}(x > r)\\
        &\leq 2(1 + r/\sigma)^{\alpha + 1} \leq 2\Big(1 + \frac{1}{\sigma}\Big)^{\alpha + 1}.
    \end{align*}
    This implies the maximal mass property since
    \begin{align*}
        \P_{\sX}\{\B(x, r)\} = \int_{(x - r)}^{x + r}p(x)\, dx \leq 2rp((x - r)\vee 0) \leq 2\Big(1 + \frac{x}{\sigma}\Big)^{\alpha + 1}p(x) r.
    \end{align*}
    On the other hand, we have, for all $0 < r \leq 1,$
    \begin{align*}
        \frac{p(x + r)}{p(x)} &= \Big(\frac{\sigma + x}{\sigma + x + r}\Big)^{\alpha + 1} \geq \Big(\frac{\sigma}{\sigma + 1}\Big)^{\alpha + 1},
    \end{align*}
    which shows that the minimal mass property holds with $\theta = 2(1 + 1/\sigma)^{\alpha+1}.$ Indeed,
    \begin{align*}
        \P_{\sX}\{\B(x, r)\} = \int_{(x - r)}^{x + r}p(x)\, dx \geq rp(x + r) \geq \frac{1}{2}\Big(\frac{\theta}{\theta + 1}\Big)^{\alpha +1}rp(x).
    \end{align*}
    Obviously, $\|p\|_{\infty} = \alpha/\sigma.$ Therefore $\P_{\sX} \in \mP(\alpha/\sigma, 2(1 + 1/\sigma)^{\alpha+1}).$\\
    Next, we can check that
    \begin{align*}
        \mT(\P_{\sX}, \P_{\sX}, s) = \big(\alpha/\sigma\big)^{1-s}\int \Big(1 + \frac{x}{\sigma}\Big)^{-(\alpha + 1)(1 - s)}\mathds{1}(x \geq 0)\, dx
    \end{align*}
    is finite if and only if $s < \tfrac{\alpha}{\alpha + 1}.$ Finally, if $\P_{\sX} = \Par(\alpha_{\P})$ and $\Q\mathstrut_{\!\sX} \in \Par(\alpha_{\Q}),$ then the inclusion of both into $\mP(D, \theta),$ with $D, \theta$ defined in the claim, is obvious and we can see that
    \begin{align*}
        \mT(\P_{\sX}, \Q\mathstrut_{\!\sX}, \gamma) = \frac{\alpha_{\Q}\sigma_{\P}^{\gamma}}{\alpha_{\P}^{\gamma}\sigma_{\Q}}\int \frac{(1 + \tfrac{x}{\sigma_{\P}})^{\gamma(\alpha_{\P} + 1)}}{(1 + \tfrac{x}{\sigma_{\Q}})^{\alpha_{\Q} + 1}}\mathds{1}(x \geq 0)\, dx
    \end{align*}
    is finite if and only if $\gamma < \alpha_{\Q}/(\alpha_{\P} + 1).$ This finishes the proof.
\end{proof}

\section{Proofs for Section 4.1}
\label{app.4.1}

\begin{lemma}\label{lem.ratio.bound}
    Consider $\P, \Q \in \mM.$ Then, 
    \begin{enumerate}
        \item [$(i)$] $\mT(\P, \Q, 0) = 1.$
        \item [$(ii)$] $\gamma \mapsto \mT(\P, \Q, \gamma)^\gamma$ is non-decreasing on its domain.
        \item [$(iii)$] $\gamma \mapsto \mT(\P, \Q, \gamma)$ is $\log$-convex on its domain.
    \end{enumerate}
\end{lemma}
\begin{proof}
    The first claim follows immediately from
    \begin{align*}
        \mT(\P, \Q, 0) = \int q(x)\, dx = 1.
    \end{align*}
    If $\mT(\P, \Q, \gamma) = \infty$ for all $\gamma > 0,$ then $(ii)$ and $(iii)$ hold trivially. Otherwise, let $X \sim \Q$ and $s > 0$ such that $\mT(\P, \Q, s) < \infty.$ For all $\gamma \in [0, s],$ the function $t \mapsto t^{s/\gamma}$ is convex on $\RR_+.$ By Jensen's inequality, we have 
    \begin{align*}
       \mT(\P, \Q, \gamma) = \big(\E[p(X)^{-\gamma}]^{s/\gamma}\big)^{\gamma/s} \leq \E[p(X)^{-s}]^{\gamma/s} = \mT(\P, \Q, s)^{\gamma/s},
    \end{align*}
    which proves $(ii).$ The last claim $(iii)$ follows from H\"older's inequality. Let $\gamma_1, \gamma_2$ in the domain of $\mT(\P, \Q, \cdot)$ and $\lambda \in [0, 1].$ Then,
    \begin{align*}
        \mT(\P, \Q, \lambda \gamma_1 + (1 - \lambda) \gamma_2) &= \int \bigg(\frac{q(x)}{p(x)^{\gamma_1}}\bigg)^{\lambda}\bigg(\frac{q(x)}{p(x)^{\gamma_2}}\bigg)^{1-\lambda}\, dx \leq \big(\mT(\P, \Q, \gamma_1\big)^{\lambda}\big(\mT(\P, \Q, \gamma_2)\big)^{1 - \lambda},
    \end{align*}
    which yields the claim.
\end{proof}

\begin{lemma}[Lemma 15 in \cite{zamolodtchikov2024transfer}]\label{lem.pseudo.moment}
Let $\Q \in \mM$ with density $q$ and $\rho, \eps > 0.$ If $X \sim \Q$ has a finite generalized moment of order $\rho + \eps,$ then $\gamma^*(\Q, \Q) \geq \rho/(\rho+d).$
\end{lemma}

\begin{proof}
Let $\tau = d/(\rho + d).$ By H\"{o}lder's inequality, we have
\begin{align*}
& \int  q(x)^{\tau} \, dx 
   = \int (1 + \|x\|^{\rho + \eps})^{\tau} q(x)^{\tau} (1 + \|x\|^{\rho + \eps})^{-\tau} \, dx\\
&\leq \bigg( \int \bigl( (1 + \|x\|^{\rho + \eps})^{\tau} q(x)^{\tau} \bigr)^{1/\tau} \, dx \bigg)^{\tau} \cdot
           \biggl( \int (1 + \|x\|^{\rho + \eps})^{- \tau / (1 - \tau)} \, dx \biggr)^{1-\tau}
\\
& \leq (1 + M)^{\tau} \biggl( \int (1 + \|x\|^{\rho + \eps})^{- \tau / (1 - \tau)} \, dx \biggr)^{1-\tau},
\end{align*}
where $M := \int\|x\|^{\rho + \eps}\, \Q(dx).$ Since $\tau/(1 - \tau) = d/\rho,$ we have $\tau(\rho + \eps) / (1 - \tau) > d$ and consequently
\begin{align*}
\int (1 + \|x\|^{\rho + \eps})^{- \tau / (1 - \tau)} \, dx < \infty,
\end{align*}
which proves the claim.
\end{proof}

\begin{proof}[\textit{\textbf{Proof of Lemma \ref{lem.renyi.to.gamma}}}]
    Let $r' = (r - 1)/r$ be the conjugate of $r,$ and $0 \leq s < \gamma^*(\P, \P).$ For $\gamma \geq 0,$ it follows from H\"older's inequality that
    \begin{align*}
        \int \frac{q(x)}{p(x)^{\gamma}}\, dx &= \int \frac{q(x)}{p(x)}p(x)^{1 - \gamma}\,dx\\
        &\leq \bigg(\int \Big(\frac{q(x)}{p(x)}\Big)^{r}\,d\P(x)\bigg)^{1/r}\bigg(\int p(x)^{-\gamma r'}\, d\P(x)\bigg)^{1/r'}\\
        &= \exp\Big(\frac{r-1}{r}D_r(\Q \Vert \P)\Big)\bigg(\int p(x)^{1-\frac{\gamma r}{r-1}}\, dx\bigg)^{(r-1)/r},
    \end{align*}
    and the latter is finite for any $\gamma < \gamma^*(\P, \P)(r-1)/r,$ which shows that $\gamma^*(\P, \Q) \geq \gamma^*(\P, \P)(r-1)/r.$
\end{proof}
\end{document}